\documentclass[11pt, twoside]{article}
\usepackage{mathrsfs}
\usepackage{amssymb}
\usepackage{amsmath}
\usepackage{amsthm}
\usepackage{amsfonts}
\usepackage{color}
\usepackage{latexsym}

\usepackage{txfonts}

\usepackage{indentfirst}
\usepackage{soul}
\usepackage[normalem]{ulem}
\usepackage{enumitem}

\usepackage[colorlinks=true,
  linkcolor=red,
  citecolor=blue,
  urlcolor=magenta]{hyperref}

\usepackage{anysize}

\allowdisplaybreaks

\def\rr{\mathbb{R}}
\def\rn{\mathbb{R}^{n}}
\def\zz{\mathbb{Z}}

\def\cc{\mathbb{C}}

\def\nn{\mathbb{N}}
\def\dd{\mathbb{D}}

\def\ca{\mathcal {A}}

\def\cg{\mathcal {G}}
\def\ch{\mathcal {H}}
\def\ci{\mathcal {I}}

\def\cm{\mathcal{ M}}

\def\cp{\mathcal {P}}
\def\cq{\mathcal {Q}}

\def\cx{\mathcal {X}}

\def\vp{{\varepsilon}}

\def\fz{\infty}

\def\diag{{\mathop\mathrm{\,diag\,}}}

\def\loc{{\mathop\mathrm{\,loc\,}}}
\def\essinf{\mathop\mathrm{\,ess\,inf\,}}
\def\esssup{\mathop\mathrm{\,ess\,sup\,}}
\def\diam{{\mathop\mathrm{\,diam\,}}}

\def\ls{\lesssim}

\def\wz{\widetilde}
\def\r{\right}
\def\lf{\left}
\def\f{\frac}
\def\rar{\rightarrow}

\def\dsum{\displaystyle\sum}

\def\gfz{\genfrac{}{}{0pt}{}}

\newtheorem{theorem}{Theorem}[section]
\newtheorem{lemma}[theorem]{Lemma}
\newtheorem{corollary}[theorem]{Corollary}

\theoremstyle{definition}
\newtheorem{remark}[theorem]{Remark}

\newtheorem{definition}[theorem]{Definition}
\renewcommand{\appendix}{\par
   \setcounter{section}{0}%
   \setcounter{subsection}{0}%
   \setcounter{subsubsection}{0}%
   \gdef\thesection{\@Alph\c@section}%
   \gdef\thesubsection{\@Alph\c@section.\@arabic\c@subsection}%
   \gdef\theHsection{\@Alph\c@section.}%
   \gdef\theHsubsection{\@Alph\c@section.\@arabic\c@subsection}%
   \csname appendixmore\endcsname
 }

\numberwithin{equation}{section}

\begin{document}

\arraycolsep=1pt

\title{\bf\Large Matrix-Weighted Poincar\'{e}-Type Inequalities
with  Applications to Logarithmic Haj\l asz--Besov Spaces
on Spaces of Homogeneous Type \footnotetext{\hspace{-0.35cm} 
2020 {\it Mathematics Subject Classification}.
Primary 46E36; Secondary 46E35, 42B35, 26D10, 26A33, 30L99.
\endgraf {\it Key words and phrases.} space of homogeneous type, matrix weight,
$A_p$ dimension, Poincar\'{e}-type inequality, Haj{\l}asz gradient,
wavelet reproducing formula with exponential decay, logarithmic matrix-weighted Besov space.
\endgraf This project is partially supported by the National
Natural Science Foundation of China (Grant Nos. 12371093 and 12431006), 
the Beijing Natural Science Foundation
(Grant No. 1262011),
and the Fundamental Research Funds
for the Central Universities (Grant No. 2253200028).}}
\author{Ziwei Li, Dachun Yang
and Wen Yuan\,\footnote{Corresponding author,
E-mail: \texttt{wenyuan@bnu.edu.cn}/{\color{red} \today}/Final version.}}
\date{}
\maketitle

\vspace{-0.8cm}

\begin{center}
\begin{minipage}{13cm}
{\small{\textbf{Abstract}}\quad Let ${\mathcal {X}}$
be a space of homogeneous type. In this article, based on the reducing
operators of matrix $A_p$-weights, the authors
introduce the vector-valued Haj\l asz gradient sequences  and
establish some related matrix-weighted Poincar\'{e}-type inequalities on
${\mathcal {X}}$. As an application, the authors introduce the matrix-weighted logarithmic
Besov spaces on ${\mathcal {X}}$ and establish their
pointwise characterization via  Haj\l asz gradient sequences.
The novelty of this article lies in that, by means of both the
$A_p$ dimension and its properties of matrix $A_p$-weights
and the wavelet reproducing formula with exponential decay
of P. Auscher and T. Hyt\"{o}nen, all the main results get rid of the dependence
on the reverse doubling conditions of both weights and ${\mathcal {X}}$ under consideration
and these results are also completely new even for unweighted logarithmic Besov spaces
on ${\mathcal {X}}$.}
\end{minipage}
\end{center}

\vspace{0.2cm}

\tableofcontents

%\vspace{0.2cm}

\section{Introduction}

As a wide framework containing many classical function spaces,
Besov--Triebel--Lizorkin spaces have been intensively studied
and widely applied in analysis; see, for instance, \cite{t83,T01,Sa18,Sa20} for some introduction of these spaces. In recent decades, to meet the needs of studies about those problems
from different areas of analysis such
as harmonic analysis and partial differential equations,
many variants or generalizations of classical Besov--Triebel--Lizorkin spaces are further introduced and  developed. For instance, to investigate the
embedding and the interpolation properties of function spaces in the critical case,
Besov and Triebel--Lizorkin spaces with generalized smoothness were frequently used,
whose smoothness is described via a nonnegative weight function; see, for instance,  \cite{G76,G80,M84,KL87,CF88,M01,T01,FL06,M07,HM04,HM08,LYY22}.
In particular, as a special case of Besov spaces with generalized smoothness,
the logarithmic Besov spaces, that is,  Besov spaces with a classical smoothness
and an extra logarithmic smoothness, have received an increasing interest in recent years.
For some representative works on these logarithmic Besov spaces,
we refer the reader to \cite{CD14,CD15,CD15b,CD16,CDT16,D17,DHT20,DT}.
The other important and useful generalization of classical Besov and Triebel--Lizorkin spaces
is their weighted variants, including the matrix-weighted ones.
The matrix $A_p$-weights were introduced and considered in \cite{TV97,NT97,V97}
in order to study some problems from the multivariate random stationary process,
the invertibility of Toeplitz operators,
and  the vector-valued Hunt--Muckenhoupt--Wheeden theorem (see \cite{HMW73} for the classical Hunt--Muckenhoupt--Wheeden theorem).
From then on, there exists an increasing interest in developing function spaces with matrix $A_p$-weights and studying the boundedness of operators on them.
In particular, Frazier and Roudenko \cite{R03,R04,FR04,FR08,FR21} developed and studied the  matrix-weighted Besov spaces
and the matrix-weighted Triebel--Lizorkin spaces  (see also \cite{I21}).
Very recently,  Bu et al. \cite{BHYY,BHYY2,BHYY3} introduced and systematically studied the
matrix-weighted Besov-type and Triebel--Lizorkin-type spaces on $\mathbb{R}^n$,
and the authors of the present article introduced the matrix-weighted
Besov--Triebel--Lizorkin spaces with logarithmic smoothness on $\mathbb{R}^n$ in \cite{LYY23}.

The space of homogeneous type introduced by Coifman and Weiss \cite{CW71} has
proved a natural and general setting for the study of function spaces and the boundedness of  operators; see, for instance, \cite{HMY06,HMY08,MY09,GLY08,GLY09,GLY09-2,YZ08,YZ10,YZ11}.
However, in many works of early stage, some extra geometrical assumptions
on the underlying space under consideration, such as
the reverse doubling property of the equipped measure and
the triangle inequality of the equipped metric, are usually needed.
A breakthrough of this aspect was made by Auscher and Hyt\"{o}nen \cite{AH13,AH15},
who constructed a wavelet system having the exponential decay on spaces of homogeneous type
without any extra assumption on the underlying space. The exponential decay
of this wavelet system has now proved the key to help one to get rid of the
dependence of the reverse doubling property of the equipped measure under consideration.
Based on this wavelet system, Han et al. \cite{HHHLP21} introduced Besov--Triebel--Lizorkin spaces on
spaces of homogeneous type. Also motivated by \cite{AH13,AH15},
He et al.  \cite{HLYY19} introduced approximations of the identity with exponential decay
(for short, exp-ATIs) and established the related Calder\'{o}n
reproducing formulae.
The exp-ATIs have been further used to systematically develop
the real-variable theories of Hardy spaces in \cite{HHLLYY19} and
 Besov--Triebel--Lizorkin spaces in \cite{WHHY21}
(which were proved in \cite{HWYY21} to coincide with those in \cite{HHHLP21} defined via wavelets).
Recently, via exp-ATIs, Bu et al. \cite{BYY23} introduced  the
matrix-weighted Besov spaces on spaces of homogeneous type
and established a series of characterizations of these spaces.
The Besov and Triebel--Lizorkin spaces associated with operators on spaces of homogeneous type
were also well developed in \cite{KP15,GN17,GKKP17,GKKP19,GK20,GK23,BBD20,BD21,BD23}.

In 1996, Haj\l asz \cite{Haj96} introduced the concept of Haj\l asz gradients,
whose fractional variants were later introduced in  \cite{Hu03,Ya03}.
In 2011, Koskela et al. \cite{KYZ11} introduced the  Haj\l asz gradient sequences and,
by establishing some Poincar\'e type inequalities related to Haj\l asz gradient sequences,
they established the pointwise characterization of Besov--Triebel--Lizorkin
spaces via some pointwise inequalities
involving Haj{\l}asz gradient sequences. These pointwise characterizations of
Besov--Triebel--Lizorkin spaces via  Haj\l asz gradient sequences have proved a powerful tool
to develop Besov--Triebel--Lizorkin spaces
on spaces of homogeneous type, especially in the study of the invariance of quasi-conformal
mappings on these spaces; see, for instance, \cite{KYZ11,k19,k20,AYY22,AYY23,AWYY23}.
In these studies,  the Poincar\'{e}-type inequalities
(with respect to the  Haj{\l}asz gradient sequences)
occupy a key position and have found some other
applications in the Sobolev embeddings and the Sobolev extensions.
However, we hardly find much information about the  related investigations on matrix-weighted
Besov or Triebel--Lizorkin spaces on spaces of homogeneous type in the literature.

Let ${\mathcal {X}}$ be a space of homogeneous type. In this article, based on
the aforementioned works, we introduce the Haj\l asz gradient sequences
in the vector-valued version and
establish the related matrix-weighted Poincar\'{e}-type inequalities on $\cx$.
As an application, we introduce the matrix-weighted logarithmic Besov spaces on $\cx$ and
establish their Haj\l asz pointwise characterization via vector-valued
Haj\l asz gradient sequences.

As known to all, compared with the scalar case,
one main difference of  matrix-weighted function spaces is the inseparability
of the matrix weight and the concerned vector-valued function,
which brings an essential obstacle for introducing
Haj\l asz gradient sequences in the matrix-weighted setting.
To overcome this obstacle,
we employ the concept of reducing operators
and find a suitable way to define the related vector-valued Haj\l asz gradient sequence.
Based on this, we establish the Poincar\'{e}-type
inequalities with respect to both the reducing operators
and the vector-valued Haj\l asz gradient sequences
and, furthermore, obtain the related Poincar\'{e}-type inequalities.
In this way, we make full use of the local properties of both
the pointwise inequalities  involving gradient sequences and the Poincar\'{e}-type inequalities
as well as the close connection between the matrix $A_p$-weight and its reducing operators.
In addition, the major novelty of this article is that,
by means of both the $A_p$ dimension and its properties of matrix $A_p$-weights
and the wavelet reproducing formula with exponential decay constructed by Auscher and Hyt\"{o}nen in
\cite{AH13,AH15}, all the main results get rid of the dependence
on the reverse doubling conditions of both weights and ${\mathcal {X}}$ under consideration
and can completely recover the corresponding classical unweighted
scalar-valued results in \cite{KYZ10,KYZ11,AWYY23}.
Moreover, these results are also completely new even
for the case when the logarithmic parameters of
matrix-weighted logarithmic Besov spaces are zero, namely for the
matrix-weighted Besov spaces introduced in \cite{BYY23}.
Finally, we mention that the results of this article might be also true for the related
matrix-weighted logarithmic  Triebel--Lizorkin spaces and, to limit the length of this article,
we will present these in the other forthcoming article.

The remainder of this article is organized as follows.

In Section \ref{sec-pre}, we first make some notational conventions in Subsection \ref{subsec-notion}
and then, in Subsection \ref{subsec-Apweight}, recall some basic concepts and properties of
matrix $A_p$-weights on $\cx$ and also extend some properties of the $A_p$
dimension to $\cx$.

In Section \ref{sec-poincare}, we first introduce Haj{\l}asz gradient sequences related to reducing operators on $\cx$, which prove to be suitable for the matrix-weighted vector-valued  case.
Via establishing the related Poincar\'{e}-type inequalities with respect to reducing operators as well as some estimates of the averaging operators $E_k$ [see \eqref{eq-def-Ek}], we obtain
the matrix-weighted Poincar\'{e}-type inequalities (see Theorem \ref{lem-Poincare}).
As an application, we introduce the homogeneous matrix-weighted logarithmic Haj{\l}asz--Besov spaces
and prove that all its elements are locally integrable.

In Section \ref{sec-chara}, we apply the matrix-weighted Poincar\'{e}-type inequalities
to establish the Haj{\l}asz pointwise characterization of logarithmic matrix-weighted Besov spaces
on $\cx$. Specifically speaking, in Subsection \ref{sec-grand},
we first introduce the logarithmic matrix-weighted Besov space
on $\cx$ via approximations of the identity with exponential decay,
by using the orthonormal wavelet basis from \cite{AH13,AH15}, some technical estimates,
and the $L^p(\cx)$-boundedness of matrix-weighted maximal operators on $\cx$,
we further establish the grand maximal function characterization
of the logarithmic matrix-weighted Besov space.
Finally, in Subsection \ref{sec-haj}, under an extra assumption that the measure
under consideration has a weak lower bound,
we apply the matrix-weighted Poincar\'{e}-type
inequalities and the grand maximal function  characterization of
logarithmic matrix-weighted Besov spaces to
establish their Haj{\l}asz pointwise characterization,
that is, we establish the equivalence between logarithmic matrix-weighted Besov spaces
and logarithmic matrix-weighted Haj\l asz--Besov spaces.

\section{Preliminaries}\label{sec-pre}

This section is divided into two parts. In Subsection \ref{subsec-notion},
we do some notational conventions and, in Subsection \ref{subsec-Apweight},
we state some basic concepts and properties of matrix $A_p$-weight on spaces of homogeneous type.

\subsection{Basic notation and notions}\label{subsec-notion}

\emph{Throughout this article}, we denote by $\rr$ the collection of all real numbers,
$\zz$ the collection of all integers, $\nn$ the collection of all positive integers,
and $\zz_{+}:=\nn\cup\{0\}$.
We use $C$ to denote some positive constant which is independent of the main parameters involved
but might be different from line to line.
We also use   $A\ls B$ to represent $A\leq CB$ and $A\sim B$ to represent $A\ls B\ls A$.
For any $a,b\in\rr$, we use $a\vee b$ to denote $\max\{a,b\}$ and
$a\wedge b$ to denote $\min\{a,b\}$.
For any $p\in(1,\fz)$, we denote by $p'$ its \emph{conjugate index}, that is, $1/p+1/p'=1$.

A tuple $(\cx,d)$ is called a \emph{quasi-metric space} if $\cx$ is a non-empty set
and $d$ is a \emph{quasi-metric}, namely
a nonnegative function defined on $\cx\times\cx$ such that, for any $x,y,z\in\cx$,
\begin{enumerate}
\item[\rm(i)] $d(x,y)$=0 if and only if $x=y$;
\item[\rm(ii)] $d(x,y)=d(y,x)$;
\item[\rm(iii)] there exists a positive constant $A_0\in[1,\fz)$, independent of $x,y,z$, such that
\begin{equation}\label{eq-def-A0}
d(x,y)\le A_0\lf[d(x,z)+d(z,y)\r].
\end{equation}
\end{enumerate}
The ball $B$ in $\cx$, which is centered at $x_0\in\cx$ with the radius $r\in(0,\fz)$, is defined by setting
$$
B:=B(x_0,r):=\lf\{x\in\cx:\ d(x,x_0)\in[0,r)\r\}.
$$
For any $\lambda\in(0,\fz)$, we use $\lambda B$ to denote the ball
with the same center as $B$ but $\lambda$-times radius of $B$.
Furthermore, a triple $(\cx,d,\mu)$ is called a \emph{space of homogeneous type}
if $(\cx,d)$ is a quasi-metric space equipped with a \emph{doubling} measure $\mu$,
that is, $\mu$ is a regular Borel measure on $\cx$ such that all of the balls defined by $d$ have
finite and positive measures and
there exists a positive constant $C_\mu\in[1,\fz)$ such that, for any ball $B\subset\cx$,
$$\mu(2B)\leq C_\mu\mu(B).$$
The doubling property of $\mu$ implies that, for any ball $B\subset\cx$ and any $\lambda\in [1,\infty)$,
\begin{align}\label{eq-doub}
	\mu(\lambda B)\leq C_\mu\lambda^{D}\mu(B),
\end{align}
where $D:=\log_2 C_\mu$ is called the \emph{doubling exponent} of $\mu$.
We always assume that
$C_\mu$ is the smallest positive constant such that \eqref{eq-doub} holds.
Without loss of generality, we may further assume that, for any $x\in\cx$,
$\mu(\{x\})=0$. For any $x,y\in\cx$ with $x\neq y$ and for any $r\in(0,\fz)$,
define
$$V(x,y):=\mu\lf(B(x,d(x,y))\r)\ \text{and}\ V_r(x):=\mu\lf(B(x,r)\r).$$

Let $(\cx,d,\mu)$ be a space of homogeneous type.
We denote by $L^{0}(\cx)$ the collection of all  measurable functions on $\cx$
which are finite almost everywhere,
$L^1_{\loc}(\cx)$ the collection of all  locally integrable functions on $\cx$,
and, for any $p\in(0,\infty]$, denote by $L^p(\cx)$ the classical Lebesgue space.
Recall that a measurable function is said to be locally integrable if,
for any $x_0\in\cx$, there exists $r_0\in(0,\fz)$ such that $f\mathbf{1}_{B(x_0,r_0)}\in L^1(\cx)$.
Here and thereafter, for any subset $E\subset\cx$, we denote by $\mathbf{1}_E$ the \emph{characteristic function} of $E$.
For any measurable set $E\subset\cx$ with $\mu(E)\in(0,\fz)$ and for any $u\in L^{0}(\cx)$, let
\begin{equation*}
u_E:=\fint_{E}u(x)\,d\mu(x):=\f{1}{\mu(E)}\int_{E}u(x)\,d\mu(x).
\end{equation*}
We also use $\cm$ to denote the classical \emph{Hardy--Littlewood maximal operator}, which is defined by setting,
for any $u\in L^0(\cx)$ and $x\in\cx$,
$$
\cm u(x):=\sup_{B\ni x}\fint_B \lf|u(y)\r|\,d\mu(y),
$$
where the supremum is taken over all balls in $\cx$ that contain $x$.

\emph{Throughout this article}, we let $m\in\nn$ and $\ch$ be a Hilbert space of dimension $m$.
We use $\langle\cdot,\cdot\rangle$ to denote the inner product of $\ch$,
$(.,\ldots,.)^T$ to denote the \emph{transpose} of the row vector,
and $\|\cdot\|$ to denote the norm of $\ch$.
In this article, we always assume $m\in\nn$ and $\ch:=\cc^m$.
Then,
for any vector $\vec a:=(a_1,\ldots,a_m)^T\in\cc^m$,
$$
\lf\|\vec a\r\|:=\sqrt{|a_1|^2+\cdots+|a_m|^2}.
$$
For any given space $X$,
we denote by $[X]^m$ the \emph{Cartesian product} of the space $X$, namely
$$
[X]^m:=X\times\cdots\times X:=\lf\{(x_1,\ldots,x_m)^T:\ x_i\in X,\ \forall\, i\in\{1,\ldots,m\}\r\}.
$$
For any given function $g$ and any vector-valued function $\vec{f}:=(f_1,\ldots,f_m)^T$,
we formally define
$
g\ast\vec{f}:=\lf(g\ast f_1,\ldots,g\ast f_m\r)^T,
$
$
g\vec{f}:=\lf(g f_1,\ldots,g f_m\r)^T,
$
$
\int\vec{f}\,d\mu:=(\int f_1\,d\mu,\ldots,\int f_m\,d\mu)^T,
$
and,
for any measurable set $E\subset\cx$ with $\mu(E)\in(0,\fz)$,
\begin{equation*}
\vec{f}_E:=\fint_{E}\vec{f}\,d\mu:=\f{1}{\mu(E)}\int_{E}\vec{f}(x)\,d\mu(x).
\end{equation*}

We frequently use the following elementary inequality that, for any $q\in(0,1]$,
\begin{align}\label{triangle}
\lf(\sum_{i\in\zz}|a_{i}|\r)^{q}\leq\sum_{i\in\zz}|a_{i}|^{q},
\ \ \forall\,\{a_{i}\}_{i\in\zz}\subset\cc.
\end{align}
We also use the following basic inequality that,
for any sequence $\{\vec{\eta}_k\}_{k\in\zz}\subset\cc^m$,
\begin{equation*}
\lf\|\sum_{k\in\zz}\vec{\eta}_k\r\|
\le\sum_{k\in\zz}\lf\|\vec{\eta}_k\r\|.
\end{equation*}

\emph{Throughout this article}, we always assume that
$(\cx,d,\mu)$ is a space of homogeneous type with $A_0$ the same as in \eqref{eq-def-A0}
and both $C_\mu$ and $D$ the same as in \eqref{eq-doub}.
we also let $\delta$, which comes from the construction of the system
of dyadic cubes on $\cx$ in Lemma \ref{lem-dyacube}, be fixed and can be very small
and let both $c_0$ and $C_0$ in Lemma \ref{lem-dyacube} also be fixed.
For any $\beta,\gamma\in(0,1)$ and $s\in(-(\beta\wedge\gamma),\beta\wedge\gamma)$,
we use the symbol
\begin{equation*}
p(s,\beta\wedge\gamma):=\max\lf\{\frac{D}{D+(\beta\wedge\gamma)},\frac{D}{D+(\beta\wedge\gamma)+s}\r\}.
\end{equation*}

\subsection{Matrix $A_p$-Weights on Spaces of Homogeneous Type}\label{subsec-Apweight}

We begin this subsection with recalling some necessary knowledge about matrices.
Let $A$ be an $m\times m$ complex-valued matrix.
The matrix $A$ is said to be \emph{positive definite} if,
for any $\vec{\eta}\in\cc^m\setminus\{\mathbf{0}\}$, $\langle A\vec{\eta},\vec{\eta}\rangle>0$;
$A$ is said to be \emph{nonnegative definite} if,
for any $\vec{\eta}\in\cc^m$, $\langle A\vec{\eta},\vec{\eta}\rangle\ge0$.
In what follows, we \emph{always} denote by $M(\cc^m)$ the collection
of all  $m\times m$ nonnegative definite complex-valued matrices.

Any matrix $A$ can be regarded as a linear bounded operator on $\cc^m$
with its operator norm defined by setting
$$
\lf\|A\r\|:=\sup_{\vec{\eta}\in\cc^m,\,\|\vec{\eta}\|=1}\lf\|A\vec{\eta}\r\|.
$$

For any $m\times m$ complex-valued matrix $A:=(a_{ij})_m$,
we use $A^*$ to denote its conjugate transpose, namely
$
A^*:=(\overline{a_{ji}})_m,
$
where, for any $i,j\in\{1,\ldots,m\}$, $\overline{a_{ji}}$ is the conjugate of $a_{ji}$.
The matrix $A$ is called a \emph{unitary matrix} if $A^*A=I_m$,
where $I_m$ denotes the $m\times m$ identity matrix;
$A$ is called a \emph{Hermitian matrix} if $A^*=A$.
It is well known that (see, for instance, \cite[Theorem 4.1.4]{HJ13})
any $A\in M(\cc^m)$ is a Hermitian matrix.

Let $A$ be an $m\times m$ positive definite complex-valued matrix.
As argued in \cite[Theorems 2.5.6(c) and 4.1.8]{HJ13},
there exists an $m\times m$ complex-valued unitary matrix $U$ and
a diagonal matrix $\Lambda:=\diag(\lambda_1,\ldots,\lambda_m)$ with
$\lambda_i\in(0,\fz),\ i\in\{1,\ldots,m\}$,
such that $A$ has the decomposition
\begin{equation}\label{eq-decom-W}
A=U\Lambda U^{*},
\end{equation}
where $\{\lambda_i\}_{i=1}^m$ are just all the eigenvalues of $A$.
This enables us to define the power of matrices; see \cite[(6.2.1)]{HJ94}.

\begin{definition}\label{def-W^r}
Let  $A$ be an $m\times m$ positive definite complex-valued matrix
and both $U$ and $\Lambda$ be the same as in \eqref{eq-decom-W}. For any given $r\in\rr$,
$A^r$ is defined by setting
$$
A^r:=U\Lambda^r U^{*}:=U\diag\lf(\lambda_1^r,\ldots,\lambda_m^r\r)U^{*}.
$$
\end{definition}

\begin{remark}
As pointed out in \cite[p.\,408]{HJ94}, $A^r$ is independent of
the order of the eigenvalues of $A$ and the choice of $U$,
which implies the rationality of the definition of $A^r$.
\end{remark}

Now, we recall the concept of  matrix weights; see, for instance, \cite{TV97,V97}.

\begin{definition}\label{def-matrix weight}
A matrix-valued function $W:\ \cx\rar M(\cc^m)$ is called a \emph{matrix weight} on $\cx$
if
\begin{enumerate}
\item[\rm(i)] for any $x\in\cx,$ $W(x)$ is nonnegative definite;
\item[\rm(ii)] for almost every $x\in\cx$, $W(x)$ is invertible;
\item[\rm(iii)] every entry of $W$ is a locally integrable function on $\cx$.
\end{enumerate}
\end{definition}

For any given matrix weight $W$ and $p\in(0,\infty]$, we use
$L^p(W)$ to denote the matrix-weighted Lebesgue space equipped with the (quasi-)norm,
for any $\vec{f}\in L^p(W)$,
\begin{align*}
\lf\|\vec{f}\,\r\|_{L^{p}(W)}:=
\begin{cases}
\ \lf[\displaystyle\int_{\cx}\Big\|W^{\frac{1}{p}}(x)\vec{f}\,\Big\|^p\,d\mu(x)\r]^{\frac 1 p}\ \ &\text{when}\ p\in(0,\fz),\\
\ \displaystyle\esssup_{x\in\cx}\lf\|\vec{f}(x)\r\|\ \ &\text{when}\ p=\fz.
\end{cases}
\end{align*}
To simplify the presentation,  for any $s\in\rr$ and any pair $b:=(b_+,b_-)$ of real numbers,
we use the symbol $L_{s,b}$ to denote the logarithmic smoothness weight function:
for any $t\in(0,\fz)$,
\begin{equation}\label{eq-def-Log}
L_{s,b}(t):=t^s\lf[1+\log_\delta(1\wedge t)\r]^{-b_+}\lf[1+\log_\delta(1\vee t)\r]^{-b_-}.
\end{equation}
Then we let $\ell_{s,b}^{q}(W,L^{p})$ be the space of all sequences $\{\vec{f}_k\}_{k\in\zz}$
of vector-valued functions such that
\begin{align*}
\lf\|\lf\{\vec{f}_k\r\}_{k\in\zz}\r\|_{\ell_{s,b}^{q}(W,L^{p})}
&:=\lf\|\lf\{\lf\|W^{\frac 1 p}\vec{f}_k\r\|\r\}_{k\in\zz}\r\|_{\ell_{s,b}^{q}(\cx,L^{p})}\\
&:=\lf[\sum_{k\in\zz}\lf\{\f{1}{L_{s,b}(\delta^k)}
\lf[\int_\cx\big\|W^{\frac 1 p}(x)\vec{f}_k(x)\big\|^p\,d\mu(x)\r]^{\frac 1 p}\r\}^{q}\r]^{1/q}
\end{align*}
is finite.
We also simply write $\ell^{q}(W,L^{p}):=\ell_{0,0}^{q}(W,L^{p})$.
Here and thereafter, when we say $b=0$, we mean that $b=(0,0).$

We also need the following concept of  reducing operators.

\begin{definition}\label{def-RQ-X}
Let $p\in(0,\fz)$, $W$ be a matrix weight, and $E$ be a measurable set of $\cx$ with $\mu(E)\in(0,\fz)$.
A positive definite $m\times m$ matrix $A_E$ is called a
\emph{reducing operator for $W$ of order $p$} if, for any $\vec{\eta}\in\cc^m$,
\begin{equation*}
\lf\|A_E\vec{\eta}\r\|_\ch\sim\lf\{\fint_E \lf\|W^{1/p}(x)\vec{\eta}\r\|^p\,d\mu(x)\r\}^{1/p},
\end{equation*}
where the positive equivalence constants are independent of both $\vec{\eta}$ and $E$.
For any given family $\mathcal{P}$ of bounded measurable sets,
denote by $\mathcal{RS}_{\mathcal{P}}^{W,p}$ the collection of all
reducing operators for $W$ of order $p$ with respect to $\mathcal{P}$,
namely $\mathcal{RS}_{\mathcal{P}}^{W,p}:=\{A_E:\ E\in\cp\}$.
\end{definition}

\begin{remark}
When $\cx:=\rn$, as argued in  \cite[p.\,450]{V97} for any $p\in(1,\fz)$ and \cite[p.\,1237]{FR04} for any $p\in(0,1)$,
we conclude that, for any $p\in(0,\fz)$ and $W\in A_p(\rn,\cc^m)$,
the sequence $\{A_Q\}_{Q\in\cq}\in\mathcal{RS}_\cq^{W,p}$ as in Definition \ref{def-RQ-X} exists.
The case for spaces of homogeneous type follows from a similar way
and has been proved in \cite[Lemma 2.20]{BYY23}.
\end{remark}

The following lemma is just a variant of \cite[Lemma 2.15 and Corollary 2.17]{BHYY}
on spaces of homogeneous type. Since its proof is similar, we omit the details.

\begin{lemma}\label{lem-AEM}
Let $p\in(0,\fz)$, $W\in A_p(\cx,\cc^m)$,
$\mathcal{P}$ be a given family of measurable subsets with finite positive measure of $\cx$
and $\mathcal{RS}_{\mathcal{P}}^{W,p}$ a family of reducing operators of order $p$ for $W$
with respect to $\mathcal{P}$.
Then, for any bounded measurable set $E\in\mathcal{P}$ and any $m\times m$ complex-valued matrix $M$,
\begin{equation}\label{eq-lem.AEM1}
\lf\|A_E M\r\|\sim\lf[\fint_E\lf\|W^{1/p}(x)M\r\|^p\,d\mu(x)\r]^{1/p}
\end{equation}
and,
especially, when $p\in(1,\fz)$,
\begin{equation}\label{eq-lem.AEM2}
\lf\|A_E^{-1} M\r\|\sim\lf[\fint_E\lf\|W^{-1/p}(x)M\r\|^{p'}\,d\mu(x)\r]^{1/p'},
\end{equation}
where the positive equivalence constants are all independent of both $E$ and $M$.
\end{lemma}

The following concept of   matrix $A_p$-weights  for any $p\in(0,\fz)$
on $(\cx,d,\mu)$ was introduced in \cite[Lemma 2.29 and Definition 5.1]{BYY23},
which is a generalization of the classical one defined on $\rn$
(see, \cite[Lemma 1.3]{R03} and \cite[p.\,1227]{FR04}).

\begin{definition}\label{def-Apweight}
Let  $p\in(0,\fz)$. A matrix weight $W$ is called a \emph{matrix $A_p$-weight} on $\cx$,
denoted by $W\in A_p(\cx,\cc^m)$, if
\begin{enumerate}
\item[\rm(i)] when $p\in(0,1]$, there exists a positive constant $C$ such that, for any ball $B\subset \cx$,
\begin{equation*}
\esssup_{y\in B}\fint_B\lf\|W^{1/p}(x)W^{-1/p}(y)\r\|^{p}\,d\mu(x)\le C;
\end{equation*}

\item[\rm(ii)] when $p\in(1,\fz)$,
there exists a positive constant $C$ such that, for any ball $B\subset \cx$,
$$
\fint_B\lf[\fint_B\lf\|W^{1/p}(x)W^{-1/p}(y)\r\|^{p'}\,d\mu(y)\r]^{p/p'}\,d\mu(x)\le C,
$$
where $1/p+1/p'=1$.
\end{enumerate}
\end{definition}

Next, we  recall the system of dyadic cubes on spaces of homogeneous type.
The following lemma is just \cite[Theorem 2.2]{HK12}, which guarantees the existence of
the dyadic cube decomposition of $\cx$.

\begin{lemma}\label{lem-dyacube}
Let $\delta\in(0,1)$ and $0<c_0\le C_0<\fz$ satisfy $12A_0^3C_0\delta\le c_0$.
Let $\{z_{\alpha}^k:\ \alpha\in\ci_k,\ k\in\zz\}\subset \cx$
(for any $k\in\zz$, $\ci_k$ denotes some index set)
be a set of points satisfying that, for any $k\in\zz$,
$$
d(z_{\alpha}^k,z_{\beta}^k)\ge c_0\delta^k\ \ \text{if}\ \alpha\neq\beta
$$
and, for any $x\in\cx$,
$$
\min_{\alpha\in\ci_k}d(x,z_{\alpha}^k)< C_0\delta^k.
$$
Then there exists a family of sets,
\begin{equation*}
\cq:=\bigcup_{k\in\zz}\cq_k:=\bigcup_{k\in\zz}\lf\{Q_\alpha^k\subset \cx:\ \alpha\in\ci_k\r\},
\end{equation*}
such that
\begin{enumerate}
\item[\rm(i)] for any $k\in\zz$,
$$
\cx= \bigcup_{\alpha\in\ci_k}Q_\alpha^k\quad\text{(pairwise disjoint union)};
$$
\item[\rm(ii)] for any $j,k\in\zz$ with $j\ge k$, $\alpha\in \ci_k$, and $\beta\in\ci_j$,
either $Q_\beta^j\subset Q_\alpha^k$ or $Q_\beta^j\cap Q_\alpha^k=\emptyset$;
\item[\rm(iii)] for any $k\in\zz$ and $\alpha\in \ci_k$, there exists a point $z_\alpha^k\in\cx$
such that
\begin{equation}\label{eq-defB(Q)}
B(z_\alpha^k,(3A_0^2)^{-1} c_0\delta^{k})\subset Q_\alpha^k\subset B(z_\alpha^k,2A_0C_0\delta^{k})=:B(Q_\alpha^k).
\end{equation}
\end{enumerate}
\end{lemma}

In Lemma \ref{lem-dyacube}, the family $\cq$ of sets, gives the (half-open) dyadic cube system of $(\cx,d,\mu)$.
\emph{Throughout this article}, the point $z_\alpha^k=:z_{Q_\alpha^k}$ in (iii)
is called the \emph{center} of $Q_\alpha^k$
and $\ell (Q_\alpha^k):=\delta^{k}$ is called the \emph{diameter} of $Q_\alpha^k$.
Moreover, for any $k\in\zz$, define
\begin{equation*}
\cg_k:=\mathcal{I}_{k+1}\setminus\mathcal{I}_{k}
\end{equation*}
with $\mathcal{I}_{k}$ the same index set coming from the construction of the dyadic cube system as above and
define
\begin{equation}\label{eq-def-yak}
\mathcal{Y}^k:=\lf\{y_\alpha^{k}\r\}_{\alpha\in\cg_k}:=\lf\{z_\alpha^{k+1}\r\}_{\alpha\in\cg_k}.
\end{equation}

\begin{remark}\label{remark-dyadic}
The study on the system of dyadic cubes started from 1990s and
we refer the reader to \cite{C90,D91} for some earlier references.
In what follows, we always fix
$$
\delta\in\left(0,\min\left\{A_0^{-D},\frac{1}{2}A_0^{-1},\frac{1}{12}A_0^{-3}\right\}\right)
$$
and $c_0,C_0\in(0,\fz)$ with $c_0\le C_0$
and use $\cq$ to denote such a system of dyadic cubes as in Lemma \ref{lem-dyacube}.
\end{remark}

The following lemma reveals some relations between matrix $A_p$-weights
and their reducing operators,
which is just a variant on $\cx$ of \cite[Proposition 2.4 and Lemma 3.3]{G03}
and \cite[Lemma 3.3]{FR21} on $\rn$; we omit its proof  due to the
independence of the original proofs of both the reverse doubling
property of the measure and the metric condition under consideration.

\begin{lemma}\label{lem-AQW}
Let $p\in(0,\fz)$, $W\in A_p(\cx,\cc^m)$,
and $\{A_Q\}\in\mathcal{RS}_{\mathcal{\cq}}^{W,p}$.
\begin{enumerate}
\item[\rm(i)]
Let $p\in(0,1]$. Then
\begin{equation}\label{lem-eq-AP,W,p<1}
\sup_{Q\in\cq}\esssup_{x\in Q}\lf\|A_Q W^{-\frac 1 p}(x)\r\|<\fz
\end{equation}
and there exists $\epsilon\in(0,\fz)$, depending only on $W$, such that, for any $r\in(0,p+\epsilon)$,
\begin{equation*}
\sup_{Q\in\cq}\fint_Q\lf\|W^{\frac 1 p}(x)A_Q^{-1}\r\|^r\,d\mu(x)\le C_{(r)}
\end{equation*}
and
\begin{equation}\label{lem-eq-W,AP',p<1}
\sup_{Q\in\cq}\fint_Q\sup_{P\in\cq:\ x\in P\subset Q}\lf\|W^{\frac 1 p}(x)A_P^{-1}\r\|^r\,d\mu(x)\le C_{(r)}
\end{equation}
for some positive constant $C_{(r)}$.
\item[\rm(ii)]
Let $p\in(1,\fz)$. Then there exists $\epsilon\in(0,\fz)$ such that,
for any $r\in(0,p'+\epsilon)$,
\begin{equation}\label{lem-eq-AP,W,p>1}
\sup_{Q\in\cq}\fint_Q\lf\|A_Q W^{-\frac 1 p}(x)\r\|^r\,d\mu(x)\le C_{(r)}
\end{equation}
and,  for any $r<p+\epsilon$,
\begin{equation*}
\sup_{Q\in\cq}\fint_Q\lf\|W^{\frac 1 p}(x)A_Q^{-1}\r\|^r\,d\mu(x)\le C_{(r)}
\end{equation*}
and
\begin{equation}\label{lem-eq-W,AP',p>1}
\sup_{Q\in\cq}\fint_Q\sup_{P\in\cq:\ x\in P\subset Q}\lf\|W^{\frac 1 p}(x)A_P^{-1}\r\|^r\,d\mu(x)\le C_{(r)}
\end{equation}
for some positive constant $C_{(r)}$.
\end{enumerate}
\end{lemma}

Now,  we introduce the $A_p$ dimension
on $\cx$ as follows,
which is a generalization of \cite[Definition 2.20]{BHYY} from $\rn$ to $\cx$.

\begin{definition}\label{Ap dim lower}
Let $p\in(0,\infty)$, $d\in\mathbb{R}$, and $W$ be a matrix weight.
Then $W$ is said to have the \emph{$A_p$ dimension $d$},
denoted by $W\in\mathbb{D}_{p,d}(\cx,\mathbb{C}^m)$,
if
\begin{enumerate}
\item[\rm(i)] when $p\in(0,1]$,
there exists a positive constant $C$ such that,
for any ball $B\subset\cx$ and any $i\in\mathbb{Z}_+$,
\begin{equation}\label{eq-ApD-p<1}
\mathop{\mathrm{\,ess\,sup\,}}_{x\in\delta^{-i}B}\fint_B \left\|W^{\frac{1}{p}}(y)W^{-\frac{1}{p}}(x)\right\|^p\,d\mu(y)
\leq C\delta^{-id};
\end{equation}

\item[\rm(ii)] when $p\in(1,\infty)$,
there exists a positive constant $C$ such that,
for any ball $B\subset\cx$ and any $i\in\mathbb{Z}_+$,
\begin{equation}\label{eq-ApD-p>1-low}
\fint_B\left[\fint_{\delta^{-i}B}\left\|W^{\frac{1}{p}}(x)W^{-\frac{1}{p}}(y)\right\|^{p'}
\,d\mu(y)\right]^{\frac{p}{p'}}\,d\mu(x)
\leq C\delta^{-id},
\end{equation}
where $1/p+1/p'=1$.
\end{enumerate}
\end{definition}

\begin{lemma}\label{lem-d<D}
Let $p\in(0,\fz)$ and $W\in A_p(\cx,\cc^m)$.
Then there exists $d\in[0,D)$ such that
$W$ has the $A_p$ dimension $d$,
where $D$ is the same as in \eqref{eq-doub}.
\end{lemma}

Lemma \ref{lem-d<D} follows from the variant of a reverse H\"{o}lder inequality
(see \cite[Proposition 2.32]{BYY23}) on $\cx$, the doubling property of $\mu$,
and an argument similar to that used in the proof of \cite[Proposition 2.27]{BHYY};
we omit its detailed proof.

The following lemma, modified from \cite[Proposition 2.28]{BHYY}, gives an
estimate where the roles $B$ and $\delta^{-i}B$ play in Definition \ref{Ap dim lower} are exchanged.

\begin{lemma}\label{Ap dim upper}
Let $p\in(0,\infty)$ and $W\in A_p(\cx,\cc^m)$.
\begin{enumerate}
\item[\rm(i)] If $p\in(0,1]$, then
there exists a positive constant $C$ such that,
for any ball $B\subset\cx$ and any $i\in\mathbb{Z}_+$,
\begin{equation*}
\mathop{\mathrm{\,ess\,sup\,}}_{x\in B}\fint_{\delta^{-i}B} \left\|W^{\frac{1}{p}}(y)W^{-\frac{1}{p}}(x)\right\|^p\,d\mu(y)
\leq C.
\end{equation*}

\item[\rm(ii)]
If $p\in(1,\infty)$, then there exists a positive constant $C$ such that,
for any ball $B\subset\cx$ and any $i\in\mathbb{Z}_+$,
\begin{equation}\label{eq-ApD-p>1-up}
\fint_{\delta^{-i}B}\left[\fint_B\left\|W^{\frac{1}{p}}(x)
W^{-\frac{1}{p}}(y)\right\|^{p'}\,d\mu(y)\right]^{\frac{p}{p'}}\,d\mu(x)
\leq C\delta^{-i\wz{d}(p-1)}
\end{equation}
if and only if the dual weight $\wz{W}:=W^{-1/(p-1)}\in A_{p'}(\cx,\cc^m)$ has the $A_{p'}$-dimension $\wz{d}$.
\end{enumerate}
\end{lemma}

We have the following estimate about reducing operators,
which, in some sense, is a variant for $\cx$ of \cite[Lemma 2.29]{BHYY} on $\rn$.

\begin{lemma}\label{lem-sharp est}
Let  $p\in(0,\infty)$, $W\in A_p(\cx,\cc^m)$ have the $A_p$ dimension $d\in[0,D)$,
and $\{A_Q\}_{Q\in\cq}\in \mathcal{RS}_{\mathcal{Q}}^{W,p}$.
\begin{enumerate}
\item[\rm(i)] If $p\in(0,1]$,
then there exists a positive constant $C$ such that,
for any $P,Q\in\cq$,
\begin{align}\label{eq-sharp-est-p<1}
\left\|A_QA_P^{-1}\right\|
&\leq C\max\left\{\left[\frac{\ell(P)}{\ell(Q)}\right]^{\frac d p},1\right\}
\left[1+\frac{d(z_P,z_Q)}{\max\{\ell(P),\ell(Q)\}}\right]^{\frac d p}.
\end{align}

\item[\rm(ii)] If $p\in(1,\fz)$ and the dual weight
$\wz{W}\in A_{p'}(\cx,\cc^m)$ has the $A_{p'}$ dimension $\wz{d}\in[0,D)$,
then there exists a positive constant $C$ such that,
for any $P,Q\in\cq$,
\begin{align}\label{eq-sharp-est-p>1}
\left\|A_QA_P^{-1}\right\|
&\leq C\max\left\{\left[\frac{\ell(P)}{\ell(Q)}\right]^{\frac d p},
\left[\frac{\ell(Q)}{\ell(P)}\right]^{\frac{\wz{d}}{p'}}\right\}
\left[1+\frac{d(z_P,z_Q)}{\max\{\ell(P),\ell(Q)\}}\right]^{\frac d p+\frac{\wz{d}}{p'}}.
\end{align}
\end{enumerate}
Here, $z_P$ and $z_Q$ are, respectively, the centers of the dyadic cubes $P$ and $Q$ defined in \eqref{eq-defB(Q)}.
\end{lemma}

\begin{proof}
Let $p\in(0,\fz)$ and $\{A_E:\ E\subset \cx\ \text{bounded measurable}\}$
be a family of reducing operators of order $p$ for $W$.
We consider the following two cases on $p$.

\noindent\textbf{Case 1)} When $p\in(0,1]$,
we first claim that there exists a positive constant $C$ such that,
for any $i\in\zz_+$, any $Q\in\cq$, and any bounded measurable set $E\subset \cx$
with $E\subset \delta^{-i} B(Q)$,
\begin{equation}\label{eq-est-claim-1}
\lf\|A_Q A_E^{-1}\r\|^p\le C\delta^{-id},
\end{equation}
where $B(Q)$ is the same as in \eqref{eq-defB(Q)}.
Indeed, using \eqref{eq-lem.AEM1}, \eqref{eq-defB(Q)},
the doubling property of $\mu$, $E\subset \delta^{-i}B(Q)$, and \eqref{eq-ApD-p<1},
we find that
\begin{align*}
\lf\|A_Q A_E^{-1}\r\|^p&\le\essinf_{x\in E}\lf\|A_Q W^{-\frac 1 p}(x)\r\|^p\lf\|W^{\frac 1 p}(x) A_E^{-1}\r\|^p\\
&\sim\essinf_{x\in E}\fint_Q\lf\|W^{\frac 1 p}(y) W^{-\frac 1 p}(x)\r\|^p\,d\mu(y)\,\lf\|W^{\frac 1 p}(x) A_E^{-1}\r\|^p\\
&\ls\essinf_{x\in E}\fint_{B(Q)}\lf\|W^{\frac 1 p}(y) W^{-\frac 1 p}(x)\r\|^p\,d\mu(y)\,\lf\|W^{\frac 1 p}(x) A_E^{-1}\r\|^p\\
&\ls\esssup_{x\in \delta^{-i}B(Q)}\fint_{B(Q)}\lf\|W^{\frac 1 p}(y) W^{-\frac 1 p}(x)\r\|^p\,d\mu(y)
\essinf_{x\in E}\lf\|W^{\frac 1 p}(x) A_E^{-1}\r\|^p\\
&\ls\delta^{-id}\essinf_{x\in E}\lf\|W^{\frac 1 p}(x) A_E^{-1}\r\|^p
\ls\delta^{-id}\fint_E\lf\|W^{\frac 1 p}(x) A_E^{-1}\r\|^p\,d\mu(x)\\
&\sim\delta^{-id}\lf\|A_E A_E^{-1}\r\|^p
\sim\delta^{-id},
\end{align*}
which proves \eqref{eq-est-claim-1}.
Furthermore, we can similarly prove that \eqref{eq-est-claim-1} still holds
if the dyadic cube $Q$ therein is replaced by any ball $B$,
$B(Q)$ is replaced by $B$, and $E$ is still an arbitrary measurable subset of $\delta^{-i} B$.

Next, let $P,Q\in\cq$.
We choose a ball $B$ containing both $P$ and $Q$ such that
\begin{equation}\label{eq-diam(B)}
\ell(B)\sim A_0^2\lf[\ell(P)+d(z_P,z_Q)+\ell(Q)\r];
\end{equation}
here and thereafter, $z_P$ and $z_Q$ are, respectively, the centers of $P$ and $Q$, and $\ell(B)$ denotes
the diameter of $B$.
Observe that there exists some $i\in\zz_+$ with $4A_0C_0\delta^{-i} \sim\ell(B)/\ell(Q)$
such that $B\subset \delta^{-i} B(Q)$.
From this, \eqref{eq-est-claim-1}, $P\subset B$, and \eqref{eq-diam(B)}, we infer that
\begin{align*}
\lf\|A_Q A_P^{-1}\r\|^p&\le\lf\|A_Q A_B^{-1}\r\|^p\lf\|A_B A_P^{-1}\r\|^p
\ls\lf[\f{\ell(B)}{\ell(Q)}\r]^{d} \\
&\sim \max\lf\{\lf[\f{\ell(P)}{\ell(Q)}\r]^{d},1\r\}
\lf[1+\f{d(z_P,z_Q)}{\max\{\ell(P),\ell(Q)\}}\r]^{d}.
\end{align*}
This shows the case when $p\in(1,\fz)$.

\noindent\textbf{Case 2)} When $p\in(1,\fz)$,
we claim that there exist positive constants $C_1$ and $C_2$ such that,
 for any $i_1,i_2\in\zz_+$, any $P,Q\in\cq$, and any measurable set $E\subset \cx$
with $E\subset \delta^{-i_1} B(P)$, $E\subset \delta^{-i_2} B(Q)$,
$\mu(\delta^{-i_1} B(P))/\mu(E)\ls1$, and $\mu(\delta^{-i_2} B(Q))/\mu(E)\ls1$, it holds that
\begin{equation}\label{eq-est-claim-2}
\lf\|A_E A_P^{-1}\r\|^p\le C_2 \delta^{-i_1\wz{d}(p-1)}\
\text{and}\ \lf\|A_Q A_E^{-1}\r\|^p\le C_1\delta^{-i_2 d}.
\end{equation}
Indeed,
notice that, for any $A,B\in \mathcal{RS}_{\mathcal{Q}}^{W,p}$, $\|AB\|=\|BA\|$.
Thus, applying this, Lemma \ref{lem-AEM}, \eqref{eq-lem.AEM2} twice,
the doubling property of $\mu$, $E \subset \delta^{-i_1} B(P)$,
$\mu(\delta^{-i_1} B(P))/\mu(E)\ls1$,  \eqref{eq-defB(Q)}, and \eqref{eq-ApD-p>1-up}, we obtain
\begin{align*}
\lf\|A_E A_P^{-1}\r\|^p
&\ls\fint_E\lf[\fint_P\lf\|W^{-\frac 1 p}(y)W^{\frac 1 p}(x)\r\|^{p'}\,d\mu(y)\r]^{\frac{p}{p'}}\,d\mu(x)\\
&\ls\fint_{\delta^{-i_1}B(P)}\lf[\fint_{B(P)}\lf\|W^{\frac 1 p}(x)W^{-\frac 1 p}(y)\r\|^{p'}\,d\mu(y)\r]^{\frac{p}{p'}}\,d\mu(x)
\ls\delta^{-i_1\wz{d}(p-1)},
\end{align*}
which proves the first inequality of \eqref{eq-est-claim-2}.
Using \eqref{eq-lem.AEM2} twice, $\mu(\delta^{-i_2} B(Q))/\mu(E)\ls1$,
$E \subset \delta^{-i_2}B(Q)$, the doubling property of $\mu$, and \eqref{eq-ApD-p>1-low},
we also conclude that
\begin{align*}
\lf\|A_Q A_E^{-1}\r\|^p
&\le\lf[\fint_Q\lf\|A_Q W^{-\frac 1 p}(x)\r\|^{p'}\,d\mu(x)\r]^{\frac{p}{p'}}
\fint_Q\lf\|W^{\frac 1 p}(x) A_E^{-1}\r\|^p\,d\mu(x)\\
&\sim\lf\|A_Q^{-1}A_Q\r\|^p
\fint_Q\lf[\fint_E\lf\|W^{-\frac 1 p}(y)W^{\frac 1 p}(x)\r\|^{p'}\,d\mu(y)\r]^{\frac{p}{p'}}\,d\mu(x)\\
&\ls\fint_{B(Q)}\lf[\fint_{\delta^{-i_2} B(Q)}\lf\|W^{\frac 1 p}(x)W^{-\frac 1 p}(y)\r\|^{p'}\,d\mu(y)\r]^{\frac{p}{p'}}\,d\mu(x)
\ls\delta^{-i_2d}.
\end{align*}
This proves \eqref{eq-est-claim-2}.
Again, \eqref{eq-est-claim-2} still holds
if the dyadic cubes $Q,R$ therein are replaced by any two balls $B_1,B_2$
and if $B(Q),B(R)$ are replaced by $B_1,B_2$.

Let $P,Q\in\cq$.
We still choose a ball $B$ containing both $P$ and $Q$ and satisfying
\begin{equation*}
\ell(B)\sim A_0^2\lf[\ell(P)+d(z_P,z_Q)+\ell(Q)\r].
\end{equation*}
By a geometrical observation, we find that
there exist some $i_1, i_2\in\zz_+$ with
$$
4A_0C_0\delta^{-i_1}\sim\ell(B)/\ell(P)\ \ \text{and}\ \
4A_0C_0\delta^{-i_2}\sim\ell(B)/\ell(Q)
$$
such that
$$
B\subset \delta^{-i_1} B(P)\subset \delta^{-1}B\ \ \text{and}\ \
B\subset \delta^{-i_2}B(Q)\subset \delta^{-1}B.
$$
Furthermore, by the doubling property of $\mu$, we obtain both
$$\frac{\mu(\delta^{-i_1} B(P))}{\mu(B)}\ls1\ \ \text{and}\ \ \frac{\mu(\delta^{-i_2} B(Q))}{\mu(B)}\ls1.$$
Thus, by this and \eqref{eq-est-claim-2} with $E$ replaced by $B$, we conclude that
\begin{align*}
\lf\|A_Q A_P^{-1}\r\|^p&\le\lf\|A_Q A_B^{-1}\r\|^p\lf\|A_B A_P^{-1}\r\|^p
\ls\lf[\f{\ell(B)}{\ell(Q)}\r]^{d}\lf[\f{\ell(B)}{\ell(P)}\r]^{\wz{d}(p-1)}\\
&\sim \max\lf\{\lf[\f{\ell(P)}{\ell(Q)}\r]^{d},\lf[\f{\ell(Q)}{\ell(P)}\r]^{\wz{d}(p-1)}\r\}
\lf[1+\f{d(z_P,z_Q)}{\max\{\ell(P),\ell(Q)\}}\r]^{d+\wz{d}(p-1)}.
\end{align*}
This finishes the proof of the case $p\in(1,\fz)$ and hence  Lemma \ref{lem-sharp est}.
\end{proof}

In the following, we frequently use the same assumptions of $W$ as in Lemma \ref{lem-sharp est};
for simplicity, we use the concept of the $A_p$ dimension pair $(d,\wz{d})$.

\begin{definition}\label{def-Apdim pair}
Let $p\in(0,\infty)$. A weight $W\in A_p(\cx,\cc^m)$
is said to \emph{have the $A_p$ dimension pair $(d,\wz{d})$}
if
\begin{enumerate}
\item[\rm(i)] $W$ has the $A_p$ dimension $d\in[0,D)$;
\item[\rm(ii)] especially, when $p\in(1,\fz)$, its dual weight $\wz{W}\in A_{p'}(\cx,\cc^m)$
has the $A_{p'}$ dimension $\wz{d}\in[0,D)$.
\end{enumerate}
\end{definition}

\begin{remark}\label{remark-dp}
As argued in \cite[Lemmas 2.45 and 2.47]{BHYY}, both \eqref{eq-sharp-est-p<1} and \eqref{eq-sharp-est-p>1}
when $\cx:=\rn$ and $\delta:=1/2$ are sharp in the sense that,
if there exist $r_1,r_2,r_3\in(0,\fz)$ such that,
for any $W\in A_p(\rn,\cc^m)$ having the $A_p$ dimension pair $(d,\wz{d})$
and $\{A_Q\}_{Q\in\cq}\in\mathcal{RS}_\cq^{W,p}$,
there exists a positive constant $C$ satisfying, for any $P,Q\in\cq,$
$$
\lf\|A_Q\,A_P^{-1}\r\|\le C\max\lf\{\lf[\f{l(Q)}{l(P)}\r]^{r_1},\lf[\f{l(P)}{l(Q)}\r]^{r_2}\r\}
\lf[1+\f{|x_P-x_Q|}{\max\{l(Q),l(P)\}}\r]^{r_3},
$$
then $r_1\in[d/p,\fz)$, $r_2\in[\wz{d}_{(p)}/p,\fz)$, and $r_3\in[d/p+\wz{d}_{(p)}/p,\fz)$,
where
\begin{align}\label{eq-def-dp}
\wz{d}_{(p)}:=
\begin{cases}
0,&\ \text{when}\ p\in(0,1],\\
\wz{d}(p-1),&\ \text{when}\ p\in(1,\fz).
\end{cases}
\end{align}
\end{remark}

\begin{remark}
Let $p\in(0,\fz)$, $W\in A_p(\cx,\cc^m)$ have the $A_p$ dimension pair $(d,\wz{d})$,
and $\wz{d}_{(p)}$ be defined in \eqref{eq-def-dp}.
Then  Lemma \ref{lem-sharp est} implies that, for
any $\{A_Q\}_{Q\in\cq}\in\mathcal{RS}_\cq^{W,p}$,
there exists a positive constant $C$ such that,
for any $P,Q\in\cq$,
\begin{align}\label{eq-sharp-est-p}
\left\|A_QA_P^{-1}\right\|
&\leq C\max\left\{\left[\frac{\diam(P)}{\diam(Q)}\right]^{\frac d p},
\left[\frac{\diam(Q)}{\diam(P)}\right]^{\frac{\wz{d}_{(p)}}{p}}\right\}\nonumber\\
&\quad\times
\left[1+\frac{d(z_P,z_Q)}{\max\{\diam(P),\diam(Q)\}}\right]^{\frac d p+\frac{\wz{d}_{(p)}}{p}}.
\end{align}
Recall that, when $\cx=\mathbb{R}^d$, $\{A_Q\}_{Q\in\cq}$ satisfying \eqref{eq-sharp-est-p}
is said to be strongly doubling of order $(\beta,p)$ with $\beta=d+\wz{d}_{(p)}$ in the sense of
Frazier and Roudenko \cite[Definition 2.1]{FR21}.
\end{remark}

The following lemma implies that, for any two dyadic cubes $P$ and $Q$
which are not far away from each other,
the reducing operators $A_P$ and $A_Q$, related to some given matrix $A_p$-weight, are comparable.

\begin{lemma}\label{lem-AQ=AQ'}
Let $p\in(0,\fz)$, $W\in A_p(\cx,\cc^m)$,
and $\{A_Q\}_{Q\in\cq}\in \mathcal{RS}_{\mathcal{Q}}^{W,p}$.
Then, for any $k\in\zz$, $x\in\cx$, $B:=B(x,\delta^{k-\tau})$ for some fixed $\tau\in\zz$,
and $Q,Q'\in\cq_k$ with both $Q\cap B\neq\emptyset$ and $Q'\cap B\neq\emptyset$,
\begin{equation*}
\left\|A_{Q}A_{Q'}^{-1}\right\|\sim1
\end{equation*}
with the  positive equivalence constants independent of both $x$ and $k$.
\end{lemma}

Lemma \ref{lem-AQ=AQ'} follows from applying \eqref{eq-sharp-est-p}; we omit its proof here.

\section{Matrix-Weighted Poincar\'{e}-Type Inequalities}\label{sec-poincare}

Poincar\'{e}-type inequalities play a key role in studying smoothness function spaces.
We refer the reader to, for instance, \cite{Haj96,HK00,Haj03,KYZ10,KYZ11,LYY22,LYY21,AWYY23}.
Especially, on the matrix-weighted Euclidean space,
Isralowitz and  Moen \cite{IM19} established some matrix-weighted variants of
the classical Poincar\'{e} inequalities for the pair $(\vec f, D\vec f)$, where
$D\vec f$ is the derivatives of $\vec f$,  and further applied them to the study on the degenerate systems
of elliptic equations.
The aim of this section is to establish some new matrix-weighted Poincar\'{e}-type inequalities
with respect to the
vector-valued $(s,b)$-Haj{\l}asz gradient sequences on the space $(\cx,d,\mu)$ of homogeneous type.
These  inequalities will be used to establish the pointwise characterization of
the matrix-weighted logarithmic Besov spaces.

We begin with some basic properties of the logarithmic smoothness weight functions $L_{s,b}$ under consideration.
Let $s\in\rr$ and $b:=(b_+,b_-)$.
Recall that
$L_{s,b}$, defined in \eqref{eq-def-Log}, is a particular example among those
generalized smoothness weight functions of admissible growth
considered in \cite{LYY22}
with the indices therein $\alpha_\phi=s=\beta_\phi$;
furthermore,
by \cite[(9)]{LYY21},  we find that,
for any given $s\in(0,\fz)$ and $r\in(0,\fz]$,
there exists a positive constant $C_0$ such that, for any $k_0\in\zz$,
\begin{equation*}
\sum_{k=-\fz}^{k_0}\lf[\f{L_{s,b}(\delta^{k_0})}{L_{s,b}(\delta^{k})}\r]^r\le C_0;
\end{equation*}
by \cite[(2.5) and (2.6)]{LYY22}, we also conclude that,
for any $s\in(0,1)$ and $N\in\zz$,
there exist positive constants $C_1$ and $C_2$, depending only on $s$, $b$, and $N$, such that
\begin{equation*}
C_1\le\delta^{Ns}\lf[\f{(1+k)\vee1}{(1+k+N)\vee1}\r]^{b_+}\lf[\f{(1+k)\wedge1}{(1+k+N)\wedge1}\r]^{b_-}\le C_2.
\end{equation*}

Now, we introduce the concept of the vector-valued
$(s,b)$-Haj{\l}asz gradient sequences on $\cx$.

\begin{definition}\label{def-gradient}
Let $s\in(0,\fz)$, $b:=(b_+,b_-)$ be a pair of real numbers,
$p\in(0,\fz)$, $W\in A_p(\cx,\cc^m)$,
$\{A_Q\}_{Q\in\cq}\in\mathcal{RS}_\cq^{W,p}$,
and $\vec{f}\in [L^0(\cx)]^m$.
A sequence $\vec{\mathbf{g}}:=\{\vec{g}_{k}\}_{k\in\zz}$ of vector-valued functions
is called a \emph{vector-valued $(s,b)$-Haj{\l}asz gradient sequence}
of $\vec{f}$ with respect to $\{A_Q\}_{Q\in\cq}$
if, for any $k\in\zz$, there exists a set $E_k\subset \cx$ with $\mu(E_k)=0$ such that,
for any $x,y\in \cx\setminus E_k$ with $d(x,y)\in[\delta^{k+1},\delta^{k})$,
\begin{align}\label{eq-def-gradient}
&\sum_{Q\in\cq_k}\lf\|A_{Q}\vec{f}(x)-A_{Q}\vec{f}(y)\r\|\mathbf{1}_Q(x)\nonumber\\
&\quad\le L_{s,b}\lf(d(x,y)\r)
\sum_{Q\in\cq_k}\lf[\lf\|A_Q\vec{g}_{k}(x)\r\|+\lf\|A_Q\vec{g}_{k}(y)\r\|\r]\mathbf{1}_Q(x).
\end{align}
Denote by $\dd^{s,b}_{\{A_Q\}}(\vec{f})$ the collection of all vector-valued $(s,b)$-Haj{\l}asz gradient sequences
of $\vec{f}$ with respect to $\{A_Q\}_{Q\in\cq}$ and, especially,
by $\dd^{s,0}_{\{A_Q\}}(\vec{f})$ the collection
when $b:=(0,0)$.
\end{definition}

\begin{remark}\label{remark-gradient}
\begin{enumerate}
\item[\rm(i)] In the unweighted scalar case,
\eqref{eq-def-gradient} in Definition \ref{def-gradient} becomes
\begin{equation}\label{eq-def-cla-gradients}
\lf|f(x)-f(y)\r|\le L_{s,b}\lf(d(x,y)\r)
\lf[g_k(x)+g_k(y)\r],
\end{equation}
and we use $\dd^{s,b}(f)$ to denote the collection of all  $(s,b)$-Haj{\l}asz gradient sequences in this case.
Furthermore, when $\delta=1/2$ and $b=0$, \eqref{eq-def-cla-gradients} reduces back to
the classical gradient inequality in \cite[Definition 1.1]{KYZ11}.
We also recall that
a nonnegative function $g$ is called an \emph{$s$-Haj{\l}asz gradient} of $f$ if,
for almost every $x,y\in\cx$,
$$
\lf|f(x)-f(y)\r|\le\lf[d(x,y)\r]^s\lf[g(x)+g(y)\r];
$$
see \cite[Definition 2.10]{AWYY23}.
 For other variants of the classical gradient inequality, we refer the reader to, for instance,
\cite[Definition 3.1]{LYY22} for the generalized smoothness and \cite[Definition 2.10]{AWYY23}
for $\cx$.

\item[\rm(ii)] We claim that  replacing  $\mathbf{1}_Q(x)$ in \eqref{eq-def-gradient}  into
$\mathbf{1}_Q(y)$ gives an equivalent definition of  vector-valued $(s,b)$-Haj{\l}asz gradient sequences.
Indeed, for any $x,y\in\cx$ with $d(x,y)\in[\delta^{k+1},\delta^{k})$
and $P\in\cq_k$ containing $y$, we have $x\in B(y,\delta^k)$ and hence,
for any $Q\in\cq_k$ containing $x$ (we may assume that $P\neq Q$),
$P\cap B(y,\delta^k)\neq\emptyset$
and $Q\cap B(y,\delta^k)\neq\emptyset$,
which, together with Lemma \ref{lem-AQ=AQ'}, further implies that
$\|A_PA_Q^{-1}\|\sim 1\sim\|A_QA_P^{-1}\|$.
Thus, if \eqref{eq-def-gradient} holds, applying this,
we conclude that there exists a positive constant $C$ such that,
for any $k\in\zz$ and almost every $x,y\in \cx$ with $d(x,y)\in[\delta^{k+1},\delta^{k})$,
\begin{align*}
&\sum_{P\in\cq_k}\lf\|A_{P}\vec{f}(x)-A_{P}\vec{f}(y)\r\|\mathbf{1}_P(y)\\
&\quad\le C\sum_{Q\in\cq_k}\lf\|A_{Q}\vec{f}(x)-A_{Q}\vec{f}(y)\r\|\mathbf{1}_Q(x)\\
&\quad\le C L_{s,b}\lf(d(x,y)\r)
\sum_{Q\in\cq_k}\lf[\lf\|A_Q\vec{g}_{k}(x)\r\|+\lf\|A_Q\vec{g}_{k}(y)\r\|\r]\mathbf{1}_Q(x)\\
&\quad\le C L_{s,b}\lf(d(x,y)\r)
\sum_{P\in\cq_k}\lf[\lf\|A_P\vec{g}_{k}(x)\r\|+\lf\|A_P\vec{g}_{k}(y)\r\|\r]\mathbf{1}_P(y),
\end{align*}
which proves the above claim.
\end{enumerate}
\end{remark}

We need the following lemma to reduce the transform from the vector-valued case to the scalar-valued case.

\begin{lemma}\label{lem-h,hj in D}
Let $s\in(0,\fz)$, $b:=(b_+,b_-)$ be a pair of real numbers,
$p\in(0,\fz)$, $W\in A_p(\cx,\cc^m)$ have the $A_p$ dimension pair $(d,\wz{d})$, and
$\{A_Q\}_{Q\in\cq}\in\mathcal{RS}_\cq^{W,p}$.
Then, for any given $k_0\in\zz$, $x_0\in\cx$, $B_0:=B(x_0,\delta^{k_0})$,
$Q_0\in \cq_{k_0}$ containing $x_0$,
$\vec{f}\in[L^0(\delta^{-1}B_0)]^m$,
and $\{\vec{g}_j\}_{j\in\zz}\in\dd^{s,b}_{\{A_Q\}}(\vec{f})$,
the sequence $\{h_j\}_{j\in\zz}$ of functions defined on $\delta^{-1}B_0$ by setting,
for any $j\in\zz$,
\begin{equation}\label{eq-def-gj}
h_j:=\begin{cases}
\ \delta^{[k_0-j]\wz{d}_{(p)}/p}\dsum_{Q\in\cq_j}
\lf\|A_{Q}\vec{g}_j\r\| \mathbf{1}_Q\ &\text{if}\ j\ge k_0-2,\\
\ 0\ &\text{if}\ j< k_0-2
\end{cases}
\end{equation}
is a positive constant multiple of an element in $\dd^{s,b}(u_\eta)$,
where, for any given $\vec{\eta}\in\cc^m$,
the function $u_\eta$ is defined on $\delta^{-1}B_0$ by setting, for any $y\in\delta^{-1}B_0$
\begin{equation}\label{eq-def-u}
u_\eta(y):=\lf\|A_{Q_0}\vec{f}(y)-\vec{\eta}\r\|
\end{equation}
and where the positive constant is independent of $x_0$, $k_0$, and $\vec{\eta}$.
\end{lemma}

\begin{proof}
As $\delta<(2A_0)^{-1/2}$, for any $y,z\in B_0$, we have $d(y,z)<\delta^{k_0-2}$.
Applying Remark \ref{remark-dp},
\eqref{eq-def-gradient}, and  Lemma \ref{lem-AQ=AQ'},
we further conclude that,
for any $j\ge k_0-2$ and almost every $y,z\in \delta^{-1}B_0$ such that $d(y,z)\in[\delta^{j+1},\delta^{j})$,
\begin{align*}
&\lf|u_\eta(y)-u_\eta(z)\r|\\
&\quad\le
\lf\|A_{Q_{0}}\vec{f}(y)-A_{Q_{0}}\vec{f}(z)\r\|\\
&\quad\le\sum_{Q\in\cq_j}
\lf\|A_{Q_0}A_Q^{-1}\r\|
\lf\|A_{Q}\vec{f}(y)-A_{Q}\vec{f}(z)\r\|\mathbf{1}_Q(y)\\
&\quad\ls L_{s,b}\lf(d(y,z)\r)\delta^{[k_0-j]\wz{d}_{(p)}/p}
\sum_{Q\in\cq_j}\lf[\lf\|A_{Q}\vec{g}_j(y)\r\|+\lf\|A_Q\vec{g}_j(z)\r\|\r]\mathbf{1}_Q(y)\\
&\quad\ls L_{s,b}\lf(d(y,z)\r)\lf[h_j(y)+h_j(z)\r],
\end{align*}
where $\wz{d}_{(p)}$ is defined in \eqref{eq-def-dp}.
This implies that there exists a positive constant $C$ such that
$\{Ch_j\}_{j\in\zz}\in\dd^{s,b}(u_\eta)$.
This finishes the proof of Lemma \ref{lem-h,hj in D}.
\end{proof}

\begin{remark}\label{remark-independence}
The proof of Lemma \ref{lem-h,hj in D} implies that
the Haj{\l}asz gradient sequence $\{h_j\}_{j\in\zz}$ in \eqref{eq-def-gj} is independent of the choice of the vector $\vec{\eta}$.
\end{remark}

For any $k\in\zz$ and $Q\in\cq_k$, define $Q^*$ by setting
\begin{equation}\label{eq-def-Q*}
Q^*:=\bigcup\lf\{P\in\cq:\ \ell(P)=\ell(Q),\ P\cap\lf[\bigcup_{y\in Q}B(y,\delta^{k-2})\r]\neq\emptyset\r\}.
\end{equation}
Obviously, $Q^*$ is a union of some pairwise disjoint dyadic cubes and,
for any  $k\in\zz$, $Q\in\cq_k$, and $x\in Q$,
\begin{equation}\label{eq-B subset Q*}
B\left(x,\delta^{k-2}\right)\subset Q^*.
\end{equation}
We also claim that there exists a uniform positive constant $C$ such that,
for any $x\in\cx$, $k\in\zz$, and $Q\in\cq_k$ containing $x$,
\begin{equation}\label{eq-Q*<B}
\mu\lf(Q^*\r)\le C\mu\lf(B\lf(x,\delta^{k-2}\r)\r).
\end{equation}
Indeed, observe that, for any $z\in Q^*$, there exists a dyadic cube $P\in\cq_k$ containing $z$ and a point $y\in Q$
such that $P\cap B(y,\delta^{k-2})\neq\emptyset$. We suppose $z_P\in P\cap B(y,\delta^{k-2})$. Then
\begin{align*}
d(x,z)&\le A_0^2\lf[d(x,y)+d(y,z_P)+d(z_P,z)\r]\\
&\le A_0^2\lf[\ell(Q)+\delta^{k-2}+\ell(P)\r]\ls\delta^{k-2},
\end{align*}
which further implies $Q^*\subset B(x,\sigma\delta^{k-2})$ for some uniform $\sigma\in(1,\fz)$.
Thus, \eqref{eq-Q*<B} follows from this and the doubling property of $\mu$. This proves the above claim.

Let $W\in A_p(\cx,\cc^m)$ and $\{A_Q\}_{Q\in\cq}\in\mathcal{RS}_\cq^{W,p}$.
For any given $k\in\zz$, define $\gamma_k$
with respect to both $W$ and $\{A_Q\}_{Q\in\cq}$ by setting, for any $x\in\cx$,
\begin{equation}\label{eq-def-rj}
\gamma_k(x):=\sum_{Q\in\cq_k}\lf\|W^{\f1p}(x)A_Q^{-1}\r\|\mathbf{1}_Q(x).
\end{equation}
For any given $k\in\zz$, define the averaging operator $E_k$
at level $k$ by setting, for any $f\in L_\loc^1(\cx)$,
\begin{equation}\label{eq-def-Ek}
E_k(f):=\sum_{Q\in\cq_k}\lf[\fint_{Q^*}f(y)\,d\mu(y)\r]\mathbf{1}_Q.
\end{equation}

We establish an $L^p(\cx)$-estimate with respect to both $\gamma_k$ and $E_k$ as follows
(which can be regarded as a variant of part of the Nazarov theorem; see, for instance, \cite[Theorem 3.7]{FR21}).

\begin{lemma}\label{lem-rkEkf}
Let $\lambda\in(0,\fz)$, $p\in(1/\lambda,\fz)$,
$W\in A_p(\cx,\cc^m)$, $\{A_Q\}_{Q\in\cq}\in\mathcal{RS}_\cq^{W,p}$,
$\{\gamma_k\}_{k\in\zz}$ be defined in \eqref{eq-def-rj},
and $\{E_k\}_{k\in\zz}$  defined in \eqref{eq-def-Ek}.
Then there exists a positive constant $C_p$, depending on $p$, such that,
for any locally integrable function  $f$ and any $k\in\zz$,
\begin{equation*}
\lf\|\sup_{k\in\zz}\gamma_k \lf[E_k(f)\r]^\lambda\r\|_{L^p(\cx)}
\le C_p\lf\|f^\lambda\r\|_{L^p(\cx)}.
\end{equation*}
\end{lemma}

\begin{proof}
Notice that, for any $f\in L^1_\loc(\cx)$ and $k\in\zz$,
$E_k(f)$ is constant and nonnegative on each dyadic cube $Q\in\cq_k$.
For simplicity and without loss of generality, we may define
$E_k(f):=\sum_{Q\in\cq_k}c_Q\mathbf{1}_Q$ with $c_Q\ge 0$ for any $Q\in\cq_k$.

For any sequence $\alpha:=\{\alpha_j\}_{j\in\zz}$ of nonnegative measurable functions,
define
\begin{align*}
\lf\|\alpha\r\|_{\mathcal{C}}:=
\sup_{Q\in\cq}\fint_Q\sup_{j\ge j_Q}\alpha_j(x)\,d\mu(x),
\end{align*}
where $j_Q:=\log_\delta\ell(Q)$.
Then, by both \eqref{lem-eq-W,AP',p<1} for any $p\in(0,1]$ and \eqref{lem-eq-W,AP',p>1} for any $p\in(1,\fz)$
with $r:=p$, we have
\begin{align*}
\lf\|\gamma^p\r\|_{\mathcal{C}}=
\sup_{P\in\cq}\fint_P\sup_{j\ge j_P}\sum_{Q\in\cq_j}\lf\|W^{\f1p}(x)A_Q^{-1}\r\|^p\mathbf{1}_Q(x)\,d\mu(x)
\le C_p,
\end{align*}
where $C_p$ is as Lemma \ref{lem-AQW} with $r$ replaced by $p$.
Thus, applying this and \cite[Lemma 3.6]{FR21} (which can be naturally generalized to doubling measures),
we conclude that, for any $k\in\zz$,
\begin{align*}
\lf\|\sup_{k\in\zz}\gamma_k \lf[E_k(f)\r]^\lambda\r\|_{L^p(\cx)}
&=\lf\|\sup_{k\in\zz}\gamma_k^p \lf[E_k(f)\r]^{\lambda p}\r\|_{L^1(\cx)}^{\frac 1 p}\\
&\le\lf\|\gamma^p\r\|_{\mathcal{C}}^{\frac 1 p}
\lf\|\sup_{k\in\zz}\lf[E_k(f)\r]^{\lambda p}\r\|_{L^1(\cx)}^{\frac 1 p}
\ls\lf\|\sup_{k\in\zz}\lf[E_k(f)\r]^{\lambda}\r\|_{L^p(\cx)}.
\end{align*}
By \eqref{eq-def-Ek}, we also find that, for any $k\in\zz$ and $x\in\cx$,
$$
E_k(f)(x)=\fint_{Q^*}f(y)\,d\mu(y)\le\cm(f)(x),
$$
where $Q$ is the dyadic cube in $\cq_k$ that contains $x$ and $Q^*$ is defined in \eqref{eq-def-Q*}.
Thus, from this, the $L^{\lambda p}(\cx)$-boundedness
of the Hardy--Littlewood maximal operator (see, for instance, \cite{S05,GLY09}), we deduce that,
for any $k\in\zz$,
\begin{align*}
\lf\|\sup_{k\in\zz}\gamma_k \lf[E_k(f)\r]^\lambda\r\|_{L^p(\cx)}
\ls\lf\|\cm(f)^\lambda\r\|_{L^{p}(\cx)}
\sim\lf\|\cm(f)\r\|_{L^{\lambda p}(\cx)}^\lambda\ls\lf\|f\r\|_{L^{\lambda p}(\cx)}^\lambda
\sim\lf\|f^\lambda\r\|_{L^p(\cx)},
\end{align*}
which completes the proof of Lemma \ref{lem-rkEkf}.
\end{proof}

Next, we arrive at the  matrix-weighted Poincar\'{e}-type inequalities.

\begin{theorem}[Poincar\'{e}-Type Inequalities]\label{lem-Poincare}
Let $s\in(0,\fz)$, $p\in(0,\fz)$, and $W\in A_p(\cx,\cc^m)$.
\begin{enumerate}
\item[\rm(i)] If $p\in(1,\fz)$, then
there exists a positive constant $C$ and  $\varepsilon\in(0,p-1)$ such that,
for any $\lambda\in[p-\varepsilon,p)$, $x\in\cx$, $k\in\zz$,
$\vec{f}\in [L^0(B(x,\delta^{k-2}))]^m$, $\{A_Q\}_{Q\in\cq}\in\mathcal{RS}_\cq^{W,p}$,
and $\{\vec{g}_j\}_{j\in\zz}\in\dd^{s,b}_{\{A_Q\}}(\vec{f})$,
\begin{align*}
&\inf_{\vec{\eta}\in\cc^m}\fint_{B(x,\delta^{k})}\lf\|W^{\f1p}(x)
\vec{f}(y)-\vec{\eta}\,\r\|\,d\mu(y)\\
&\quad\le C  L_{s,b}(\delta^{k})\gamma_k(x)\sum_{j=k-3}^{k-1}
\lf[E_k\lf(\lf\|W^{\frac 1 p}\vec{g}_j\r\|^{\lambda}\r)(x)\r]^{\frac{1}{\lambda}}.
\end{align*}

\item[\rm(ii)] If $p\in(0,1]$, $\lambda\in(0,1]$, and $A_0\delta^{\lambda/D}<1$, then, for any given $\varepsilon,\varepsilon'\in(0,\min\{1,s\})$ with $\varepsilon<\varepsilon'$,
there exists a positive constant $C$
such that, for any $x\in\cx$, $k\in\zz$,
 $\vec{f}\in[L^0(B(x,\delta^{k-1}))]^m$, $\{A_Q\}_{Q\in\cq}\in\mathcal{RS}_\cq^{W,p}$,
and $\{g_j\}_{j\in\zz}\in\dd^{s,b}_{\{A_Q\}}(\vec{f})$,
\begin{align*}
&\inf_{\vec{\eta}\in\cc^m}\lf[\fint_{B(x,\delta^{k})}\lf\|W^{\f1p}(x)\vec{f}(y)-\vec{\eta}\,\r\|
^{\lambda^*}\,d\mu(y)\r]^{\f{1}{\lambda^*}}\\
&\quad\le C\delta^{k\vp'}\gamma_k(x)\sum_{j=k-2}^{\fz}\delta^{-j\vp'}L_{s,b}(\delta^{j})
\lf[E_k\lf(\lf\|W^{\frac 1 p}\vec{g}_j\r\|^{\lambda}\r)\r]^{\frac{1}{\lambda}},
\end{align*}
where $\lambda^*:=D\lambda/(D-\varepsilon\lambda)$.
\end{enumerate}
\end{theorem}

To prove Theorem \ref{lem-Poincare}, we first need
 the following analogue in averaging version
with respect to a given sequence of reducing operators.

\begin{lemma}\label{lem-Poin-p>1}
Let $s\in(0,\fz)$, $p\in(0,\fz)$, $W\in A_p(\cx,\cc^m)$, and
$\{A_Q\}_{Q\in\cq}\in\mathcal{RS}_\cq^{W,p}$. Then
there exists a positive constant $C$ such that,
for any $x\in\cx$, $k\in\zz$, $Q\in\cq_k$ containing $x$,
$\vec{f}\in [L^0(B(x,\delta^{k-2}))]^m$,
and $\{\vec{g}_j\}_{j\in\zz}\in\dd^{s}_{\{A_Q\}}(\vec{f})$,
\begin{align}\label{eq-Poin-p>1}
\inf_{\vec{\eta}\in\cc^m}\fint_{B(x,\delta^{k})}\lf\|A_{Q}
\vec{f}(y)-\vec{\eta}\,\r\|\,d\mu(y)
\le CL_{s,b}(\delta^{k})\sum_{j=k-3}^{k-1} \fint_{B(x,\delta^{k-2})}\lf\|A_Q\vec{g}_j(y)\r\|\,d\mu(y).
\end{align}
\end{lemma}

\begin{proof}
For any $y\in B(x,\delta^{k})$ and $z\in B(x,\delta^{k-2})\setminus B(x,A_0\delta^{k-1})$,
we have
$$
d(y,z)\le A_0\left[d(y,x)+d(x,z)\right]<2A_0\delta^{k-2}.
$$
As $\delta<(2A_0)^{-1}$, we further obtain
$d(y,z)\in[\delta^{k},\delta^{k-3})$.
From this, Lemma \ref{lem-AQ=AQ'}, the definition of $\dd^{s,b}_{\{A_Q\}}(\vec{f})$,
and the doubling property of $\mu$,
we infer that, for any $x\in\cx$ and $k\in\zz$,
\begin{align*}
&\inf_{\vec{\eta}\in\cc^m}\fint_{B(x,\delta^{k})}\lf\|A_{Q}
\vec{f}(y)-\vec{\eta}\,\r\|\,d\mu(y)\\
&\quad\le\fint_{B(x,\delta^{k})}\lf\|A_{Q}\lf[
\vec{f}(y)-\vec{f}_{B(x,\delta^{k-2})\setminus B(x,A_0\delta^{k-1})}\r]\,\r\|\,d\mu(y)\\
&\quad\le\fint_{B(x,\delta^{k})}\fint_{B(x,\delta^{k-2})\setminus B(x,A_0\delta^{k-1})}
\lf\|A_{Q}\vec{f}(y)-A_{Q}\vec{f}(z)\,\r\|\,d\mu(z)d\mu(y)\\
&\quad\ls L_{s,b}(\delta^{k})\sum_{j=k-3}^{k-1}\fint_{B(x,\delta^{k})}\fint_{B(x,\delta^{k-2})\setminus B(x,A_0\delta^{k-1})}
\left[\lf\|A_Q\vec{g}_j(y)\r\|+\lf\|A_Q\vec{g}_j(z)\r\|\right]\,d\mu(z)d\mu(y)\\
&\quad\ls L_{s,b}(\delta^{k})\sum_{j=k-3}^{k-1} \fint_{B(x,\delta^{k-2})}\lf\|A_Q\vec{g}_j(y)\r\|\,d\mu(y).
\end{align*}
This finishes the proof of Lemma \ref{lem-Poin-p>1}.
\end{proof}

Now, we prove Theorem  \ref{lem-Poincare}(i).

\begin{proof}[Proof of Theorem \ref{lem-Poincare}(i)]
Let $\gamma_k$ be the same as in \eqref{eq-def-rj}
and $Q_k^x$ denote the dyadic cube in $\cq_k$ that contains $x$.
Observe that, for  any $x\in\cx$, $k\in\zz$,
$\vec{f}\in[L^0(B(x,\delta^{k-1}))]^m$,
and $\{g_j\}_{j\in\zz}\in\dd^{s,b}_{\{A_Q\}}(\vec{f})$,
\begin{align}\label{eq-Poin-W-AQ}
&\inf_{\vec{\eta}\in\cc^m}\fint_{B(x,\delta^{k})}\lf\|W^{\f1p}(x)
\vec{f}(y)-\vec{\eta}\,\r\|\,d\mu(y)\nonumber\\
&\quad\le\lf\|W^{\f1p}(x)A_{Q_k^x}^{-1}\r\|\inf_{\vec{\xi}\in\cc^m}\fint_{B(x,\delta^{k})}
\lf\|A_{Q_k^x}\vec{f}(y)-\vec{\xi}\,\r\|\,d\mu(y)\nonumber\\
&\quad=\gamma_k(x)\inf_{\vec{\eta}\in\cc^m}\fint_{B(x,\delta^{k})}\lf\|A_{Q_k^x}
\vec{f}(y)-\vec{\eta}\,\r\|\,d\mu(y).
\end{align}
We apply \eqref{eq-Poin-p>1} to \eqref{eq-Poin-W-AQ} and obtain
\begin{align}\label{eq-Poin-pf,p>1}
&\inf_{\vec{\eta}\in\cc^m}\fint_{B(x,\delta^{k})}\lf\|W^{\f1p}(x)
\vec{f}(y)-\vec{\eta}\,\r\|\,d\mu(y)\nonumber\\
&\quad\ls L_{s,b}(\delta^{k})\gamma_k(x)\sum_{j=k-3}^{k-1}
\fint_{B(x,\delta^{k-2})}\lf\|A_{Q_k^x}\vec{g}_j(y)\r\|\,d\mu(y).
\end{align}

By the property of the system of dyadic cubes and the doubling property of $\mu$,
we find that, for any $k\in\zz$ and $x\in\cx$,
there exists a uniform natural number $M$ and at most $M$ dyadic cubes in $\cq_k$ that
cover $B(x,\delta^{k-2})$;
for simplicity, we denote by $\mathcal{R}_{k,x}$ the collection of all these dyadic cubes.
(Here, we assume that, for any $R\in\mathcal{R}_{k,x}$, $R\cap B(x,\delta^{k-2})\neq\emptyset$.)
Obviously,  we have, for any $k\in\zz$, $x\in\cx$, and $Q\in\cq_k$ containing $x$,
$$\bigcup_{R\in\mathcal{R}_{k,x}}R \subset Q^*.$$
Let $\epsilon$ be the same as in Lemma \ref{lem-AQW}, $t\in(p',p'+\epsilon)$, and $\lambda\in[t/(t-1),p)$.
From these, the  H\"{o}lder inequality, Lemma \ref{lem-AQ=AQ'}, \eqref{eq-B subset Q*},
\eqref{eq-Q*<B}, \eqref{eq-def-Ek}, and \eqref{lem-eq-AP,W,p>1}, we deduce that, for any $k\in\zz$ and $x\in\cx$,
\begin{align*}
&\fint_{B(x,\delta^{k-2})}\lf\|A_{Q_k^x}\vec{g}_j(y)\r\|\,d\mu(y)\\
&\quad\le\lf\{\fint_{B(x,\delta^{k-2})}\lf\|A_{Q_k^x}W^{-\frac 1 p}(y)\r\|^{t}\,d\mu(y)\r\}^{\frac{1}{t}}
\lf\{\fint_{B(x,\delta^{k-2})}\lf\|W^{\frac 1 p}(y)\vec{g}_j(y)\r\|^{\frac{t}{t-1}}\,d\mu(y)\r\}^{\frac{t-1}{t}}\\
&\quad\ls\lf\{\sum_{R\in\mathcal{R}_{k,x}}\lf[\fint_{R}\lf\|A_{R}W^{-\frac 1 p}(y)\r\|^{t}\,d\mu(y)\r]^{\frac{1}{t}}\r\}
\sum_{Q\in\cq_k}\lf\{\fint_{Q^*}\lf\|W^{\frac 1 p}(y)\vec{g}_j(y)\r\|^{\frac{t}{t-1}}\,d\mu(y)\r\}^{\frac{t-1}{t}}
\mathbf{1}_Q(x)\\
&\quad\ls M\sup_{Q\in\cq}\lf\{\fint_{Q}\lf\|A_{Q}W^{-\frac 1 p}(y)\r\|^{t}\,d\mu(y)\r\}^{\frac{1}{t}}
\lf[E_k\lf(\lf\|W^{\frac 1 p}\vec{g}_j\r\|^{\lambda}\r)(x)\r]^{\frac{1}{\lambda}}\\
&\quad\ls\lf[E_k\lf(\lf\|W^{\frac 1 p}\vec{g}_j\r\|^{\lambda}\r)(x)\r]^{\frac{1}{\lambda}},
\end{align*}
which, together with \eqref{eq-Poin-pf,p>1}, further implies that
\begin{align*}
&\inf_{\vec{\eta}\in\cc^m}\fint_{B(x,\delta^{k})}\lf\|W^{\f1p}(x)
\vec{f}(y)-\vec{\eta}\,\r\|\,d\mu(y)\\
&\quad\ls L_{s,b}(\delta^{k})\gamma_k(x)\sum_{j=k-3}^{k-1}
 \lf[E_k\lf(\lf\|W^{\frac 1 p}\vec{g}_j\r\|^{\lambda}\r)(x)\r]^{\frac{1}{\lambda}}.
\end{align*}
This finishes the proof of Theorem  \ref{lem-Poincare}(i).
\end{proof}

To show Theorem  \ref{lem-Poincare}(ii),
we need a series of auxiliary lemmas as follows.
\begin{lemma}\label{lem-Poin-s0}
Let $s\in(0,\fz)$,  $\lambda\in(0,D/s)$,
$\lambda^*:=(D\lambda)/(D-s\lambda)$,
$p\in(0,1]$, $W\in A_p(\cx,\cc^m)$, and
$\{A_Q\}_{Q\in\cq}\in\mathcal{RS}_\cq^{W,p}$.
If $A_0\delta^{\lambda/D}<1$, then
there exists a positive constant $C$
such that, for any $x_0\in\cx$, $k\in\zz$, $B_k:=B(x_0,\delta^{k})$,
$Q\in \cq_{k}$ containing $x_0$,
$\vec{f}\in[L^0(\delta^{-1}B_k)]^m$,
and $\{\vec{g}_j\}_{j\in\zz}\in\dd^{s,0}_{\{A_Q\}}(\vec{f})$,
\begin{align}\label{eq-lem-Poin-s0}
&\lf\{\inf_{\vec{\eta}\in\cc^m}\fint_{B_0}\lf\|A_{Q}
\vec{f}(y)-\vec{\eta}\,\r\|^{\lambda^*}\,d\mu(y)\r\}^{\frac{1}{\lambda^*}}\nonumber\\
&\quad\le C\delta^{ks}\lf\{\fint_{\delta^{-1}B_k}\sup_{j\ge k-3}\dsum_{P\in\cq_j}
\lf\|A_{P}\vec{g}_j(y)\r\|^{\lambda} \mathbf{1}_P(y)\,d\mu(y)\r\}^{\frac{1}{\lambda}}.
\end{align}
\end{lemma}

\begin{proof}
As $p\in(0,1]$, we have $\wz{d}_{(p)}=0$.
Let $x_0\in\cx$, $k\in\zz$, $B_k:=B(x_0,\delta^{k})$, and $Q\in \cq_{k}$ containing $x_0$.
For any given $\eta\in\cc^m$,
define the auxiliary function $u_\eta$ as in \eqref{eq-def-u}.
For any given $\{\vec{g}_j\}_{j\in\zz}\in\dd^{s}_{\{A_Q\}}(\vec{f})$,
define the function $g$ by setting, for any $y\in\delta^{-1}B_0$,
\begin{equation*}
g(y):=\sup_{j\ge k-3}\dsum_{P\in\cq_j}
\lf\|A_{P}\vec{g}_j(y)\r\|\mathbf{1}_P(y).
\end{equation*}
From $\wz{d}_{(p)}=0$ and the proof of Lemma \ref{lem-h,hj in D},
we infer that there exists a positive constant $C$ such that
$Cg$  is an $s$-Haj{\l}asz gradient of $u_\eta$.

Define a sequence of auxiliary sets, $\{G_j\}_{j\in\zz}$, by setting, for any $j\in\zz$,
$$
G_j:=\lf\{x\in\delta^{-1}B_0:\ g(x)\le \delta^{-j}\r\}.
$$
We choose a suitable $\eta\in\cc^m$ such that $\essinf_{G_{k_0}} u_\eta=0$
for some special $k_0$,
where the choice of $k_0$ is the same as in the proof of \cite[Lemma 3.11]{AWYY23}
(see also \cite[Theorem 8.7]{Haj03} for an original proof on $\rn$)
and depends only on $C_\mu,\ D,\ A_0,\ \delta,\ s,\ B$, and $g$.
Thus, following an argument similar to that used in the proof of
\cite[Lemma 3.11]{AWYY23} with $p$ therein replaced by $\lambda$, we obtain
\begin{equation}\label{eq-poin-cla}
\lf\{\fint_{B_0}|u_\eta(y)|^{\lambda^*}\,d\mu(y)\r\}^{\frac{1}{\lambda^*}}
\ls \delta^{ks}\lf\{\fint_{\delta^{-1}B_0}\lf[g(y)\r]^{\lambda}\,d\mu(y)\r\}^{\frac{1}{\lambda}},
\end{equation}
which, combined with the definitions of both $u_\eta$ and $g$ and also with
the arbitrariness of $\eta$, further implies \eqref{eq-lem-Poin-s0}.
This finishes the proof of Lemma \ref{lem-Poin-s0}.
\end{proof}

\begin{lemma}\label{lem-Poin-dd,s,b0}
Let $p\in(0,1]$, $W\in A_p(\cx,\cc^m)$, $\{A_Q\}_{Q\in\cq}\in\mathcal{RS}_\cq^{W,p}$,
$s\in(0,\fz)$, $b:=(b_+,b_-)$ be a pair of real numbers, and $\lambda\in(0,1]$.
If $A_0\delta^{\lambda/D}<1$, then, for any given $\varepsilon,\varepsilon'\in(0,\min\{1,s\})$
with $\varepsilon<\varepsilon'$ and $\lambda^*:=(D\lambda)/(D-\varepsilon\lambda)$,
there exists a positive constant $C$
such that, for any $x_0\in\cx$, $k\in\zz$, $B_k:=B(x_0,\delta^{k})$,
$Q\in \cq_{k}$ containing $x_0$,
$\vec{f}\in[L^0(\delta^{-1}B_k)]^m$,
and $\{\vec{g}_j\}_{j\in\zz}\in\dd^{s,b}_{\{A_Q\}}(\vec{f})$,
\begin{align}\label{eq-lem-poin-dd-s,b0}
&\inf_{\vec{\eta}\in\cc^m}\lf\{\fint_{B_k}
\lf\|A_{Q}\vec{f}(y)-\vec{\eta}\,\r\|^{\lambda^*}\,d\mu(y)\r\}^{\frac{1}{\lambda^*}}\nonumber\\
&\quad\le C\delta^{k\vp'}\sum_{j=k-2}^{\fz}\delta^{-j\vp'}L_{s,b}(\delta^{j})
\lf\{\fint_{\delta^{-1}B_k}\sum_{P\in\cq_j}\lf\|A_{P}\vec{g}_j(y)\r\|^{\lambda}
\mathbf{1}_P(y)\,d\mu(y)\r\}^{\frac{1}{\lambda}}.
\end{align}
\end{lemma}

\begin{proof}
Without loss of generality, we may assume that
the right-hand side of \eqref{eq-lem-poin-dd-s,b0} is finite.
Let $x_0\in\cx$, $k\in\zz$, $B_k:=B(x_0,\delta^{k})$, and $Q\in \cq_{k}$ contain $x$.

For an arbitrary $\vec{\eta}\in\cc^m$, define $u_\eta$ as in \eqref{eq-def-u}
and, for any $\varepsilon\in(0,s)$, define $h$ by setting, for any $y\in\delta^{-1}B_0$,
\begin{equation*}
h(y):=\lf\{\sum_{j=k-2}^{\fz}\delta^{-j\varepsilon \lambda}\lf[L_{s,b}(\delta^j)\r]^{\lambda}
\sum_{P\in\cq_j}\lf\|A_P\vec{g}_j\r\|^{\lambda}\mathbf{1}_P(y)\r\}^{\frac {1}{\lambda}}.
\end{equation*}
As $p\in(0,1]$, we have $\wz{d}_{(p)}=0$.
Thus, by an argument similar to that used in the proof of Lemma \ref{lem-h,hj in D},
we find that $h$
is a positive constant multiple of an $\varepsilon$-Haj{\l}asz gradient of $u_\eta$.
Furthermore,  applying the independence of $\vec{\eta}$ of $h$, \eqref{eq-poin-cla},
and the H\"{o}lder inequality,
we conclude that, for any $\varepsilon'\in(\varepsilon,s)$,
\begin{align*}
&\inf_{\vec{\eta}\in\cc^m}\lf\{\fint_{B_0}\lf\|A_{Q}
\vec{f}(y)-\vec{\eta}\,\r\|^{\lambda^*}\,d\mu(y)\r\}^{\frac{1}{\lambda^*}}\\
&\quad\ls\delta^{k\varepsilon}\lf\{\fint_{\delta^{-1}B_0}
\lf[h(y)\r]^\lambda\,d\mu(y)\r\}^{\frac{1}{\lambda}}\\
&\quad\ls\delta^{k\varepsilon}\lf\{\fint_{\delta^{-1}B_0}
\sum_{j=k-2}^{\fz}\delta^{-j\varepsilon \lambda}\lf[L_{s,b}(\delta^j)\r]^{\lambda}
\sum_{P\in\cq_j}\lf\|A_P\vec{g}_j(y)\r\|^{\lambda}\mathbf{1}_P(y)\,d\mu(y)\r\}^{\frac{1}{\lambda}}\\
&\quad\ls\delta^{k\vp'}\sum_{j=k-2}^{\fz}\delta^{-j\vp'}L_{s,b}(\delta^{j})
\lf\{\fint_{\delta^{-1}B_0}\sum_{P\in\cq_j}\lf\|A_{P}\vec{g}_j(y)\r\|^{\lambda}
\mathbf{1}_P(y)\,d\mu(y)\r\}^{\frac{1}{\lambda}}.
\end{align*}
This finishes the proof of Lemma \ref{lem-Poin-dd,s,b0}.
\end{proof}

Next, we prove Theorem  \ref{lem-Poincare}(ii).

\begin{proof}[Proof of Theorem  \ref{lem-Poincare}(ii)]
Applying \eqref{eq-lem-poin-dd-s,b0},  \eqref{lem-eq-AP,W,p<1},
\eqref{eq-B subset Q*}, and \eqref{eq-Q*<B}, we find that, for any $x\in\cx$ and $k\in\zz$,
\begin{align*}
&\inf_{\vec{\eta}\in\cc^m}\lf\{\fint_{B(x,\delta^{k})}
\lf\|W^{\f1p}(x)\vec{f}(y)-\vec{\eta}\,\r\|^{\lambda^*}\,d\mu(y)\r\}^{\frac{1}{\lambda^*}}\\
&\quad\le\sum_{Q\in\cq_k}\lf\|W^{\f1p}(x)A_{Q}^{-1}\r\|\inf_{\vec{\xi}\in\cc^m}\lf\{\fint_{B(x,\delta^{k})}
\lf\|A_{Q}\vec{f}(y)-\vec{\xi}\,\r\|^{\lambda^*}\,d\mu(y)\r\}^{\frac{1}{\lambda^*}}\mathbf{1}_Q(x)\nonumber\\
&\quad\ls\delta^{k\vp'}\gamma_k(x)\sum_{j=k-2}^{\fz}\delta^{-j\vp'}L_{s,b}(\delta^{j})
\lf\{\fint_{B(x,\delta^{k-1})}\sum_{P\in\cq_j}\lf\|A_{P}\vec{g}_j(y)\r\|^{\lambda}
\mathbf{1}_P(y)\,d\mu(y)\r\}^{\frac{1}{\lambda}}\nonumber\\
&\quad\ls\delta^{k\vp'}\gamma_k(x)\sum_{j=k-2}^{\fz}\delta^{-j\vp'}L_{s,b}(\delta^{j})
\sup_{P\in\cq}\esssup_{y\in P}\lf\|A_{P}W^{-\frac 1 p}(y)\r\|\nonumber\\
&\quad\quad\times\lf\{\fint_{B(x,\delta^{k-1})}
\lf\|W^{\frac 1 p}(y)\vec{g}_j(y)\r\|^{\lambda}
\,d\mu(y)\r\}^{\frac{1}{\lambda}}\nonumber\\
&\quad\ls\delta^{k\vp'}\gamma_k(x)\sum_{j=k-2}^{\fz}\delta^{-j\vp'}L_{s,b}(\delta^{j})
\sum_{Q\in\cq_k}\lf\{\fint_{Q^*}
\lf\|W^{\frac 1 p}(y)\vec{g}_j(y)\r\|^{\lambda}
\,d\mu(y)\r\}^{\frac{1}{\lambda}}\mathbf{1}_Q(x)\nonumber\\
&\quad\ls\delta^{k\vp'}\gamma_k(x)\sum_{j=k-2}^{\fz}\delta^{-j\vp'}L_{s,b}(\delta^{j})
\lf[E_k\lf(\lf\|W^{\frac 1 p}\vec{g}_j\r\|^{\lambda}\r)\r]^{\frac{1}{\lambda}},\nonumber
\end{align*}
which completes the proof of Theorem \ref{lem-Poincare}(ii) and hence Theorem \ref{lem-Poincare}.
\end{proof}

Finally, we introduce the following homogeneous matrix-weighted logarithmic Haj{\l}asz--Besov spaces.

\begin{definition}\label{def-Haj-space}
Let $s\in(0,\fz)$, $b:=(b_+,b_-)$ be a pair of real numbers, $p\in(0,\fz)$, $q\in(0,\fz]$,
$W\in A_p(\cx,\cc^m)$, and
$\{A_Q\}_{Q\in\cq}\in\mathcal{RS}_\cq^{W,p}$.
The \emph{homogeneous matrix-weighted logarithmic Haj{\l}asz--Besov space},
$\dot{N}^{s,b}_{p,q}(\cx,W)$,
is defined as the collection of all $\vec{f}\in [L^0(\cx)]^m$ such that
\begin{equation*}
\lf\|\vec{f}\,\r\|_{\dot{N}^{s,b}_{p,q}(\cx,W)}:=
\inf_{\vec{\mathbf{g}}\in\dd^{s,b}_{\{A_Q\}}(\vec{f})}
\lf\|\lf\{\vec{g}_k\r\}_{k\in\zz}\r\|_{\ell^q(W,L^p)}<\fz
\end{equation*}
with the usual modification made for $q=\fz$.
\end{definition}

\begin{remark}
In the next section, we  prove that $\dot{N}^{s,b}_{p,q}(\cx,W)$ coincides with
the matrix-weighted logarithmic  Besov space (see Theorem \ref{haj-c}),
which further implies that
$\dot{N}^{s,b}_{p,q}(\cx,W)$ is independent of the choice of $\{A_Q\}_{Q\in\cq}\in\mathcal{RS}_\cq^{W,p}$.
\end{remark}

The above Poincar\'{e}-type inequalities and their proofs also imply that
functions in $\dot{N}^{s,b}_{p,q}(\cx,W)$ for any
$p\in(D/[D+s],\fz)$, $q\in(0,\fz]$, and $W\in A_p(\cx,\cc^m)$
are all locally integrable.

\begin{corollary}\label{rm-locinte}
Let $s\in(0,\fz)$, $b:=(b_+,b_-)$ be a pair of real numbers,
$p\in(D/[D+s],\fz)$, $q\in(0,\fz]$, and $W\in A_p(\cx,\cc^m)$.
Then, for any vector-valued function
$\vec{f}\in \dot{N}^{s,b}_{p,q}(\cx,W)$,
$\vec{f}\in[L_\loc^1(\cx)]^m$.
\end{corollary}

\begin{proof}
Let $\{A_Q\}_{Q\in\cq}\in\mathcal{RS}_\cq^{W,p}$,
$\vec{f}\in \dot{N}^{s,b}_{p,q}(\cx,W)$, and $\{\vec{g}_j\}_{j\in\zz}\in\dd^{s,b}_{\{A_Q\}}(\vec{f})$
be such that $$\|\{\vec{g}_j\}_{j\in\zz}\|_{\ell^q(W,L^p)}\le2\|\vec{f}\,\|_{\dot{N}^{s,b}_{p,q}(\cx,W)}.$$
We consider the following two cases on $p$.

When $p\in(1,\fz)$, by \eqref{eq-Poin-p>1}, the  H\"{o}lder inequality, and \eqref{lem-eq-AP,W,p>1},
we find that, for any $x\in\cx$, $k\in\zz$, and $Q\in\cq_k$ containing $x$,
\begin{align*}
&\inf_{\vec{\eta}\in\cc^m}\fint_{B(x,\delta^{k})}\lf\|A_{Q}
\vec{f}(y)-\vec{\eta}\,\r\|\,d\mu(y)\\
&\quad\ls L_{s,b}(\delta^{k})\sum_{j=k-3}^{k-1}
\lf\{\fint_{B(x,\delta^{k-2})}\lf\|W^{\frac{1}{p}}(y)\vec{g}_j(y)\r\|^p\,d\mu(y)\r\}^{\frac 1 p}\\
&\quad\quad\times \sup_{Q\in\cq}
\lf\{\fint_{Q}\lf\|A_{Q}W^{-\frac 1 p}(y)\r\|^{p'}\,d\mu(y)\r\}^{\frac{1}{p'}}\\
&\quad\ls L_{s,b}(\delta^{k})\sum_{j=k-3}^{k-1}
\lf\{\fint_{B(x,\delta^{k-2})}\lf\|W^{\frac{1}{p}}(y)\vec{g}_j(y)\r\|^p\,d\mu(y)\r\}^{\frac 1 p}.
\end{align*}

When $p\in(D/(D+s),1]$, we choose $\varepsilon\in(0,s)$ such that $p>D/(D+\varepsilon)$.
Then $p^*:=Dp/(D-\varepsilon p)\in(1,\fz)$.
Since $A_0\delta^{p/D}<1$,
then, from these, the  H\"{o}lder inequality, \eqref{eq-lem-poin-dd-s,b0}, and \eqref{lem-eq-AP,W,p<1},
it follows that, for any $x\in\cx$, $k\in\zz$, and $Q\in\cq_k$ containing $x$,
\begin{align*}
&\inf_{\vec{\eta}\in\cc^m}\fint_{B(x,\delta^{k})}\lf\|A_{Q}
\vec{f}(y)-\vec{\eta}\,\r\|\,d\mu(y)\\
&\quad\le\inf_{\vec{\eta}\in\cc^m}\lf\{\fint_{B(x,\delta^{k})}\lf\|A_{Q}
\vec{f}(y)-\vec{\eta}\,\r\|^{p^*}\,d\mu(y)\r\}^{\frac 1 p}\\
&\quad\ls\delta^{k\vp'}\sum_{j=k-2}^{\fz}\delta^{-j\vp'}L_{s,b}(\delta^{j})
\lf\{\fint_{B(x,\delta^{k-1})}\lf\|W^{\frac 1 p}(y)\vec{g}_j(y)\r\|^{p}\,d\mu(y)\r\}^{\frac{1}{p}}\\
&\quad\quad\times\sup_{P\in\cq}\esssup_{y\in P}\lf\|A_{P}W^{-\frac 1 p}(y)\r\|\\
&\quad\ls\delta^{k\vp'}\sum_{j=k-2}^{\fz}\delta^{-j\vp'}L_{s,b}(\delta^{j})
\lf\{\fint_{B(x,\delta^{k-1})}\lf\|W^{\frac 1 p}(y)\vec{g}_j(y)\r\|^{p}\,d\mu(y)\r\}^{\frac{1}{p}}.
\end{align*}
Altogether,  by the doubling property of $\mu$, we further obtain,
for any $p\in(D/[D+s],\fz)$,
\begin{align*}
&\int_{B(x,\delta^{k})}\lf\|A_{Q}\vec{f}(y)\,\r\|\,d\mu(y)\\
&\quad\le \mu(B(x,\delta^{k}))\inf_{\vec{\eta}\in\cc^m}
\left\{\fint_{B(x,\delta^{k})}\lf\|A_{Q}
\vec{f}(y)-\vec{\eta}\,\r\|\,d\mu(y) +\lf\|\vec{\eta}\,\r\|\right\}\\
&\quad\ls\lf[\mu(B(x,\delta^{k}))\r]^{1-\frac{1}{p}} L_{s,b}(\delta^{k})
\lf\|\lf\{\vec{g}_k\r\}_{k\in\zz}\r\|_{\ell^q(W,L^p)}+\mu(B(x,\delta^{k}))\lf\|\vec{\eta}\,\r\|
<\fz.
\end{align*}

Notice that the total number of the cubes in the set
$U_{k}:=\{Q\in\cq_{k}:\ Q\cap B(x,\delta^{k})\neq\emptyset\}$
is finite and $\{A_Q^{-1}\}_{Q\in\cq}$ is a family of linear bounded operators.
Thus,  we conclude that, for any $x\in\cx$ and $k\in\zz$,
\begin{align*}
\int_{B(x,\delta^{k})}\lf\|\vec{f}(y)\r\|\,dy
&\le\int_{B(x,\delta^{k})}\sum_{Q\in U_k}
\lf\|A_Q^{-1}\r\|\lf\|A_Q\vec{f}(y)\r\|\,dy<\fz,
\end{align*}
which means that  $\vec{f}\in[L_\loc^1(\cx)]^m$.
This finishes the proof of Corollary \ref{rm-locinte}.
\end{proof}

\section{Characterization of Logarithmic Matrix-Weighted Besov Spaces
via Haj{\l}asz Gradient Sequences}\label{sec-chara}

In this section, we apply the  matrix-weighted Poincar\'{e}-type inequalities in the last section to
establish the Haj{\l}asz pointwise characterization of matrix-weighted logarithmic Besov spaces
on $\cx$.
We first introduce logarithmic matrix-weighted Besov spaces via using
approximations of the identity with exponential decay and then
establish their characterization
via grand maximal functions in Subsection \ref{sec-grand}.
Applying this characterization and the matrix-weighted Poincar\'e inequalities obtained in Section \ref{sec-poincare},
in Subsection \ref{sec-haj},
we further establish the Haj\l asz pointwise characterization of logarithmic matrix-weighted Besov spaces,
that is, we establish the equivalence between logarithmic matrix-weighted Besov spaces
and logarithmic matrix-weighted Haj\l asz--Besov spaces.

\subsection{Grand Maximal Function Characterization of \\ Logarithmic Matrix-Weighted Besov Spaces}\label{sec-grand}

To introduce the matrix-weighted logarithmic Besov spaces
on spaces of homogeneous type, we need to recall the concepts of
both test functions and distributions on $\cx$;
see \cite{HMY06,HMY08}.

\begin{definition}[Test Functions]\label{def-homo-test}
Let $x_1\in\cx$, $r\in(0,\fz)$, $\beta\in(0,1]$, and $\gamma\in(0,\fz)$.
Define $D_{\gamma}$ by setting, for any $x\in\cx$,
\begin{equation}\label{eq-def-D}
D_{\gamma}(x_1,x;r):=\f{1}{V_r(x_1)+V(x_1,x)}\lf[\f{r}{r+d(x_1,x)}\r]^\gamma.
\end{equation}
Then a measurable function $f$ on $\cx$ is called a \emph{test function of type $(x_1,r,\beta,\gamma)$} if
there exists a positive constant $C$ such that
\begin{enumerate}
\item[\rm(i)] for any $x\in\cx$,
\begin{equation}\label{eq-def-decay}
\lf|f(x)\r|\le CD_{\gamma}(x_1,x;r);
\end{equation}
\item[\rm(ii)] for any $x,y\in\cx$ satisfying $d(x,y)\le(2A_0)^{-1}[r+d(x_1,x)]$,
\begin{equation}\label{eq-def-derivative}
\lf|f(x)-f(y)\r|\le C\lf[\f{d(x,y)}{r+d(x_1,x)}\r]^{\beta}D_{\gamma}(x_1,x;r).
\end{equation}
\end{enumerate}
The collection of all   test functions of type $(x_1,r,\beta,\gamma)$ is denoted by $\mathcal{G}(x_1,r,\beta,\gamma)$
and is equipped with the norm, for any $f\in\mathcal{G}(x_1,r,\beta,\gamma)$,
$$
\lf\|f\r\|_{\mathcal{G}(x_1,r,\beta,\gamma)}:=\inf\lf\{C\in(0,\fz):\
\text{both}\ \eqref{eq-def-decay}\ \text{and}\ \eqref{eq-def-derivative}\ \text{hold}\r\}.
$$
Its subspace $\mathring{\mathcal{G}}(x_1,r,\beta,\gamma)$ is defined by setting
\begin{equation}\label{eq-def-Go}
\mathring{\mathcal{G}}(x_1,r,\beta,\gamma):=\lf\{f\in\mathcal{G}(x_1,r,\beta,\gamma):\
\int_{\cx}f(x)\,d\mu(x)=0\r\}
\end{equation}
and is equipped with the same norm as $\mathcal{G}(x_1,r,\beta,\gamma)$.
\end{definition}

It is known that, for any fixed $x_1,x_2\in\cx$ and $r_1,r_2\in(0,\fz)$, one has
$\mathcal{G}(x_1,r_1,\beta,\gamma)=\mathcal{G}(x_2,r_2,\beta,\gamma)$
and $\mathring{\mathcal{G}}(x_1,r_1,\beta,\gamma)=\mathring{\mathcal{G}}(x_2,r_2,\beta,\gamma)$
with equivalent norms, but the positive equivalence constants may depend on $x_1,\ x_2,\ r_1,$ and $r_2$.
In what follows, we may fix $x_1=x_0\in\cx$ and $r=1$ and, for simplicity,
denote $\mathcal{G}(x_0,1,\beta,\gamma)$ and $\mathring{\mathcal{G}}(x_0,1,\beta,\gamma)$, respectively, by
$\mathcal{G}(\beta,\gamma)$ and $\mathring{\mathcal{G}}(\beta,\gamma)$.
Conventionally, both  $\mathcal{G}(\beta,\gamma)$ and $\mathring{\mathcal{G}}(\beta,\gamma)$ are called
the \emph{spaces of test functions} on $\cx$.

For any fixed $\varepsilon\in(0,1]$ and $\beta,\gamma\in(0,\varepsilon]$,
let $\mathcal{G}_0^\varepsilon(\beta,\gamma)$ [resp. $\mathring{\mathcal{G}}_0^\varepsilon(\beta,\gamma)$]
be the completion of the set
$\mathcal{G}(\varepsilon,\varepsilon)$ [resp. $\mathring{\mathcal{G}}(\varepsilon,\varepsilon)$]
in $\mathcal{G}(\beta,\gamma)$ [resp. $\mathring{\mathcal{G}}(\beta,\gamma)$] with the norm defined by setting
$\|\cdot\|_{\mathcal{G}_0^\varepsilon(\beta,\gamma)}:=\|\cdot\|_{\mathcal{G}(\beta,\gamma)}$
[resp. $\|\cdot\|_{\mathring{\mathcal{G}}_0^\varepsilon(\beta,\gamma)}
:=\|\cdot\|_{\mathcal{G}(\beta,\gamma)}$].
Moreover, its dual space $(\mathcal{G}_0^\varepsilon(\beta,\gamma))'$
[resp. $(\mathring{\mathcal{G}}_0^\varepsilon(\beta,\gamma))'$] is defined as the collection of
all continuous linear functionals from $\mathcal{G}_0^\varepsilon(\beta,\gamma)$
[resp. $\mathring{\mathcal{G}}_0^\varepsilon(\beta,\gamma)$] to $\cc$ equipped with the weak-$\ast$ topology.
The spaces $(\mathcal{G}_0^\varepsilon(\beta,\gamma))'$ and $(\mathring{\mathcal{G}}_0^\varepsilon(\beta,\gamma))'$
are called the \emph{spaces of distributions} on $\cx$.

The following lemma contains some basic estimates about $\cx$;
we refer the reader to \cite[Lemma 2.1]{HMY08} or \cite[Lemma 2.4]{HLYY19} for more details.
\begin{lemma}\label{lem-X-basic}
Let $\beta,\gamma\in(0,\fz)$, $\lambda\in(0,1)$, and $D_\gamma$ be
the same as in \eqref{eq-def-D}.
\begin{enumerate}
\item[\rm(i)] For any $x,y\in\cx$ and $r\in(0,\fz)$, $V(x,y)\sim V(y,x)$ and
$$
V_r(x)+V_r(y)+V(x,y)\sim V_r(x)+V(x,y)\sim V_r(y)+V(x,y)\sim\mu(B(x,r+d(x,y)));
$$
moreover, if $d(x,y)\le r$, then $V_r(x)\sim V_r(y)$.
Here, all the positive equivalence constants are independent of $x,\ y$, and $r$.

\item[\rm(ii)] There exists a positive constant $C$ such that, for any $r\in(0,\fz)$
and  $x,y,z\in\cx$ satisfying $d(x,y)\le\lambda[r+d(x,z)]$,
$V_r(z)+V(x,z)\le C[V_r(z)+V(y,z)]$ and
$$
V_r(x)+V(x,z)\le C\lf[V_r(y)+V(y,z)\r].
$$

\item[\rm(iii)]
There exists a positive constant $C$ such that, for any $x\in\cx$ and $R\in(0,\fz)$,
$$
\int_{\{z\in\cx:\ d(x,z)\le R\}}\f{1}{V(x,y)}\lf[\f{d(x,y)}{R}\r]^\beta\,d\mu(y)\le C
$$
and
$$
\int_{\{y\in\cx:\ d(x,y)> R\}}\f{1}{V(x,y)}\lf[\f{R}{d(x,y)}\r]^\beta\,d\mu(y)\le C.
$$
\item[\rm(iv)] There exists a positive constant $C$ such that, for any $x_0\in\cx$ and $r\in(0,\fz)$,
$$
\int_{\cx}D_\gamma\lf(x_0,y;r\r)\,d\mu(y)\le C.
$$
\item[\rm(v)] There exists a positive constant $C$ such that, for any $x_0\in\cx$ and $r,R\in(0,\fz)$,
$$
\int_{\{y\in\cx:\ d(x_0,y)> R\}}D_\gamma\lf(x_0,y;r\r)\,d\mu(y)\le C\lf(\f{r}{r+R}\r)^\gamma.
$$
\end{enumerate}
\end{lemma}

The following concept of  approximations of the identity with exponential decay,
originated from \cite[Definition 2.7]{HLYY19}, helps to introduce the Besov space
or its variants on $\cx$.

\begin{definition}\label{def-exp-ATI}
A sequence $\{Q_k\}_{k\in\zz}$ of bounded linear integral operators on $L^2(\cx)$ is called an
\emph{approximation of the identity with exponential decay} (for short, exp-ATI) if there exist constants
$C,\nu\in(0,\fz)$, $a\in(0,1]$, and $\theta\in(0,1)$ such that, for any $k\in\zz$,
the kernel of the operator $Q_k$, a function on $\cx\times\cx$ which is still denoted by $Q_k$,
satisfies that
\begin{enumerate}
\item[\rm(i)](the \emph{identity condition})
$$
\sum_{k=-\fz}^{\fz}Q_k=I\ \ \text{in}\ L^2(\cx),
$$
where $I$ denotes the identity operator on $L^2(\cx)$;
\item[\rm(ii)](the \emph{size condition}) for any $x,y\in\cx$,
$$
\lf|Q_k(x,y)\r|\le C\f{1}{\sqrt{V_{\delta^k}(x)V_{\delta^k}(y)}}H_k(x,y);
$$
\item[\rm(iii)](the \emph{regularity condition}) for any $x,x',y\in\cx$ with $d(x,x')\le\delta^k$,
\begin{align*}
&\lf|Q_k(x,y)-Q_k(x',y)\r|+\lf|Q_k(y,x)-Q_k(y,x')\r|\\
&\quad\le C\lf[\delta^{-k}d(x,x')\r]^\theta\f{1}{\sqrt{V_{\delta^k}(x)V_{\delta^k}(y)}}H_k(x,y);
\end{align*}
\item[\rm(iv)](the \emph{second difference regularity condition}) for any $x,x',y,y'\in\cx$ with
both $d(x,x')\le\delta^k$ and $d(y,y')\le\delta^k$,
\begin{align*}
&\lf|\lf[Q_k(x,y)-Q_k(x',y)\r]-\lf[Q_k(x,y')-Q_k(x',y')\r]\r|\\
&\quad\le C\lf[\delta^{-k}d(x,x')\r]^\theta\lf[\delta^{-k}d(y,y')\r]^\theta
\f{1}{\sqrt{V_{\delta^k}(x)V_{\delta^k}(y)}}H_k(x,y);
\end{align*}
\item[\rm(v)](the \emph{cancellation condition}) for any $x,y\in\cx$,
$$
\int_\cx Q_k(x,y')\,d\mu(y')=0=\int_\cx Q_k(x',y)\,d\mu(x').
$$
\end{enumerate}
Here, for any $x,y\in\cx$,
$$
H_k(x,y):=\exp\lf\{-\nu\lf[\delta^{-k}d(x,y)\r]^a\r\}
\exp\lf\{-\nu\lf[\delta^{-k}\max\lf\{d(x,\mathcal{Y}^k),d(y,\mathcal{Y}^k)\r\}\r]^a\r\}
$$
and, for any $k\in\zz$, $\mathcal{Y}^k$ is defined in \eqref{eq-def-yak} and
$$
d(x,\mathcal{Y}^k):=\inf_{y\in\mathcal{Y}_k}d(x,y).
$$
\end{definition}

By \cite[Theorem 7.1]{AH13}, we find that such an exp-ATI
on spaces of homogeneous type always exists
with $\theta$ therein sometimes very small.

About exp-ATIs, we have the following properties which
are frequently used;
we refer the reader to \cite[Proposition 2.10]{HLYY19} for more details.

\begin{lemma}
Let $\{Q_k\}_{k\in\zz}$ be an exp-ATI with $\theta\in(0,1)$.
Then, for any given $\gamma\in(0,\fz)$, there exists a positive constant $C$ such that,
for any $k\in\zz$, the kernel $Q_k$ satisfies
\begin{enumerate}
\item[\rm(i)] for any $x,y\in\cx$,
\begin{equation}\label{eq-ATI-1}
\lf|Q_k(x,y)\r|\le C D_\gamma(x,y;\delta^k),
\end{equation}
where $D_\gamma(x,y;\delta^k)$ is the same as in \eqref{eq-def-D} with $x_1,\ x$, and $r$ replaced,
respectively, by $x$, $y$, and $\delta^k$;
\item[\rm(ii)] for any $x,x',y\in\cx$ with $d(x,x')\le(2A_0)^{-1}[\delta^k+d(x,y)]$,
\begin{align}\label{eq-ATI-2}
\lf|Q_k(x,y)-Q_k(x',y)\r|+\lf|Q_k(y,x)-Q_k(y,x')\r|
\le C\lf[\f{d(x,x')}{\delta^{k}+d(x,y)}\r]^\theta D_\gamma(x,y;\delta^k);
\end{align}
\item[\rm(iii)] for any $x,x',y,y'\in\cx$ with
$$
d(x,x')\le(2A_0)^{-2}[\delta^k+d(x,y)]\ \ \text{and}\ \ d(y,y')\le(2A_0)^{-2}[\delta^k+d(x,y)],
$$
\begin{align*}
&\lf|\lf[Q_k(x,y)-Q_k(x',y)\r]-\lf[Q_k(x,y')-Q_k(x',y')\r]\r|\\
&\quad\le C\lf[\f{d(x,x')}{\delta^{k}+d(x,y)}\r]^\theta\lf[\f{d(y,y')}{\delta^{k}+d(x,y)}\r]^\theta
D_\gamma(x,y;\delta^k)\nonumber.
\end{align*}
\end{enumerate}
\end{lemma}

Based on  exp-ATIs, we introduce the homogeneous logarithmic matrix-weighted Besov spaces.

\begin{definition}
Let $\theta\in(0,1)$ be the same as in Definition \ref{def-exp-ATI}, $\beta,\gamma\in(0,\theta)$,
$s\in(-(\beta\wedge\gamma),\beta\wedge\gamma)$, $b:=(b_+,b_-)$ be a pair of real numbers,
$p\in(p(s,\beta\wedge\gamma),\fz)$,
$q\in(0,\fz]$, $W\in A_p(\cx,\cc^m)$, and $\{Q_k\}_{k\in\zz}$ be an exp-ATI.
The \emph{homogeneous logarithmic matrix-weighted Besov space}, $\dot{B}^{s,b}_{p,q}(\cx,W)$, is defined by setting
$$
\dot{B}^{s,b}_{p,q}(\cx,W):=\lf\{\vec{f}\in\lf[\lf(\mathring{\mathcal{G}}_0^\theta(\beta,\gamma)\r)'\r]^m:\
\lf\|\vec{f}\,\r\|_{\dot{B}^{s,b}_{p,q}(\cx,W)}<\fz\r\},
$$
where, for any $\vec{f}\in[(\mathring{\mathcal{G}}_0^\theta(\beta,\gamma))']^m$,
$$
\lf\|\vec{f}\,\r\|_{\dot{B}^{s,b}_{p,q}(\cx,W)}:=
\lf\|\lf\{Q_k\vec{f}\r\}_{k\in\zz}\r\|_{\ell^q_{s,b}(W,L^p)}
$$
with the usual modification made when $q=\fz$.
\end{definition}

\begin{remark}
The grand characterization in  Theorem \ref{thm-B=AB} below implies that the
space $\dot{B}^{s,b}_{p,q}(\cx,W)$ is independent of the choice of
exp-ATIs.
\end{remark}

Furthermore, we also need the following concept of approximations of the identity
with exponential decay and integration $1$ (see \cite[Definition 2.8]{HHLLYY19}),
whose existence, as argued in \cite[Remark 2.9]{HHLLYY19},
is ensured by \cite[Lemma 10.1]{AH13}.

\begin{definition}\label{def-1expATI}
A sequence $\{Q_k\}_{k\in\zz}$ of bounded linear integral operators on $L^2(\cx)$ is called an
\emph{approximation of the identity with exponential decay and integration $1$} (for short, 1-exp-ATI)
if $\{Q_k\}_{k\in\zz}$ satisfies that
\begin{enumerate}
\item[\rm(i)]
for any $k\in\zz$, $Q_k$ satisfies
(ii), (iii), and (iv) of Definition \ref{def-exp-ATI}
without the decay factor
$$
\exp\lf\{-\nu\lf[\delta^{-k}\max\lf\{d(x,\mathcal{Y}^k),d(y,\mathcal{Y}^k)\r\}\r]^a\r\};
$$

\item[\rm(ii)]
for any $k\in\zz$ and $x\in\cx$,
$$
\int_\cx Q_k(x,y)\,d\mu(y)=1=\int_\cx Q_k(y,x)\,d\mu(y);
$$
\item[\rm(iii)]
$\{P_k\}_{k\in\zz}$, defined by setting, for any $k\in\zz$,
$P_k:=Q_k-Q_{k-1}$, is an exp-ATI.
\end{enumerate}
\end{definition}

As mentioned in the introduction, Auscher and Hyt\"{o}nen \cite{AH13,AH15} constructed
a wavelet system on $\cx$
which has become an essential tool in the  investigations on function spaces
on $\cx$.
Now, we recall their orthonormal wavelet basis of $L^2(\cx)$
(see \cite[Theorem 7.1]{AH13}) and the wavelet reproducing formula
based on this orthonormal wavelet basis (see \cite[Theorem 3.4 and Corollary 3.5]{HLW18}).

\begin{lemma}\label{lem-def-basis}
Let $a:=(1+2\log_2 A_0)^{-1}$. Then there exist positive constants $\theta\in(0,1)$, $\nu\in(0,\fz)$, and $C$
and an orthonormal wavelet basis $\{\psi^k_\alpha:\ k\in\zz,\ \alpha\in\cg_k\}$ of $L^2(\cx)$ such that,
for any $k\in\zz$ and $\alpha\in\cg_k$,
\begin{enumerate}
\item[\rm(i)] $($the \emph{decay condition}$)$ for any $x\in\cx$,
$$
\lf|\psi^k_\alpha(x)\r|\le\f{C}{\sqrt{V_{\delta^k}(y^k_\alpha)}}
\exp\lf\{-\nu\lf[\delta^{-k}d(y^k_\alpha,x)\r]^a\r\};
$$
\item[\rm(ii)] $($the \emph{H\"{o}lder-regularity condition}$)$ for any $x,y\in\cx$ with $d(x,y)\le\delta^k$,
$$
\lf|\psi^k_\alpha(x)-\psi^k_\alpha(y)\r|\le
\f{C}{\sqrt{V_{\delta^k}(y^k_\alpha)}}\lf[\delta^{-k}d(x,y)\r]^\theta
\exp\lf\{-\nu\lf[\delta^{-k}d(y^k_\alpha,x)\r]^a\r\};
$$
\item[\rm(iii)] $($the \emph{cancellation condition}$)$
$$
\int_{\cx}\psi^k_\alpha(x)\,d\mu(x)=0.
$$
\end{enumerate}
\end{lemma}

\begin{lemma}\label{lem-wave-repro}
Let $\theta\in(0,1)$, $\{\psi^k_\alpha:\ k\in\zz,\ \alpha\in\cg_k\}$ be the same as in Lemma \ref{lem-def-basis}, and $\beta,\gamma\in(0,\theta)$. Then, for any $f\in \mathring{\mathcal{G}}_0^\theta(\beta,\gamma)$
[resp. $f\in (\mathring{\mathcal{G}}_0^\theta(\beta,\gamma))'$],
$$
f=\sum_{k\in\zz}\sum_{\alpha\in\cg_k}\lf\langle f,\psi^k_\alpha\r\rangle\psi^k_\alpha
$$
in $\mathring{\mathcal{G}}_0^\theta(\beta,\gamma)$ [resp. $(\mathring{\mathcal{G}}_0^\theta(\beta,\gamma))'$].
\end{lemma}

Next, we introduce the
homogeneous matrix-weighted grand logarithmic Besov spaces.

\begin{definition}
Let $\theta$ be the same as in Definition \ref{def-exp-ATI}, $\beta,\gamma\in(0,\theta)$, $s\in(-\theta,\theta)$,
$b:=(b_+,b_-)$ be a pair of real numbers,
$p\in(0,\fz)$, $W\in A_p(\cx,\cc^m)$, $q\in(0,\fz]$, and, for any $k\in\zz$ and $x\in\cx$,
\begin{equation}\label{eq-def-Fk}
\mathcal{F}_k(x):=\lf\{\phi\in\mathring{\mathcal{G}}_0^\theta(\beta,\gamma):\ \lf\|\phi\r\|_{\mathring{\mathcal{G}}(x,\delta^k,\beta,\gamma)}\le1\r\}.
\end{equation}
The \emph{homogeneous matrix-weighted grand logarithmic Besov space},
$\mathcal{A}\dot{B}^{s,b}_{p,q}(W)$, is defined by setting
$$
\mathcal{A}\dot{B}^{s,b}_{p,q}(W):=\lf\{\vec{f}\in
\lf[\lf(\mathring{\mathcal{G}}_0^\theta(\beta,\gamma)\r)'\r]^m:\
\lf\|\vec{f}\,\r\|_{\mathcal{A}\dot{B}^{s,b}_{p,q}(W)}<\fz\r\},
$$
where, for any $\vec{f}\in[(\mathring{\mathcal{G}}_0^\theta(\beta,\gamma))']^m$,
$$
\lf\|\vec{f}\,\r\|_{\mathcal{A}\dot{B}^{s,b}_{p,q}(W)}:=
\lf\|\lf\{\sup_{\phi\in\mathcal{F}_k(\cdot)}
\lf\|W^{\f 1 p}(\cdot)\lf\langle\vec{f},\phi\r\rangle\r\|\r\}_{k\in\zz}\r\|_{\ell^q_{s,b}(\cx,L^p)}
$$
with the usual modification made when $q=\fz$.
\end{definition}

We have the following results.

\begin{lemma}\label{lem-phi,in,Fk(x)}
Let $\theta$ be the same as in Definition \ref{def-exp-ATI}, $\beta,\gamma\in(0,\theta)$,
and, for any $k\in\zz$ and $x\in\cx$, $\mathcal{F}_k(x)$ be the same as in \eqref{eq-def-Fk}.
\begin{enumerate}
\item[\rm(i)]
Let $\{Q_k\}_{k\in\zz}$ be a 1-exp-ATI.
Then, for any $k\in\zz$ and $x,y\in\cx$ with $d(x,y)\in(0,\delta^k]$,
the function $\phi_k^{x,y}$, defined by setting, for any $z\in\cx$,
$$
\phi_k^{x,y}(z):=Q_k(x,z)-Q_k(y,z),
$$
belongs to $\mathcal{F}_k(x)$.
\item[\rm(ii)]
There exists a positive constant $C$ such that,
for any $k\in\zz$ and $x,y\in\cx$ with $d(x,y)\in(0,\delta^k]$,
$C\phi\in\mathcal{F}_k(y)$ whenever $\phi\in\mathcal{F}_k(x)$.
\end{enumerate}
\end{lemma}

Lemma \ref{lem-phi,in,Fk(x)}(i) follows from \cite[Proposition 2.10]{HLYY19} (see also \cite[Lemma 3.22]{AWYY23})
and (ii) is easy to be checked by using Lemma \ref{lem-X-basic}(i) and the fact that,
for any $k\in\zz$ and $x,y,z\in\cx$ with $d(x,y)\in(0,\delta^k]$,
$$
\delta^k+d(x,z)\sim \delta^k+d(y,z).
$$
We omit further details here.

Now, we establish the equivalence between the matrix-weighted logarithmic Besov spaces and
the matrix-weighted grand  logarithmic Besov spaces.

\begin{theorem}\label{thm-B=AB}
Let $\theta$ be the same as in Definition \ref{def-exp-ATI}, $\beta,\gamma\in(0,\theta)$,
$s\in(-[\beta\wedge\gamma],\beta\wedge\gamma)$, $b:=(b_+,b_-)$ be a pair of real numbers,
$p\in(p(s,\beta\wedge\gamma),\fz)$, $q\in(0,\fz]$,
and $W\in A_p(\cx,\cc^m)$. Then
$\dot{B}^{s,b}_{p,q}(W)=\mathcal{A}\dot{B}^{s,b}_{p,q}(W)$
with equivalent quasi-norms.
\end{theorem}

To prove Theorem \ref{thm-B=AB},
we do some preparations.

The following Lemma \ref{lem-bas&test} reveals the relation between
the wavelet basis and the test functions,
which is just \cite[Theorem 3.3]{HLW18}.

\begin{lemma}\label{lem-bas&test}
Let $\theta\in(0,1)$, $\{\psi^k_\alpha:\ k\in\zz,\ \alpha\in\cg_k\}$ be the same as in Lemma \ref{lem-def-basis},
and $\gamma\in(0,\fz)$. Then there exists a positive constant $C$ such that,
for any $k\in\zz$ and $\alpha\in\cg_k$,
$$
\lf\|\psi^k_\alpha\r\|_{\mathcal{G}(y^k_\alpha,\delta^k,\theta,\gamma)}
\le C\lf[\mu\lf(Q^{k+1}_\alpha\r)\r]^{\f12}.
$$
\end{lemma}

The proof of the following lemma uses some ideas from the proof of \cite[Lemma 3.10]{WHHY21};
for the convenience of the reader,
we give some details.

\begin{lemma}\label{lem-est,psi,phi}
Let $\theta\in(0,1)$, $\{\psi^k_\alpha:\ k\in\zz,\ \alpha\in\cg_k\}$ be the same as in Lemma \ref{lem-def-basis},
and $\beta,\gamma\in(0,\theta)$ as in Lemma \ref{lem-wave-repro}.
Then, for any given $\gamma'\in(0,\beta\wedge\gamma)$,
there exists a positive constant $C$ such that,
for any $k,j\in\zz$, $x_0\in\cx$, and $\phi\in\mathcal{F}_j(x_0)$,
\begin{align*}
\lf|\lf\langle\psi^k_\alpha,\phi\r\rangle\r|\le
C\delta^{|k-j|\gamma'}\lf[\mu\lf(Q_{\alpha}^{k+1}\r)\r]^{\f12}D_\gamma\lf(x_0,y_\alpha^k;\delta^{k\wedge j}\r),
\end{align*}
where $D_\gamma(x_0,x_P;\delta^{k\wedge j})$ is defined as in \eqref{eq-def-D}.
\end{lemma}

\begin{proof}
By similarity, we only show the case when $k\ge j$.
Applying the cancellation condition of $\psi^k_\alpha$ in Lemma \ref{lem-def-basis}(iii)
[or the definition of $\mathring{\mathcal{G}}_0^\theta(\beta,\gamma)$ for the case when $k\le j$],
we find that, for any $k,j\in\zz,\ \alpha\in\cg_k$, $x_0\in\cx$, and $\phi\in\mathcal{F}_j(x_0)$,
\begin{align*}
\lf|\lf\langle\psi^k_\alpha,\phi\r\rangle\r|&=
\lf|\int_{\cx}\psi^k_\alpha(y)\lf[\phi(y)-\phi(y_\alpha^k)\r]\,d\mu(y)\r|\\
&\le\int_{\{y\in\cx:\ d(y,y_\alpha^k)\le(2A_0)^{-1}[\delta^j+d(x_0,y_\alpha^k)]\}}
\lf|\psi^k_\alpha(y)\r|\lf|\phi(y)-\phi(y_\alpha^k)\r|\,d\mu(y)\\
&\quad+\int_{\{y\in\cx:\ d(y,y_\alpha^k)>(2A_0)^{-1}[\delta^j+d(x_0,y_\alpha^k)]\}}
\lf|\psi^k_\alpha(y)\r|\lf|\phi(y)\r|\,d\mu(y)\\
&\quad+\lf|\phi(y_\alpha^k)\r|
\int_{\{y\in\cx:\ d(y,y_\alpha^k)>(2A_0)^{-1}[\delta^j+d(x_0,y_\alpha^k)]\}}
\lf|\psi^k_\alpha(y)\r|\,d\mu(y)\\
&=:\mathrm{I}_1+\mathrm{I}_2+\mathrm{I}_3.
\end{align*}
For simplicity, in what follows, we write $R:=(2A_0)^{-1}[\delta^j+d(x_0,y_\alpha^k)]$.

Next, we estimate $\mathrm{I}_1,\ \mathrm{I}_2$, and $\mathrm{I}_3$, respectively.
For any given $\gamma'\in(0,\beta\wedge\gamma)$,
by Lemma \ref{lem-bas&test} with $\gamma$ therein replaced by $\gamma'$,
Definition \ref{def-homo-test}, Lemma \ref{lem-X-basic}(iii) with $\beta-\gamma'\in(0,\fz)$,
we have
\begin{align*}
\mathrm{I}_1&\ls\lf[\mu\lf(Q_{\alpha}^{k+1}\r)\r]^{\f12}
\int_{\{y\in\cx:\ d(y,y_\alpha^k)\le R\}}
D_{\gamma'}(y_\alpha^k,y;\delta^k)\lf[\f{d(y,y_\alpha^k)}{\delta^j+d(x_0,y_\alpha^k)}\r]^{\beta}
D_{\gamma}(x_0,y_\alpha^k;\delta^j)\,d\mu(y)\\
&\sim\delta^{(k-j)\gamma'}\lf[\mu\lf(Q_{\alpha}^{k+1}\r)\r]^{\f12} D_{\gamma+\beta}(x_0,y_\alpha^k;\delta^j)\\
&\quad\times
\int_{\{y\in\cx:\ d(y,y_\alpha^k)\le R\}}\f{\delta^{j(\gamma'-\beta)} }{V_{\delta^k}(y_\alpha^k)+V(y_\alpha^k,y)}
\f{[d(y,y_\alpha^k)]^\beta}{[\delta^k+d(y_\alpha^k,y)]^{\gamma'}}\,d\mu(y)\\
&\ls\delta^{(k-j)\gamma'}\lf[\mu\lf(Q_{\alpha}^{k+1}\r)\r]^{\f12}
D_{\gamma+\beta}(x_0,y_\alpha^k;\delta^j)
\int_{\{y\in\cx:\ d(y,y_\alpha^k)\le R\}}\f{\delta^{j(\gamma'-\beta)}}{V(y_\alpha^k,y)}
\lf[d(y,y_\alpha^k)\r]^{\beta-\gamma'}\,d\mu(y)\\
&\sim\delta^{(k-j)\gamma'}\lf[\mu\lf(Q_{\alpha}^{k+1}\r)\r]^{\f12}
D_{\gamma}(x_0,y_\alpha^k;\delta^j)\lf[\f{\delta^j}{\delta^j+d(x_0,y_\alpha^k)}\r]^{\gamma'}\\
&\quad\times\int_{\{y\in\cx:\ d(y,y_\alpha^k)\le R\}}\f{1}{V(y_\alpha^k,y)}
\lf[\f{d(y,y_\alpha^k)}{R}\r]^{\beta-\gamma'}\,d\mu(y)\\
&\ls\delta^{(k-j)\gamma'}\lf[\mu\lf(Q_{\alpha}^{k+1}\r)\r]^{\f12}
D_{\gamma}(x_0,y_\alpha^k;\delta^j).
\end{align*}
Using Lemma \ref{lem-X-basic}(i) and the doubling property of $\mu$,
we find that, for any $y\in\{y\in\cx:\ d(y,y_\alpha^k)> R\}$,
$$
V(y,y^k_\alpha)\sim\mu\lf(B(y^k_\alpha,d(y,y^k_\alpha))\r)\gtrsim
\mu\lf(B(y^k_\alpha,\delta^j+d(x_0,y^k_\alpha))\r)
\sim V_{\delta^j}(x_0)+V(x_0,y^k_\alpha).
$$
Thus, from this,  Lemma \ref{lem-bas&test} with $\gamma\in(0,\theta)$,
Definition \ref{def-homo-test}, Lemma \ref{lem-X-basic}(iv), and $\gamma'\in(0,\beta\wedge\gamma)$,
we deduce that
\begin{align*}
\mathrm{I}_2&\ls
\lf[\mu\lf(Q_{\alpha}^{k+1}\r)\r]^{\f12}
\int_{\{y\in\cx:\ d(y,y_\alpha^k)> R\}}
D_{\gamma}(y_\alpha^k,y;\delta^k)D_{\gamma}(x_0,y;\delta^j)\,d\mu(y)\\
&\ls\delta^{(k-j)\gamma}\lf[\mu\lf(Q_{\alpha}^{k+1}\r)\r]^{\f12}
D_{\gamma}(x_0,y_\alpha^k;\delta^j)\\
&\quad\times\int_{\{y\in\cx:\ d(y,y_\alpha^k)> R\}}
\f{V_{\delta^j}(x_0)+V(x_0,y_\alpha^k)}{V_{\delta^k}(y_\alpha^k)+V(y_\alpha^k,y)}
\lf[\f{\delta^j+d(x_0,y_\alpha^k)}{\delta^k+R}\r]^\gamma
D_{\gamma}(x_0,y;\delta^j)\,d\mu(y)\\
&\ls\delta^{(k-j)\gamma}\lf[\mu\lf(Q_{\alpha}^{k+1}\r)\r]^{\f12}
D_{\gamma}(x_0,y_\alpha^k;\delta^j)\int_{\cx}
D_{\gamma}(x_0,y;\delta^j)\,d\mu(y)\\
&\ls\delta^{(k-j)\gamma'}\lf[\mu\lf(Q_{\alpha}^{k+1}\r)\r]^{\f12}
D_{\gamma}(x_0,y_\alpha^k;\delta^j).
\end{align*}
Finally, by Lemma \ref{lem-bas&test},
Definition \ref{def-homo-test}, Lemma \ref{lem-X-basic}(v), and $\gamma'\in(0,\beta\wedge\gamma)$,
we obtain
\begin{align*}
\mathrm{I}_3&\ls
\lf[\mu\lf(Q_{\alpha}^{k+1}\r)\r]^{\f12}
D_{\gamma}(x_0,y_\alpha^k;\delta^j)
\int_{\{y\in\cx:\ d(y,y_\alpha^k)> R\}}
D_{\gamma}(y_\alpha^k,y;\delta^k)\,d\mu(y)\\
&\ls\lf[\mu\lf(Q_{\alpha}^{k+1}\r)\r]^{\f12}
D_{\gamma}(x_0,y_\alpha^k;\delta^j)
\lf[\f{\delta^k}{\delta^k+R}\r]^{\gamma}\\
&\ls\delta^{(k-j)\gamma}\lf[\mu\lf(Q_{\alpha}^{k+1}\r)\r]^{\f12}
D_{\gamma}(x_0,y_\alpha^k;\delta^j)\\
&\ls\delta^{(k-j)\gamma'}\lf[\mu\lf(Q_{\alpha}^{k+1}\r)\r]^{\f12}
D_{\gamma}(x_0,y_\alpha^k;\delta^j).
\end{align*}
Altogether, we prove the case when $k\ge j$.
The remainder follows from a similar way. This finishes the proof of Lemma \ref{lem-est,psi,phi}.
\end{proof}

For any $p\in(1,\fz)$ and $W\in A_p(\cx,\cc^m)$,
the \emph{matrix-weighted Hardy--Littlewood maximal operator}
$\cm_{W,p}$ on $\cx$
is defined by setting, for any $\vec{f}\in[L^0(\cx)]^m$ and $x\in\cx$,
\begin{equation*}
\cm_{W,p}\lf(\vec{f}\,\r)(x):=\sup_{B\ni x}\fint_B
\lf\|W^{\frac{1}{p}}(x)W^{-\frac{1}{p}}(y)\vec{f}(y)\r\|\,d\mu(y),
\end{equation*}
where the supremum is taken over all  balls of $\cx$ that contain $x$;
see, for instance, \cite[(1.3)]{CG01} and \cite[(14)]{G03}.
We also recall that $W\in A_p(\cx,\cc^m)$ if and only if $\cm_{W,p}$ is bounded
from $[L^p(\cx)]^m$ to $L^p(\cx)$ (see \cite[Theorem 3.3]{BYY23}
and \cite[Theorem 3.2]{G03}).

The following lemma is also needed.

\begin{lemma}\label{lem-est-Mw}
Let $A\in(0,1]$, $\gamma\in(0,\fz)$, $p\in(1,\fz)$, and $W\in A_p(\cx,\cc^m)$.
Then there exists a positive constant $C$ such that,
for any set $\{\vec{a}_{k,\alpha}:\ k\in\zz,\ \alpha\in\cg_k\}\subset\cc^m$, any $k,l\in\zz$,
and any $x\in\cx$,
\begin{align}\label{eq-lem,A}
&\sum_{\alpha\in\cg_k}
\lf[\mu\lf(Q_{\alpha}^{k+1}\r)\r]^{\f A2}
\lf\|W^{\f 1 p}(x)\vec{a}_{k,\alpha}\r\|^A
\lf[D_\gamma\lf(x,y_\alpha^k;\delta^{k\wedge l}\r)\r]^A\nonumber\\
&\quad\le C \sum_{j=0}^{\fz} \delta^{j[\gamma A+D(A-1)]}\delta^{[(k\wedge l)-k]D(1-A)}
\fint_{B(x,\delta^{(k\wedge l)-j})}\lf\|W^{\f 1 p}(x)
\lf[\sum_{\alpha\in\cg_k}\vec{a}_{k,\alpha}
\wz{\mathbf{1}}_{Q_{\alpha}^{k+1}}(y)\r]\r\|^A\,d\mu(y);
\end{align}
especially, when $A=1$,
\begin{align*}
\sum_{\alpha\in\cg_k}
\lf[\mu\lf(Q_{\alpha}^{k+1}\r)\r]^{\f 12}
\lf\|W^{\f 1 p}(x)\vec{a}_{k,\alpha}\r\|
D_\gamma\lf(x,y_\alpha^k;\delta^{k\wedge l}\r)
\le C\cm_{W,p}\,\Bigg(W^{\f 1 p}\sum_{\alpha\in\cg_k}\vec{a}_{k,\alpha}
\wz{\mathbf{1}}_{Q_{\alpha}^{k+1}}\Bigg)\lf(x\r).
\end{align*}
Here and thereafter, for any $k\in\zz$ and $\alpha\in\cg_k$,
$\wz{\mathbf{1}}_{Q_{\alpha}^{k+1}}:=
[\mu(Q_{\alpha}^{k+1})]^{-1/2}\mathbf{1}_{Q_{\alpha}^{k+1}}$
\end{lemma}

\begin{proof}
Since, for any $k\in\zz$, $\{Q_\alpha^{k+1}:\ \alpha\in\cg_k\}$ are pairwise disjoint,
we then infer that,
for any given $A\in(0,1]$, any $k,l\in\zz$, and any  $x\in\cx$,
\begin{align*}
&\sum_{\alpha\in\cg_k}
\lf[\mu\lf(Q_{\alpha}^{k+1}\r)\r]^{\f A2}
\lf\|W^{\f 1 p}(x)\vec{a}_{k,\alpha}\r\|^A
\lf[D_\gamma\lf(x,y_\alpha^k;\delta^{k\wedge l}\r)\r]^A\\
&\quad=\int_{B(x,\delta^{k\wedge l})}\sum_{\alpha\in\cg_k}
\lf[\mu\lf(Q_{\alpha}^{k+1}\r)\r]^{\frac{A}{2}-1}
\lf\|W^{\f 1 p}(x)\vec{a}_{k,\alpha}\mathbf{1}_{Q_{\alpha}^{k+1}}(y)\r\|^A\\
&\quad\quad\times\lf[\f{1}{V_{\delta^{k\wedge l}}(x)+V(x,y_\alpha^k)}\r]^A
\lf[\f{\delta^{k\wedge l}}{\delta^{k\wedge l}+d(x,y_\alpha^k)}\r]^{\gamma A}\,d\mu(y)\\
&\quad\quad+\sum_{j=0}^{\fz}
\int_{B(x,\delta^{(k\wedge l)-j-1})\setminus B(x,\delta^{(k\wedge l)-j})}\cdots\\
&\quad=:\mathrm{I}+\sum_{j=0}^{\fz}\mathrm{I}_j.
\end{align*}
Now, we estimate $\mathrm{I}$ and $\mathrm{I}_j$ for any $j\in\zz_+$, respectively.

By Lemma \ref{lem-X-basic}(i), the doubling property of $\mu$,
and \eqref{eq-defB(Q)} with a geometrical observation, we find that,
for any $k,l\in\zz$, $x\in\cx$, and $y\in B(x,\delta^{k\wedge l})\cap Q_{\alpha}^{k+1}$
with some $\alpha\in\cg_k$,
\begin{equation}\label{eq-lem,est-ca}
V_{\delta^{k\wedge l}}(x)\sim V_{\delta^{k\wedge l}}(y)\ls
\delta^{[(k\wedge l)-k]D}V_{\delta^k}(y)\sim \delta^{[(k\wedge l)-k]D}\mu\lf(Q_{\alpha}^{k+1}\r).
\end{equation}
Thus, using this and $A\in(0,1]$, we obtain
 \begin{align*}
\mathrm{I}
&\ls\f{1}{V_{\delta^{k\wedge l}}(x)}\int_{B(x,\delta^{k\wedge l})}\sum_{\alpha\in\cg_k}
\lf[\f{\mu(Q_{\alpha}^{k+1})}{V_{\delta^{k\wedge l}}(x)}\r]^{A-1}
\lf\|W^{\f 1 p}(x)\vec{a}_{k,\alpha}\wz{\mathbf{1}}_{Q_{\alpha}^{k+1}}(y)\r\|^A\,d\mu(y)\\
&\ls \delta^{[k-(k\wedge l)]D(A-1)}\fint_{B(x,\delta^{k\wedge l})}\sum_{\alpha\in\cg_k}
\lf\|W^{\f 1 p}(x)\vec{a}_{k,\alpha}\wz{\mathbf{1}}_{Q_{\alpha}^{k+1}}(y)\r\|^A\,d\mu(y)\\
&\sim \delta^{[k-(k\wedge l)]D(A-1)}\fint_{B(x,\delta^{k\wedge l})}
\lf\|W^{\f 1 p}(x)\lf[\sum_{\alpha\in\cg_k}\vec{a}_{k,\alpha}
\wz{\mathbf{1}}_{Q_{\alpha}^{k+1}}(y)\r]\r\|^A\,d\mu(y).
 \end{align*}

On the other hand, to estimate $\mathrm{I}_j$ for $j\in\zz_+$,
notice that, for any $k,l\in\zz$, $\alpha\in\cg_k$, and $y\in Q_{\alpha}^{k+1}$,
we have
$$d(y,y_\alpha^k)\ls \delta^k\ls\delta^{k\wedge l}.$$
It then follows that,
for any $k,l\in\zz$, $\alpha\in\cg_k$, $y\in Q_{\alpha}^{k+1}$, and $x\in\cx$,
\begin{equation}\label{eq-delta+d=,1}
\delta^{k\wedge l}+d(x,y_\alpha^k)\sim\delta^{k\wedge l}+d(x,y)
\end{equation}
and, moreover, by both (i) and (ii) of Lemma \ref{lem-X-basic}, we have
\begin{align}\label{eq-delta+d=,2}
V_{\delta^{k\wedge l}}(x)+V(x,y_\alpha^k)\sim
V_{\delta^{k\wedge l}}(y_\alpha^k)+V(x,y_\alpha^k)\sim
V_{\delta^{k\wedge l}}(y)+V(x,y)\sim
V_{\delta^{k\wedge l}+d(x,y)}(x).
\end{align}
In addition, for any $j\in\zz_+$, if $y\in B(x,\delta^{(k\wedge l)-j-1})\setminus B(x,\delta^{(k\wedge l)-j})$,
then
\begin{equation}\label{eq-delta+d=,3}
\delta^{k\wedge l}+d(x,y)\sim\delta^{(k\wedge l)-j}.
\end{equation}
Thus, from \eqref{eq-delta+d=,1}, \eqref{eq-delta+d=,2}, \eqref{eq-delta+d=,3},
the doubling property of $\mu$, \eqref{eq-lem,est-ca}, and $A\in(0,1]$,
we deduce that, for any $j\in\zz_+$,
\begin{align*}
\mathrm{I}_j
&\sim\int_{B(x,\delta^{(k\wedge l)-j-1})\setminus B(x,\delta^{(k\wedge l)-j})}\sum_{\alpha\in\cg_k}
\lf[\mu\lf(Q_{\alpha}^{k+1}\r)\r]^{\frac{A}{2}-1}
\lf\|W^{\f 1 p}(x)\vec{a}_{k,\alpha}\mathbf{1}_{Q_{\alpha}^{k+1}}(y)\r\|^A\\
&\quad\times\lf[\f{1}{V_{\delta^{k\wedge l}}(x)+V(x,y_\alpha^k)}\r]^A
\lf[\f{\delta^{k\wedge l}}{\delta^{k\wedge l}+d(x,y)}\r]^{\gamma A}\,d\mu(y)\\
&\ls\lf[\f{\delta^{j\gamma}}{V_{\delta^{(k\wedge l)-j-1}}(x)}\r]^A
\int_{B(x,\delta^{(k\wedge l)-j-1})\setminus B(x,\delta^{(k\wedge l)-j})}
\sum_{\alpha\in\cg_k}
\lf[\mu\lf(Q_{\alpha}^{k+1}\r)\r]^{\frac{A}{2}-1}\\
&\quad\times
\lf\|W^{\f 1 p}(x)\vec{a}_{k,\alpha}\mathbf{1}_{Q_{\alpha}^{k+1}}(y)\r\|^A\,d\mu(y)\\
&\ls\delta^{j\gamma A}\fint_{B(x,\delta^{(k\wedge l)-j-1})}\sum_{\alpha\in\cg_k}
\lf[\f{\mu\lf(Q_{\alpha}^{k+1}\r)}{V_{\delta^{(k\wedge l)-j-1}}(x)}\r]^{A-1}
\lf\|W^{\f 1 p}(x)\vec{a}_{k,\alpha}\wz{\mathbf{1}}_{Q_{\alpha}^{k+1}}(y)\r\|^A\,d\mu(y)\\
&\ls\delta^{j[\gamma A+D(A-1)]}\delta^{[(k\wedge l)-k]D(1-A)}
\fint_{B(x,\delta^{(k\wedge l)-j-1})}
\lf\|W^{\f 1 p}(x)\lf[\sum_{\alpha\in\cg_k}\vec{a}_{k,\alpha}\wz{\mathbf{1}}_{Q_{\alpha}^{k+1}}(y)\r]\r\|^A\,d\mu(y).
\end{align*}

Altogether, we obtain \eqref{eq-lem,A}.
Especially, when $A=1$, we further conclude that, for any $x\in\cx$,
\begin{align*}
&\sum_{\alpha\in\cg_k}
\lf[\mu\lf(Q_{\alpha}^{k+1}\r)\r]^{\f 1 2}
\lf\|W^{\f 1 p}(x)\vec{a}_{k,\alpha}\r\|
D_\gamma\lf(x,y_\alpha^k;\delta^{k\wedge l}\r)\\
&\quad\ls\sum_{j=0}^{\fz} \delta^{j\gamma}
\fint_{B(x,\delta^{(k\wedge l)-j})}\lf\|W^{\f 1 p}(x)W^{-\f 1 p}(y)W^{\f 1 p}(y)
\lf[\sum_{\alpha\in\cg_k}\vec{a}_{k,\alpha}\wz{\mathbf{1}}_{Q_{\alpha}^{k+1}}(y)\r]\r\|\,d\mu(y)\\
&\quad\ls\sum_{j=0}^{\fz} \delta^{j\gamma}\cm_{W,p}\,\Bigg(W^{\f 1 p}\sum_{\alpha\in\cg_k}\vec{a}_{k,\alpha}
\wz{\mathbf{1}}_{Q_{\alpha}^{k+1}}\Bigg)\lf(x\r)
\ls\cm_{W,p}\,\Bigg(W^{\f 1 p}\sum_{\alpha\in\cg_k}\vec{a}_{k,\alpha}
\wz{\mathbf{1}}_{Q_{\alpha}^{k+1}}\Bigg)\lf(x\r).
\end{align*}
This finishes the proof of Lemma \ref{lem-est-Mw}.
\end{proof}

With all these preparations,
we next show Theorem \ref{thm-B=AB}.

\begin{proof}[Proof of Theorem \ref{thm-B=AB}]
We first assume that $\vec{f}\in \mathcal{A}\dot{B}^{s,b}_{p,q}(W)$.
Let $\{Q_k\}_{k\in\zz}$ be an exp-ATI.
By \cite[Remark 3.2]{AWYY23}, we find that
$$
\lf\|W^{\f 1 p}Q_k\lf(\vec{f}\r)\r\|\le\sup_{\phi\in\mathcal{F}_k(\cdot)}
\lf\|W^{\f 1 p}(\cdot)\lf\langle\vec{f},\phi\r\rangle\r\|,
$$
which further implies that
$\|\vec{f}\,\|_{\dot{B}^{s,b}_{p,q}(W)}\le\|\vec{f}\,\|_{\mathcal{A}\dot{B}^{s,b}_{p,q}(W)}$.

Conversely, we assume that $\vec{f}\in \dot{B}^{s,b}_{p,q}(W)$.
Applying the vector-valued variant of Lemma \ref{lem-wave-repro}, we conclude that there exists
an orthonormal wavelet basis $\{\psi^k_\alpha:\ k\in\zz,\ \alpha\in\cg_k\}$ as in Lemma \ref{lem-def-basis} such that,
for any  $l\in\zz$, $x\in\cx$, and $\phi\in\mathcal{F}_l(x)$,
\begin{align}\label{eq-pf-repro}
\lf\langle\vec{f},\phi\r\rangle=\sum_{k\in\zz}\sum_{\alpha\in\cg_k}
\lf\langle\vec{f},\psi_\alpha^k\r\rangle\lf\langle\psi_\alpha^k,\phi\r\rangle.
\end{align}
Thus, from \eqref{eq-pf-repro} and Lemma \ref{lem-est,psi,phi},  we infer that,
for any given $\gamma'\in(0,\beta\wedge\gamma)$, any $l\in\zz$, any $x\in\cx$, and any $\phi\in\mathcal{F}_l(x)$,
\begin{align}\label{eq-est-ca,1}
&\lf\|W^{\f 1 p}(x)\lf\langle\vec{f},\phi\r\rangle\r\|\nonumber\\
&\quad\le\sum_{k\in\zz}\sum_{\alpha\in\cg_k}
\lf\|W^{\f 1 p}(x)\lf\langle\vec{f},\psi_\alpha^k\r\rangle\r\|
\lf|\lf\langle\psi_\alpha^k,\phi\r\rangle\r|\nonumber\\
&\quad\ls\sum_{k\in\zz}\delta^{|k-l|\gamma'}\sum_{\alpha\in\cg_k}
\lf[\mu\lf(Q_{\alpha}^{k+1}\r)\r]^{\f 12}
\lf\|W^{\f 1 p}(x)\lf\langle\vec{f},\psi_\alpha^k\r\rangle\r\|
D_\gamma\lf(x,y_\alpha^k;\delta^{k\wedge l}\r).
\end{align}
In what follows,
we fix $\gamma'\in(0\vee s,\beta\wedge\gamma)$.

We first let $p\in(1,\fz)$.
Applying Lemma \ref{lem-est-Mw} to \eqref{eq-est-ca,1},
we find that,
for any $l\in\zz$ and $x\in\cx$,
\begin{align*}
\sup_{\phi\in\mathcal{F}_l(x)}
\lf\|W^{\f 1 p}(x)\lf\langle\vec{f},\phi\r\rangle\r\|
\ls\sum_{k\in\zz}\delta^{|k-l|\gamma'}\cm_{W,p}\,\Bigg(W^{\f 1 p}
\sum_{\alpha\in\cg_k}\lf\langle\vec{f},\psi_\alpha^k\r\rangle
\wz{\mathbf{1}}_{Q_{\alpha}^{k+1}}\Bigg)\lf(x\r),
\end{align*}
where $\wz{\mathbf{1}}_{Q_{\alpha}^{k+1}}:=
[\mu(Q_{\alpha}^{k+1})]^{-1/2}\mathbf{1}_{Q_{\alpha}^{k+1}}$
and $\cm_{W,p}$ is the matrix-weighted Hardy--Littlewood maximal operator.
Thus, from this, the Minkowski inequality, the $L^p(\cx)$-boundedness of $\cm_{W,p}$ (see \cite[Theorem 3.3]{BYY23}),
\eqref{triangle} when $q\in(0,1]$ or the H\"{o}lder inequality when $q\in(1,\fz]$, and the choice of $\gamma'$,
we deduce that
\begin{align*}
\lf\|\vec{f}\,\r\|_{\mathcal{A}\dot{B}^{s,b}_{p,q}(W)}
&\ls \lf\{\sum_{l\in\zz}\lf[L^\delta_{s,b}(\delta^l)\r]^{-q}
\lf\|\sum_{k\in\zz}\delta^{|k-l|\gamma'}\cm_{W,p}\,\Bigg(W^{\f 1 p}
\sum_{\alpha\in\cg_k}\lf\langle\vec{f},\psi_\alpha^k\r\rangle
\wz{\mathbf{1}}_{Q_{\alpha}^{k+1}}\Bigg)\r\|_{L^p(\cx)}^q\r\}^{\f 1 q}\\
&\ls \lf\{\sum_{l\in\zz}\lf[L^\delta_{s,b}(\delta^l)\r]^{-q}
\lf[\sum_{k\in\zz}\delta^{|k-l|\gamma'}\lf\|\cm_{W,p}\,\Bigg(W^{\f 1 p}
\sum_{\alpha\in\cg_k}\lf\langle\vec{f},\psi_\alpha^k\r\rangle
\wz{\mathbf{1}}_{Q_{\alpha}^{k+1}}\Bigg)\r\|_{L^p(\cx)}\r]^q\r\}^{\f 1 q}\\
&\ls \lf[\sum_{l\in\zz}
\lf\{\sum_{k\in\zz}\lf[L^\delta_{s,b}(\delta^l)\r]^{-1}\delta^{|k-l|\gamma'}\lf\|W^{\f 1 p}\sum_{\alpha\in\cg_k}
\lf\langle\vec{f},\psi_\alpha^k\r\rangle
\wz{\mathbf{1}}_{Q_{\alpha}^{k+1}}\r\|_{L^p(\cx)}\r\}^q\r]^{\f 1 q}\\
&\sim\lf[\sum_{l\in\zz}
\lf\{\sum_{k\in\zz}\f{L^\delta_{s,b}(\delta^k)}{L^\delta_{s,b}(\delta^l)}\delta^{|k-l|\gamma'}
\lf[L^\delta_{s,b}(\delta^k)\r]^{-1}\lf\|W^{\f 1 p}\sum_{\alpha\in\cg_k}
\lf\langle\vec{f},\psi_\alpha^k\r\rangle
\wz{\mathbf{1}}_{Q_{\alpha}^{k+1}}\r\|_{L^p(\cx)}\r\}^q\r]^{\f 1 q}\\
&\ls \lf\{\sum_{k\in\zz}\sum_{l\in\zz}
\lf[\f{L^\delta_{s,b}(\delta^k)}{L^\delta_{s,b}(\delta^l)}\r]^{q\wedge1}\delta^{|k-l|\gamma'(q\wedge1)}
\lf[L^\delta_{s,b}(\delta^k)\r]^{-q}\r.\\
&\quad\lf.\times\lf\|W^{\f 1 p}\sum_{\alpha\in\cg_k}
\lf\langle\vec{f},\psi_\alpha^k\r\rangle
\wz{\mathbf{1}}_{Q_{\alpha}^{k+1}}\r\|_{L^p(\cx)}^q\r\}^{\f 1 q}\\
&\ls \lf\{\sum_{k\in\zz}\lf[L^\delta_{s,b}(\delta^k)\r]^{-q}
\lf\|\sum_{\alpha\in\cg_k}\lf\langle\vec{f},\psi_\alpha^k\r\rangle
\wz{\mathbf{1}}_{Q_{\alpha}^{k+1}}\r\|_{L^p(W)}^q\r\}^{\f 1 q},
\end{align*}
which, together with the estimate in \cite[p.\,52]{BYY23}
that, for any fixed $\eta'\in(0\vee s,\beta\wedge\gamma)$,
$$
\lf\|\sum_{\alpha\in\cg_k}\lf\langle\vec{f},\psi_\alpha^k\r\rangle
\wz{\mathbf{1}}_{Q_{\alpha}^{k+1}}\r\|_{L^p(W)}
\ls\sum_{k'\in\zz}\delta^{|k-k'|\eta'}\lf\|Q_{k'}\lf(\vec{f}\,\r)\r\|_{L^p(W)}
$$
with $\{Q_{k'}\}_{k'\in\zz}$ being an exp-ATI,
further implies that
\begin{align*}
\lf\|\vec{f}\,\r\|_{\mathcal{A}\dot{B}^{s,b}_{p,q}(W)}
\ls \lf\{\sum_{k\in\zz}\lf[L^\delta_{s,b}(\delta^k)\r]^{-q}
\lf[\sum_{k'\in\zz}\delta^{|k-k'|\eta'}\lf\|Q_{k'}\lf(\vec{f}\,\r)\r\|_{L^p(W)}\r]^q\r\}^{\f 1 q}\ls\lf\|\vec{f}\,\r\|_{\dot{B}^{s,b}_{p,q}(W)}.
\end{align*}
This proves the case when $p\in(1,\fz)$.

Next, we let $p\in(p(s,\beta\wedge\gamma),1]$.
We first claim that, for any $k,l\in\zz$ and $\alpha\in\cg_k$,
\begin{align}\label{eq-claim1}
&\int_\cx\lf\|W^{\f 1 p}(x)\lf\langle\vec{f},\psi_\alpha^k\r\rangle\r\|^p
\lf[D_\gamma\lf(x,y_\alpha^k;\delta^{k\wedge l}\r)\r]^p\,d\mu(x)\nonumber\\
&\quad\ls\delta^{[(k\wedge l)-k]D(1-p)}
\lf[\mu\lf(Q_{\alpha}^{k+1}\r)\r]^{1- p}\inf_{z\in Q_{\alpha}^{k+1}\setminus E}
\lf\|W^{\f 1 p}(z)\lf\langle\vec{f},\psi_\alpha^k\r\rangle\r\|^p,
\end{align}
where $E$ is some measurable set with $\mu(E)=0$
which comes from the definition of the matrix $A_p$-weight (see Definition \ref{def-Apweight}).
Indeed, by an argument similar to that used in the proof of Lemma \ref{lem-est-Mw},
Lemma \ref{lem-X-basic}(i),
the definition of matrix $A_p$-weights, \eqref{eq-lem,est-ca}, the doubling property of $\mu$,
and $p\in(p(s,\beta\wedge\gamma),1]$,
we find that, for any $k,l\in\zz$, $\alpha\in\cg_k$, and almost every $z\in Q_\alpha^{k+1}$,
\begin{align*}
&\int_\cx\lf\|W^{\f 1 p}(x)\lf\langle\vec{f},\psi_\alpha^k\r\rangle\r\|^p
\lf[D_\gamma\lf(x,y_\alpha^k;\delta^{k\wedge l}\r)\r]^p\,d\mu(x)\\
&\quad=\int_{B(y_\alpha^k,c_0\delta^{k\wedge l})}\lf\|W^{\f 1 p}(x)\lf\langle\vec{f},\psi_\alpha^k\r\rangle\r\|^p
\lf[D_\gamma\lf(x,y_\alpha^k;\delta^{k\wedge l}\r)\r]^p\,d\mu(x)\\
&\quad\quad+\sum_{j=0}^{\fz}\int_{B(y_\alpha^k,c_0\delta^{(k\wedge l)-j-1})\setminus
B(y_\alpha^k,c_0\delta^{(k\wedge l)-j})}\cdots \\
&\quad\ls\lf[V_{c_0\delta^{k\wedge l}}(y_\alpha^k)\r]^{1-p}
\fint_{B(y_\alpha^k,c_0\delta^{k\wedge l})}\lf\|W^{\f 1 p}(x)\lf\langle\vec{f},\psi_\alpha^k\r\rangle\r\|^p \,d\mu(x)\\
&\quad\quad+\sum_{j=0}^{\fz}\delta^{j\gamma p}\lf[V_{c_0\delta^{(k\wedge l)-j-1}}(y_\alpha^k)\r]^{1-p}
\fint_{B(y_\alpha^k,c_0\delta^{(k\wedge l)-j-1})}\cdots \\
&\quad\ls\lf[V_{c_0\delta^{k\wedge l}}(y_\alpha^k)\r]^{1-p}
\lf\|W^{\f 1 p}(z)\lf\langle\vec{f},\psi_\alpha^k\r\rangle\r\|^p
\fint_{B(y_\alpha^k,c_0\delta^{k\wedge l})}\lf\|W^{\f 1 p}(x)W^{-\f 1 p}(z)\r\|^p\,d\mu(x)\\
&\quad\quad+\sum_{j=0}^{\fz}\delta^{j\gamma p}\lf[V_{c_0\delta^{(k\wedge l)-j-1}}(y_\alpha^k)\r]^{1-p}
\lf\|W^{\f 1 p}(z)\lf\langle\vec{f},\psi_\alpha^k\r\rangle\r\|^p
\fint_{B(y_\alpha^k,c_0\delta^{(k\wedge l)-j-1})}\cdots \\
&\quad\ls\lf[\mu\lf(Q_{\alpha}^{k+1}\r)\r]^{1- p}
\lf\|W^{\f 1 p}(z)\lf\langle\vec{f},\psi_\alpha^k\r\rangle\r\|^p
\sum_{j=0}^{\fz}\delta^{j\gamma p}\delta^{[(k\wedge l)-k-j]D(1-p)}\\
&\quad\ls\delta^{[(k\wedge l)-k]D(1-p)}\lf[\mu\lf(Q_{\alpha}^{k+1}\r)\r]^{1- p}
\lf\|W^{\f 1 p}(z)\lf\langle\vec{f},\psi_\alpha^k\r\rangle\r\|^p,
\end{align*}
where $c_0$ is some fixed positive constant such that, for any $k\in\zz$ and $\alpha\in\cg_k$,
$Q_\alpha^{k+1}\subset B(y_\alpha^k,c_0\delta^k)$.
This, together with the  arbitrariness of $z\in Q_\alpha^{k+1}$, proves the claim.

By \eqref{eq-est-ca,1} with the arbitrariness of both $\phi\in\mathcal{F}_l(x)$ and $x\in\cx$,
the  Minkowski  inequality, \eqref{triangle},
and the claim \eqref{eq-claim1},
we conclude that, for any $l\in\zz$,
\begin{align}\label{eq-est-ca,2}
&\int_\cx\sup_{\phi\in\mathcal{F}_l(x)}
\lf\|W^{\f 1 p}(x)\lf\langle\vec{f},\phi\r\rangle\r\|^p\,d\mu(x)\nonumber\\
&\quad\ls\sum_{k\in\zz}\delta^{|k-l|\gamma'p}\sum_{\alpha\in\cg_k}
\lf[\mu\lf(Q_{\alpha}^{k+1}\r)\r]^{\f p2}
\int_\cx \lf\|W^{\f 1 p}(x)\lf\langle\vec{f},\psi_\alpha^k\r\rangle\r\|^p
\lf[D_\gamma\lf(x,y_\alpha^k;\delta^{k\wedge l}\r)\r]^p\,d\mu(x)\nonumber\\
&\quad\ls\sum_{k\in\zz}\delta^{|k-l|\gamma'p}\sum_{\alpha\in\cg_k}
\lf[\mu\lf(Q_{\alpha}^{k+1}\r)\r]^{1-\f p2}\delta^{[(k\wedge l)-k]D(1-p)}\inf_{z\in Q_{\alpha}^{k+1}\setminus E}
\lf\|W^{\f 1 p}(z)\lf\langle\vec{f},\psi_\alpha^k\r\rangle\r\|^p\nonumber\\
&\quad\ls\sum_{k\in\zz}\delta^{|k-l|\gamma'p}\delta^{[(k\wedge l)-k]D(1-p)}\sum_{\alpha\in\cg_k}
\int_{\cx}\lf\|W^{\f 1 p}(z)\lf\langle\vec{f},\psi_\alpha^k\r\rangle
\wz{\mathbf{1}}_{Q_{\alpha}^{k+1}}(z)\r\|^p\,d\mu(z)\nonumber\\
&\quad\ls\sum_{k\in\zz}\delta^{|k-l|[\gamma'p-D(1-p)]}
\lf\|\sum_{\alpha\in\cg_k}\lf\langle\vec{f},\psi_\alpha^k\r\rangle
\wz{\mathbf{1}}_{Q_{\alpha}^{k+1}}\r\|^p_{L^p(W)},
\end{align}
where $E$ is some measurable set with $\mu(E)=0$.
As $s\in(-[\beta\wedge\gamma],\beta\wedge\gamma)$ and
$p\in(p(s,\beta\wedge\gamma),\fz)$, we may
choose $\gamma'\in(0, \beta\wedge\gamma)$ such that
$$
s\in\lf(-\gamma',\gamma'\r)\ \ \text{and}\ \
p\in\lf(\max\lf\{\frac{D}{D+\gamma'},\frac{D}{D+\gamma'+s}\r\},1\r].
$$
Thus, from this, \eqref{eq-est-ca,2},
\eqref{triangle} when $q/p\in(0,1]$ or the H\"{o}lder inequality when $q/p\in(1,\fz)$,
and the choice of $\gamma'$,
it follows that
\begin{align*}
\lf\|\vec{f}\,\r\|_{\mathcal{A}\dot{B}^{s,b}_{p,q}(W)}
&\ls \lf\{\sum_{l\in\zz}\lf[L^\delta_{s,b}(\delta^l)\r]^{-q}
\lf\{\sum_{k\in\zz}\delta^{|k-l|[\gamma'p-D(1-p)]}\r.\r.\\
&\quad\times\lf.\lf.\lf\|\sum_{\alpha\in\cg_k}\lf\langle\vec{f},\psi_\alpha^k\r\rangle
\wz{\mathbf{1}}_{Q_{\alpha}^{k+1}}\r\|^p_{L^p(W)}\r\}^{\f q p}\r\}^{\f 1 q}\\
&\ls \lf\{\sum_{k\in\zz}\lf[L^\delta_{s,b}(\delta^k)\r]^{-q}
\sum_{l\in\zz}\delta^{|k-l|[\gamma'p-D(1-p)](\f q p\wedge 1)}\r.\\
&\quad\lf.\times\lf[\f{L^\delta_{s,b}(\delta^k)}{L^\delta_{s,b}(\delta^l)}\r]^{q\wedge p}
\lf\|\sum_{\alpha\in\cg_k}\lf\langle\vec{f},\psi_\alpha^k\r\rangle
\wz{\mathbf{1}}_{Q_{\alpha}^{k+1}}\r\|_{L^p(W)}^{q}\r\}^{\f 1 q}\\
&\ls \lf\{\sum_{k\in\zz}\lf[L^\delta_{s,b}(\delta^k)\r]^{-q}
\lf\|\sum_{\alpha\in\cg_k}\lf\langle\vec{f},\psi_\alpha^k\r\rangle
\wz{\mathbf{1}}_{Q_{\alpha}^{k+1}}\r\|_{L^p(W)}^q\r\}^{\f 1 q}.
\end{align*}
Using the estimate in \cite[p.\,51]{BYY23}
that, for any fixed $\eta'\in(0\vee s,\beta\wedge\gamma)$,
$$
\lf\|\sum_{\alpha\in\cg_k}\lf\langle\vec{f},\psi_\alpha^k\r\rangle
\wz{\mathbf{1}}_{Q_{\alpha}^{k+1}}\r\|_{L^p(W)}^p
\ls\sum_{k'\in\zz}\delta^{|k-k'|\eta' p}\delta^{[(k\wedge k')-k']D(1-p)}\lf\|Q_{k'}\lf(\vec{f}\,\r)\r\|_{L^p(W)}^p
$$
with $\{Q_{k'}\}_{k'\in\zz}$ being an exp-ATI,
we further obtain
\begin{align*}
&\lf\|\vec{f}\,\r\|_{\mathcal{A}\dot{B}^{s,b}_{p,q}(W)}\\
&\quad\ls\lf\{\sum_{k\in\zz}\lf[L^\delta_{s,b}(\delta^k)\r]^{-q}
\lf[\sum_{k'\in\zz}\delta^{|k-k'|\eta' p}\delta^{[(k\wedge k')-k']D(1-p)}
\lf\|Q_{k'}\lf(\vec{f}\,\r)\r\|_{L^p(W)}^p\r] ^{\frac{q}{p}}\r\}^{\f 1 q}\nonumber\\
&\quad\ls\lf\{\sum_{k'\in\zz}\lf[L^\delta_{s,b}(\delta^{k'})\r]^{-q}
\lf\|Q_{k'}\lf(\vec{f}\,\r)\r\|_{L^p(W)}^q\r\}^{\f 1 q}
\sim\lf\|\vec{f}\,\r\|_{\dot{B}^{s,b}_{p,q}(W)}\nonumber,
\end{align*}
which proves the case when $p\in(p(s,\beta\wedge\gamma),1]$.
This finishes the proof of Theorem \ref{thm-B=AB}.
\end{proof}

\subsection{Pointwise Characterizations of Matrix-Weighted
Logarithmic Besov Spaces}\label{sec-haj}

In this subsection, we apply the  matrix-weighted Poincar\'{e}-type
inequalities established in Section \ref{sec-poincare}
to prove the equivalence between the matrix-weighted logarithmic Haj{\l}asz--Besov spaces
and the matrix-weighted logarithmic grand Besov spaces.

We first recall the concept of the (weak) lower bound; see, for instance,
\cite[(2) and (3)]{AGH20} and \cite[Definition 1.1]{HHHLP21}.

\begin{definition}
The measure $\mu$ is said to \emph{have a lower bound $\lambda$} with $\lambda\in(0,D]$ if
there exists a positive constant $C$ such that, for any $x\in\cx$ and $r\in(0,\fz)$, $\mu(B(x,r))\ge Cr^{\lambda}$;
the measure $\mu$ is said to \emph{have a weak lower bound $\lambda$} with $\lambda\in(0,D]$ if
there exists a point $x_0\in\cx$ and a positive constant $C$ such that,
for any $r\in[1,\fz)$, $\mu(B(x_0,r))\ge Cr^{\lambda}$.
\end{definition}

\begin{remark}
As pointed out in \cite[Proposition 2.15]{AWYY23}, with $D$ defined the same as in \eqref{eq-doub},
the doubling measure $\mu$ has a weak lower bound $\lambda=D$
if and only if it has a lower bound $\lambda=D$.
\end{remark}

The main theorem of this section reads as follows.

\begin{theorem}\label{thm-N=AB}
Assume that the doubling measure $\mu$ of $\cx$ has a weak lower bound $\lambda=D$.
Let $\theta$ be the same as in Definition \ref{def-exp-ATI}, $\beta,\gamma\in(0,\theta)$,
$s\in(0,\beta\wedge\gamma)$, $b:=(b_+,b_-)$ be a pair of real numbers,
$p\in(D/[D+s],\fz)$, $q\in(0,\fz]$,
and $W\in A_p(\cx,\cc^m)$. Then
$\dot{N}^{s,b}_{p,q}(W)=\mathcal{A}\dot{B}^{s,b}_{p,q}(W)$.
\end{theorem}

To prove Theorem \ref{thm-N=AB}, we  need the following lemma.

\begin{lemma}\label{lem-Lloc}
Assume that the doubling measure $\mu$ of $\cx$ has a weak lower bound $\lambda=D$.
Let $\theta$ be the same as in Definition \ref{def-exp-ATI}, $\beta,\gamma\in(0,\theta)$, $s\in(0,\beta\wedge\gamma)$,
$b:=(b_+,b_-)$ be a pair of real numbers,
$p\in(D/[D+s],\fz)$, $q\in(0,\fz]$, and $W\in A_p(\cx,\cc^m)$.
Then, for any $\vec{f}\in \ca\dot{B}^{s,b}_{p,q}(W)$,
there exists a function $\vec{f}^*\in[L_\loc^1(\cx)]^m$ such that
$\vec{f}^*=\vec{f}$ in $[(\mathring{\mathcal{G}}_0^\theta(\beta,\gamma))']^m$.
\end{lemma}

\begin{proof}
Let $\vec{f}\in\ca \dot{B}^{s,b}_{p,q}(W)$
and $\{Q_k\}_{k\in\zz}$ be a 1-exp-ATI (see Definition \ref{def-1expATI}).

Recall that $Q_k\vec{f}\rar\vec{f}$ in $[(\mathring{\mathcal{G}}_0^\theta(\beta,\gamma))']^m$ as $k\rar\fz$.
Then
\begin{equation}\label{eq-f=Qkf,in,()'}
\vec{f}=Q_0\vec{f}-
\sum_{k=0}^{\fz}\lf(Q_k\vec{f}-Q_{k+1}\vec{f}\,\r)\ \ \text{in}\
\lf[\lf(\mathring{\mathcal{G}}_0^\theta(\beta,\gamma)\r)'\r]^m.
\end{equation}
Next, we consider the following two cases on $p$.

When $p\in[1,\fz)$,  we find that, for any $x\in\cx$,
\begin{align}\label{eq-pf-p>1-1}
&\lf\|W^{\frac 1 p}(x)\sum_{k=0}^{\fz}\lf[Q_{k}\vec{f}(x)
-Q_{k+1}\vec{f}(x)\r]\r\|\nonumber\\
&\quad\le\sum_{k=0}^{\fz}\lf\|W^{\frac 1 p}(x)\lf[Q_{k}\vec{f}(x)
-Q_{k+1}\vec{f}(x)\r]\r\|\nonumber\\
&\quad=\sum_{k=0}^{\fz}\lf\|W^{\frac 1 p}(x)
\lf\langle\vec{f}(\cdot),Q_{k}(x,\cdot)-Q_{k+1}(x,\cdot)\r\rangle\r\|.
\end{align}
Combining (ii) and (iii) of Definition \ref{def-1expATI}, \eqref{eq-ATI-1}, \eqref{eq-ATI-2},
and \eqref{eq-def-Fk}, we find that, for any $k\in\zz$ and $x\in\cx$,
\begin{equation}\label{eq-Qk-Qk,in,Fj}
Q_{k}(x,\cdot)-Q_{k+1}(x,\cdot)\in \mathcal{F}_{k+1}(x).
\end{equation}
Thus, \eqref{eq-pf-p>1-1} implies that
\begin{align}\label{eq-pf-p>1-2}
&\lf\|W^{\frac 1 p}(x)\sum_{k=0}^{\fz}\lf[Q_{k}\vec{f}(x)
-Q_{k+1}\vec{f}(x)\r]\r\|
\le\sum_{k=1}^{\fz}\sup_{\phi\in\mathcal{F}_{k}(x)}\lf\|W^{\frac 1 p}(x)
\lf\langle\vec{f},\phi\r\rangle\r\|,
\end{align}
which, combined with the Minkowski inequality and $s\in(0,\beta\wedge\gamma)$, further implies that
\begin{align*}
&\lf\|\sum_{k=0}^{\fz}\lf[Q_{k}\vec{f}(x)
-Q_{k+1}\vec{f}(x)\r]\r\|_{L^p(W)}\\
&\quad\ls \lf\|\lf\{\f{1}{L_{s,b}(\delta^{k})}\sup_{\phi\in\mathcal{F}_{k}(x)}\lf\|W^{\frac 1 p}(x)
\lf\langle\vec{f},\phi\r\rangle\r\|\r\}_{k\in\zz}\r\|_{\ell^q(W,L^p)}
\ls\lf\|\vec{f}\,\r\|_{\ca \dot{B}^{s,b}_{p,q}(W)}<\fz\nonumber.
\end{align*}
From this and the H\"{o}lder inequality, it follows that, for any given $i\in\zz$ and $x\in\cx$,
\begin{align*}
&\int_{B(x,\delta^i)}\sum_{Q\in\cq_i}\lf\|A_Q\lf\{\sum_{k=0}^{\fz}\lf[Q_{k}\vec{f}(y)
-Q_{k+1}\vec{f}(y)\r]\r\}\r\|\mathbf{1}_Q(y)\,d\mu(y)\\
&\quad\le\lf\{\int_{B(x,\delta^i)}\sum_{Q\in\cq_i}
\lf\|A_QW^{-\frac 1 p}(y)\r\|^{p'}\mathbf{1}_Q(y)\,d\mu(y)\r\}^{\frac{1}{p'}}\lf\|\sum_{k=0}^{\fz}\lf[Q_{k}\vec{f}(x)
-Q_{k+1}\vec{f}(x)\r]\r\|_{L^p(W)}\\
&\quad\ls\lf\{\sum_{\gfz{Q\in\cq_i}{Q\cap B(x,\delta^i)\neq\emptyset}}\mu(Q)
\fint_{Q}\lf\|A_QW^{-\frac 1 p}(y)\r\|^{p'}\mathbf{1}_Q(y)\,d\mu(y)\r\}^{\frac{1}{p'}}<\fz.
\end{align*}
Notice that $\{A_Q\}_{Q\in\cq}$ and also their inverse operators are linear bounded operators
and that the total number of the set $\{Q\in\cq_i:\ Q\cap B(x,\delta^i)\neq\emptyset\}$ is finite.
It then follows that
$$\sum_{k=0}^{\fz}\lf(Q_{k}\vec{f}
-Q_{k+1}\vec{f}\,\r)\in L^1(B(x,\delta^i)),$$
which means that
$$\sum_{k=0}^{\fz}\lf(Q_{k}\vec{f}
-Q_{k+1}\vec{f}\,\r)\in L_\loc^1(\cx).$$

When $p\in(D/[D+s],1)$, applying Lemmas \ref{lem-AQ=AQ'},  \ref{lem-AQW}(i),
and \ref{lem-phi,in,Fk(x)}(ii) and the assumption of $\mu$,
we find that, for any given $k\in\zz$, any $x\in\cx$, and any $Q\in\cq_k$ containing $x$,
\begin{align*}
&\sup_{\phi\in\mathcal{F}_{k}(x)}\lf\|A_Q\lf\langle\vec{f},\phi\r\rangle\r\|\\
&\quad\le\lf[\fint_{B(x,\delta^k)}\lf\|A_QW^{-\frac 1 p}(y)\r\|^p
\sup_{\phi\in\mathcal{F}_{k}(x)}\lf\|W^{\frac 1 p}(y)
\lf\langle\vec{f},\phi\r\rangle\r\|^p\,d\mu(y)\r]^{\frac 1 p}\\
&\quad\ls\sup_{P\in\cq}\esssup_{y\in P}\lf\|A_PW^{-\frac 1 p}(y)\r\|
\lf[\fint_{B(x,\delta^k)}\sup_{\phi\in\mathcal{F}_{k}(x)}\lf\|W^{\frac 1 p}(y)
\lf\langle\vec{f},\phi\r\rangle\r\|^p\,d\mu(y)\r]^{\frac 1 p}\\
&\quad\ls\lf[\fint_{B(x,\delta^k)}\sup_{\phi\in\mathcal{F}_{k}(y)}
\lf\|W^{\frac 1 p}(y)\lf\langle\vec{f},\phi\r\rangle\r\|^p\,d\mu(y)\r]^{\frac 1 p}\\
&\quad\ls\delta^{-\frac{kD}{p}}\lf[\int_{\cx}\sup_{\phi\in\mathcal{F}_{k}(y)}\lf\|W^{\frac 1 p}(y)
\lf\langle\vec{f},\phi\r\rangle\r\|^p\,d\mu(y)\r]^{\frac 1 p}
\end{align*}
and hence
\begin{align}\label{eq-*}
&\lf\|\sum_{Q\in\cq_k}\sup_{\phi\in\mathcal{F}_{k}(\cdot)}\lf\|A_Q
\lf\langle\vec{f},\phi\r\rangle\r\|\mathbf{1}_Q\,\r\|_{L^1(\cx)}\nonumber\\
&\quad=\int_{\cx}\sum_{Q\in\cq_k}\sup_{\phi\in\mathcal{F}_{k}(x)}
\lf\|A_Q\lf\langle\vec{f},\phi\r\rangle\r\|^{1-p}
\sup_{\phi\in\mathcal{F}_{k}(x)}
\lf\|A_Q\lf\langle\vec{f},\phi\r\rangle\r\|^p\mathbf{1}_Q\,d\mu(x)\nonumber\\
&\quad\ls\delta^{-\frac{(1-p)kD}{p}}
\lf\|\sup_{\phi\in\mathcal{F}_{k}(\cdot)}\lf\|W^{\frac 1 p}(\cdot)
\lf\langle\vec{f},\phi\r\rangle\r\|\,\r\|_{L^p(\cx)}^{1-p}\nonumber\\
&\quad\quad\times\sup_{Q\in\cq}\esssup_{x\in Q}\lf\|A_QW^{-\frac 1 p}(x)\r\|^p
\lf\|\sup_{\phi\in\mathcal{F}_{k}(\cdot)}\lf\|W^{\frac 1 p}(\cdot)
\lf\langle\vec{f},\phi\r\rangle\r\|\,\r\|_{L^p(\cx)}^p\nonumber\\
&\quad\ls\delta^{(D-\frac{D}{p})k}
\lf\|\sup_{\phi\in\mathcal{F}_{k}(\cdot)}\lf\|W^{\frac 1 p}(\cdot)
\lf\langle\vec{f},\phi\r\rangle\r\|\,\r\|_{L^p(\cx)}.
\end{align}
Since, by an argument similar to that used in \eqref{eq-pf-p>1-2} with $W^{1/p}(x)$ replaced by $A_Q$,
and \eqref{eq-sharp-est-p<1}, we conclude that, for any $x\in\cx$ and $Q\in\cq_0$ containing $x$,
\begin{align*}
\lf\|A_Q\sum_{k=0}^{\fz}\lf[Q_{k}\vec{f}(x)
-Q_{k+1}\vec{f}(x)\r]\r\|
&\le\sum_{k=1}^{\fz}\sup_{\phi\in\mathcal{F}_{k}(x)}
\lf\|A_Q\lf\langle\vec{f},\phi\r\rangle\r\|\\
&\ls\sum_{k=1}^{\fz}\sum_{P\in\cq_k}\sup_{\phi\in\mathcal{F}_{k}(x)}
\lf\|A_P\lf\langle\vec{f},\phi\r\rangle\r\|\mathbf{1}_P(x),
\end{align*}
then, from this, \eqref{eq-*}, and $p\in(D/(D+s),1]$, we infer that
\begin{align*}
&\lf\|\sum_{Q\in\cq_0}A_Q\lf(\sum_{k=0}^{\fz}Q_{k}\vec{f}
-Q_{k+1}\vec{f}\r)\mathbf{1}_Q\,\r\|_{L^1(\cx)}\\
&\quad\ls\sum_{k=1}^{\fz}\delta^{(D-\frac{D}{p})k}
\lf\|\sup_{\phi\in\mathcal{F}_{k}(\cdot)}\lf\|W^{\frac 1 p}(\cdot)
\lf\langle\vec{f},\phi\r\rangle\r\|\,\r\|_{L^p(\cx)}\\
&\quad\ls \sum_{k=1}^{\fz}\delta^{(D-\frac{D}{p})k}L_{s,b}(\delta^k)
\lf\|\vec{f}\,\r\|_{\mathcal{A}\dot{B}^{s,b}_{p,q}(W)}
\ls\lf\|\vec{f}\,\r\|_{\mathcal{A}\dot{B}^{s,b}_{p,q}(W)}<\fz.
\end{align*}
Moreover, we use again the fact that $\{A_Q^{-1}\}_{Q\in\cq}$
is a family of linear bounded operators to obtain
$$\sum_{k=0}^{\fz}\lf(Q_{k}\vec{f}
-Q_{k+1}\vec{f}\,\r)\in L_\loc^1(\cx).$$

Finally,
let $$\wz{\vec{f}}:=Q_0\vec{f}-\sum_{k=0}^{\fz}(Q_{k}\vec{f}-Q_{k+1}\vec{f}\,).$$
Then $\wz{\vec{f}}\in L_\loc^1(\cx)$
and, by \eqref{eq-f=Qkf,in,()'},
 $\vec{f}=\wz{\vec{f}}$ in
$[(\mathring{\mathcal{G}}_0^\theta(\beta,\gamma))']^m$.
In this sense, we have $\vec{f}\in L_\loc^1(W)$.
This finishes the proof of Lemma \ref{lem-Lloc}.
\end{proof}

Now, we show Theorem \ref{thm-N=AB}.

\begin{proof}[Proof of Theorem \ref{thm-N=AB}]
Let $\vec{f}\in \dot{N}^{s,b}_{p,q}(W)$, $\{A_Q\}_{Q\in\cq}\in\mathcal{RS}_\cq^{W,p}$, and
$\vec{\mathbf{g}}:=\{\vec{g}_k\}_{k\in\zz}\in\dd^{s,b}_{\{A_Q\}}(\vec{f})$ be such that
$\|\{\vec{g}_k\}_{k\in\zz}\|_{\ell^q(W,L^p)}\ls\|\vec{f}\,\|_{\dot{N}^{s,b}_{p,q}(W)}$.
We first prove $\|\vec{f}\,\|_{\mathcal{A}\dot{B}^{s,b}_{p,q}(W)}\ls\|\vec{f}\,\|_{\dot{N}^{s,b}_{p,q}(W)}$.

By \eqref{eq-def-Fk}, \eqref{eq-def-Go}, \eqref{eq-def-decay}, \eqref{eq-def-D},
and Lemma \ref{lem-X-basic}(i), we find that,
for any $k\in\zz$, $x\in\cx$, and $\phi\in\mathcal{F}_k(x)$,
\begin{align}\label{eq-thmpf-1}
&\lf\|W^{\f 1 p}(x)\lf\langle\vec{f},\phi\r\rangle\r\|\nonumber\\
&\quad\le\int_\cx\lf\|W^{\f 1 p}(x)\lf[\vec{f}(y)-\vec{f}_{B(x,\delta^k)}\r]\r\|\lf|\phi(y)\r|\,d\mu(y)\nonumber\\
&\quad\le\int_\cx\lf\|W^{\f 1 p}(x)\lf[\vec{f}(y)-\vec{f}_{B(x,\delta^k)}\r]\r\|D_\gamma(x,y;\delta^k)\,d\mu(y)\nonumber\\
&\quad\le\int_{B(x,\delta^k)}\lf\|W^{\f 1 p}(x)\lf[\vec{f}(y)-\vec{f}_{B(x,\delta^k)}\r]\r\|
\f{1}{V_{\delta^k}(x)+V(x,y)}\lf[\f{\delta^k}{\delta^k+d(x,y)}\r]^\gamma\,d\mu(y)\nonumber\\
&\quad\quad+\sum_{j=0}^{\fz}\int_{B(x,\delta^{k-j-1})\setminus B(x,\delta^{k-j})}\cdots\nonumber\\
&\quad\ls \sum_{j=0}^{\fz}\delta^{j\gamma}\fint_{B(x,\delta^{k-j})}
\lf\|W^{\f 1 p}(x)\lf[\vec{f}(y)-\vec{f}_{B(x,\delta^k)}\r]\r\|\,d\mu(y).
\end{align}
Furthermore, from the doubling property of $\mu$,
it follows that, for any $k\in\zz$ and $j\in\zz_+$,
\begin{align*}
&\fint_{B(x,\delta^{k-j})}\lf\|W^{\f 1 p}(x)\lf[\vec{f}(y)-\vec{f}_{B(x,\delta^k)}\r]\r\|\,d\mu(y)\\
&\quad\ls\fint_{B(x,\delta^{k-j})}\lf\|W^{\f 1 p}(x)\lf[\vec{f}(y)-\vec{f}_{B(x,\delta^{k-j})}\r]\r\|\,d\mu(y)\\
&\quad\quad+
\fint_{B(x,\delta^{k-j+1})}\lf\|W^{\f 1 p}(x)\lf[\vec{f}(y)-\vec{f}_{B(x,\delta^{k-j})}\r]\r\|\,d\mu(y)\\
&\quad\quad+\cdots+
\fint_{B(x,\delta^{k})}\lf\|W^{\f 1p}(x)\lf[\vec{f}(y)-\vec{f}_{B(x,\delta^{k-1})}\r]\r\|\,d\mu(y)\\
&\quad\ls\sum_{i=0}^{j}\fint_{B(x,\delta^{k-i})}
\lf\|W^{\f 1 p}(x)\lf[\vec{f}(y)-\vec{f}_{B(x,\delta^{k-i})}\r]\r\|\,d\mu(y).
\end{align*}
Inserting this into \eqref{eq-thmpf-1}, we obtain,
for any $k\in\zz$, $x\in\cx$, and $\phi\in\mathcal{F}_k(x)$,
\begin{align}\label{eq-thmpf-2}
\lf\|W^{\f 1 p}(x)\lf\langle\vec{f},\phi\r\rangle\r\|
&\ls\sum_{j=0}^{\fz}\delta^{j\gamma}\sum_{i=0}^{j}\fint_{B(x,\delta^{k-i})}
\lf\|W^{\f 1 p}(x)\lf[\vec{f}(y)-\vec{f}_{B(x,\delta^{k-i})}\r]\r\|\,d\mu(y)\nonumber\\
&\ls\sum_{i=0}^{\fz}\delta^{i\gamma}
\inf_{\vec{\eta}\in\cc^m}\fint_{B(x,\delta^{k-i})}
\lf\|W^{\f 1 p}(x)\lf[\vec{f}(y)-\vec{\eta}\r]\r\|\,d\mu(y).
\end{align}

When $p\in(1,\fz)$, we choose $\lambda\in[p-\varepsilon,p)$
with $\varepsilon\in(0,1)$ as in Theorem \ref {lem-Poincare}(i).
Applying Theorem\ref{lem-Poincare}(i) to \eqref{eq-thmpf-2},
we conclude that, for any $x\in\cx$,
\begin{align*}
\lf\|W^{\f 1 p}(x)\lf\langle\vec{f},\phi\r\rangle\r\|
&\ls\sum_{i=0}^{\fz}\delta^{i\gamma}
L_{s,b}(\delta^{k-i})\gamma_{k-i}(x)\sum_{l=k-i-3}^{k-i-1}
\lf[E_{k-i}\lf(\lf\|W^{\frac 1 p}\vec{g}_l\r\|^{\lambda}\r)(x)\r]^{\frac{1}{\lambda}}.\\
&\ls\sum_{l=1}^3\delta^{k\gamma}\sum_{j=-\fz}^{k}\delta^{-j\gamma}
L_{s,b}(\delta^{j})\gamma_{j}(x) \lf[E_{j}\lf(\lf\|W^{\frac 1 p}\vec{g}_{j-l}\r\|^{\lambda}\r)(x)\r]^{\frac{1}{\lambda}},
\end{align*}
where $\gamma_{j}$ is defined in \eqref{eq-def-rj}
and $E_{j}$ defined in \eqref{eq-def-Ek}.
Thus, from this, the Minkowski inequality,
\eqref{triangle} for any $q\in(0,1]$ or the H\"{o}lder inequality for any $q\in(1,\fz]$,
$s<\gamma$, Lemma \ref{lem-rkEkf}, $\lambda<p$,
and the choice of $\vec{\mathbf{g}}$, we deduce that
\begin{align*}
\lf\|\vec{f}\,\r\|_{\mathcal{A}\dot{B}^{s,b}_{p,q}(W)}
&\ls\sum_{l=0}^3\lf\{\sum_{k\in\zz}\lf[L_{s,b}(\delta^k)\r]^{-q}\delta^{k\gamma q}
\lf\|\sum_{j=-\fz}^{k}\delta^{-j\gamma}
L_{s,b}(\delta^{j})\gamma_{j}
\lf[E_{j}\lf(\lf\|W^{\frac 1 p}\vec{g}_{j-l}\r\|^{\lambda}\r)(x)\r]
^{\frac{1}{\lambda}}\r\|_{L^p(\cx)}^q\r\}^{\frac 1 q}\\
&\ls\sum_{l=0}^3\lf\{\sum_{k\in\zz}\lf[L_{s,b}(\delta^k)\r]^{-q}\delta^{k\gamma q}
\lf[\sum_{j=-\fz}^{k}\delta^{-j\gamma}
L_{s,b}(\delta^{j})\r.\r.\\
&\quad\times\lf.\lf.\lf\|\gamma_{j} \lf[E_{j}\lf(\lf\|W^{\frac 1 p}\vec{g}_{j-l}\r\|
^{\lambda}\r)(x)\r]^{\frac{1}{\lambda}}\r\|_{L^p(\cx)}\r]^q\r\}^{\frac 1 q}\\
&\ls\sum_{l=0}^3\lf\{\sum_{j\in\zz}\sum_{k=j}^{\fz}
\lf[\f{\delta^{k\gamma}}{L_{s,b}(\delta^k)}\r]^{\min\{1,q\}}
\lf[\delta^{-j\gamma}L_{s,b}(\delta^{j})\r]^{\min\{1,q\}}\r.\\
&\quad\times\lf.\lf\|\gamma_{j}
\lf[E_{j}\lf(\lf\|W^{\frac 1 p}\vec{g}_{j-l}\r\|^{\lambda}\r)(x)\r]
^{\frac{1}{\lambda}}\r\|_{L^p(\cx)}^q\r\}^{\frac 1 q}\\
&\ls\sum_{l=0}^3\lf\{\sum_{j\in\zz}\lf\|\gamma_{j}
\lf[E_{j}\lf(\lf\|W^{\frac 1 p}\vec{g}_{j-l}\r\|^{\lambda}\r)(x)\r]
^{\frac{1}{\lambda}}\r\|_{L^p(\cx)}^q\r\}^{\frac 1 q}\\
&\ls\|\{\vec{g}_j\}_{j\in\zz}\|_{\ell^q(W,L^p)}\ls\lf\|\vec{f}\,\r\|_{\dot{N}^{s,b}_{p,q}(W)}.
\end{align*}

When $p\in(D/[D+s],1]$, applying Theorem \ref{lem-Poincare}(ii) to \eqref{eq-thmpf-2},
we obtain
\begin{align*}
\lf\|W^{\f 1 p}(x)\lf\langle\vec{f},\phi\r\rangle\r\|
&\ls \sum_{i=0}^{\fz}\delta^{i\gamma}
\delta^{(k-i)\vp'}\sum_{l=k-i-2}^{\fz}\delta^{-l\vp'}L_{s,b}(\delta^{l})
\gamma_{k-i}(x)\lf[E_{k-i}\lf(\lf\|W^{\frac 1 p}\vec{g}_l\r\|
^{\frac{D}{D+\varepsilon}}\r)(x)\r]^{\frac{D+\varepsilon}{D}}
\end{align*}
for some $\varepsilon,\varepsilon'\in(0,s)$ satisfying $\varepsilon<\varepsilon'$.
Thus, by using this, dividing the sum $\sum_{i=0}^{\fz}\sum_{l=k-i-2}^{\fz}$ into
$\sum_{l=-\fz}^{k-2}\sum_{i=k-l-2}^{\fz}$ and $\sum_{l=k-1}^{\fz}\sum_{i=0}^{\fz}$,
and applying Lemma \ref{lem-rkEkf} with $\lambda$ replaced by $(D+\varepsilon)/D$,
we conclude that, for any $k\in\zz$,
\begin{align*}
&\lf\|\sup_{\phi\in\mathcal{F}_k(\cdot)}\lf\|W^{\f 1 p}\lf\langle\vec{f},\phi\r\rangle\r\|\,\r\|_{L^p(\cx)}\\
&\quad\ls\delta^{k\vp'}\lf\|\sum_{i=0}^{\fz}\delta^{i(\gamma-\varepsilon')}
\sum_{l=k-i-2}^{\fz}\delta^{-l\vp'}L_{s,b}(\delta^{l})
\gamma_{k-i}\lf[E_{k-i}\lf(\lf\|W^{\frac 1 p}\vec{g}_l\r\|
^{\frac{D}{D+\varepsilon}}\r)\r]^{\frac{D+\varepsilon}{D}}\r\|_{L^p(\cx)}\\
&\quad\ls\delta^{k\vp'}\lf\|\sum_{l=-\fz}^{k-2}\sum_{i=k-l-2}^{\fz}\delta^{i(\gamma-\varepsilon')}
\delta^{-l\vp'}L_{s,b}(\delta^{l})
\sup_{j\in\zz}\gamma_{j}\lf[E_{j}\lf(\lf\|W^{\frac 1 p}\vec{g}_l\r\|
^{\frac{D}{D+\varepsilon}}\r)\r]^{\frac{D+\varepsilon}{D}}\r\|_{L^p(\cx)}\\
&\quad\quad+\delta^{k\vp'}\lf\|\sum_{l=k-1}^{\fz}\sum_{i=0}^{\fz}\delta^{i(\gamma-\varepsilon')}
\delta^{-l\vp'}L_{s,b}(\delta^{l})
\sup_{j\in\zz}\gamma_{j}\lf[E_{j}\lf(\lf\|W^{\frac 1 p}\vec{g}_l\r\|
^{\frac{D}{D+\varepsilon}}\r)\r]^{\frac{D+\varepsilon}{D}}\r\|_{L^p(\cx)}\\
&\quad\ls\delta^{k\gamma}\lf\|\sum_{l=-\fz}^{k-2}\delta^{-l\gamma}L_{s,b}(\delta^{l})
\sup_{j\in\zz}\gamma_{j}\lf[E_{j}\lf(\lf\|W^{\frac 1 p}\vec{g}_l\r\|
^{\frac{D}{D+\varepsilon}}\r)\r]^{\frac{D+\varepsilon}{D}}\r\|_{L^p(\cx)}\\
&\quad\quad+\delta^{k\varepsilon'}\lf\|\sum_{l=k-1}^{\fz}\delta^{-l\varepsilon'}L_{s,b}(\delta^{l})
\sup_{j\in\zz}\gamma_{j}\lf[E_{j}\lf(\lf\|W^{\frac 1 p}\vec{g}_l\r\|
^{\frac{D}{D+\varepsilon}}\r)\r]^{\frac{D+\varepsilon}{D}}\r\|_{L^p(\cx)}\\
&\quad\ls\delta^{k\gamma}\lf\{\sum_{l=-\fz}^{k-2}\delta^{-l\gamma p}\lf[L_{s,b}(\delta^{l})\r]^p
\lf\|\sup_{j\in\zz}\gamma_{j}\lf[E_{j}\lf(\lf\|W^{\frac 1 p}\vec{g}_l\r\|
^{\frac{D}{D+\varepsilon}}\r)\r]^{\frac{D+\varepsilon}{D}}
\r\|_{L^p(\cx)}^p\r\}^{\frac 1 p}\\
&\quad\quad+\delta^{k\varepsilon'}\lf\{\sum_{l=k-1}^{\fz}\delta^{-l\varepsilon' p}\lf[L_{s,b}(\delta^{l})\r]^p
\lf\|\sup_{j\in\zz}\gamma_{j}\lf[E_{j}\lf(\lf\|W^{\frac 1 p}\vec{g}_l\r\|
^{\frac{D}{D+\varepsilon}}\r)\r]^{\frac{D+\varepsilon}{D}}
\r\|_{L^p(\cx)}^p\r\}^{\frac 1 p}\\
&\quad\ls\delta^{k\gamma}\lf\{\sum_{l=-\fz}^{k-2}\delta^{-l\gamma p}\lf[L_{s,b}(\delta^{l})\r]^p
\lf\|\vec{g}_l\r\|_{L^p(W)}^p\r\}^{\frac 1 p}\\
&\quad\quad+\delta^{k\varepsilon'}\lf\{\sum_{l=k-1}^{\fz}\delta^{-l\varepsilon' p}\lf[L_{s,b}(\delta^{l})\r]^p
\lf\|\vec{g}_l\r\|_{L^p(W)}^p\r\}^{\frac 1 p},
\end{align*}
which, together with \eqref{triangle} when $q/p\in(0,1]$ or the H\"{o}lder inequality when $q/p\in(1,\fz]$
and with $\varepsilon'<s<\gamma$, further implies that
\begin{align*}
\lf\|\vec{f}\,\r\|_{\mathcal{A}\dot{B}^{s,b}_{p,q}(W)}&\ls
\lf[\sum_{k\in\zz}\lf[\delta^{-k\gamma}L_{s,b}(\delta^k)\r]^{-q}
\lf\{\sum_{l=-\fz}^{k-2}\delta^{-l\gamma p}\lf[L_{s,b}(\delta^{l})\r]^p
\lf\|\vec{g}_l\r\|_{L^p(W)}^p\r\}^{\frac q p}\r]^{\frac 1 q}\\
&\quad+\lf[\sum_{k\in\zz}\lf[\delta^{-k\vp'}L_{s,b}(\delta^k)\r]^{-q}
\lf\{\sum_{l=k-1}^{\fz}\delta^{-l\vp'p}\lf[L_{s,b}(\delta^{l})\r]^p
\lf\|\vec{g}_l\r\|_{L^p(W)}^p\r\}^{\frac q p}\r]^{\frac 1 q}\\
&\ls
\lf[\sum_{l\in\zz}\sum_{k=l+2}^{\fz}\lf[\delta^{-k\gamma}L_{s,b}(\delta^k)\r]^{-\min\{p,q\}}
\lf[\delta^{-l\gamma}L_{s,b}(\delta^{l})\r]^{\min\{p,q\}}
\lf\|\vec{g}_l\r\|_{L^p(W)}^q\r]^{\frac 1 q}\\
&\quad+\lf[\sum_{l\in\zz}\sum_{k=-\fz}^{l+1}\lf[\delta^{-k\vp'}L_{s,b}(\delta^k)\r]^{-\min\{p,q\}}
\lf[\delta^{-l\vp'}L_{s,b}(\delta^{l})\r]^{\min\{p,q\}}
\lf\|\vec{g}_l\r\|_{L^p(W)}^q\r]^{\frac 1 q}\\
&\ls\lf\|\{\vec{g}_l\}_{l\in\zz}\r\|_{\ell^q(W,L^p)}\ls\lf\|\vec{f}\,\r\|_{\dot{N}^{s,b}_{p,q}(W)}.
\end{align*}

Altogether, this finishes the proof of $\|\vec{f}\,\|_{\mathcal{A}\dot{B}^{s,b}_{p,q}(W)}\ls\|\vec{f}\,\|_{\dot{N}^{s,b}_{p,q}(W)}$.

Conversely,  let $\vec{f}\in \mathcal{A}\dot{B}^{s,b}_{p,q}(W)$ and we prove
$\|\vec{f}\,\|_{\dot{N}^{s,b}_{p,q}(W)}\ls\|\vec{f}\,\|_{\mathcal{A}\dot{B}^{s,b}_{p,q}(W)}$.
By Lemma \ref{lem-Lloc}, without loss of generality, we may assume that $\vec{f}\in[L_\loc^1(\cx)]^m$.

Let $x,y\in\cx$ and $k_0\in\zz$ be such that $d(x,y)\in[\delta^{k_0+1},\delta^{k_0})$
and let $Q$ be the unique dyadic cube of $\cq_{k_0}$ that contains $x$.
Notice that there exists a subsequence, which, for simplicity,
we still denote by $\{Q_k\vec{f}\}_{k\in\nn}$, such that,
for almost every $x\in\cx$, $Q_k\vec{f}(x)\rar\vec{f}(x)$ as $k\rar\fz$.
Then, applying \eqref{eq-Qk-Qk,in,Fj} and Lemma \ref{lem-AQ=AQ'},  we find that
\begin{align*}
&\lf\|A_Q\vec{f}(x)-A_Q\vec{f}(y)\r\|\\
&\quad\le\Bigg\|A_Q\Bigg\{\sum_{k=k_0}^{\fz}\lf[Q_{k}\vec{f}(x)-Q_{k+1}\vec{f}(x)\r]
+\lf[Q_{k}\vec{f}(y)-Q_{k+1}\vec{f}(y)\r]\\
&\quad\quad+\lf[Q_{k_0}\vec{f}(x)-Q_{k_0}\vec{f}(y)\r]\Bigg\}\Bigg\|\\
&\quad\ls\Bigg\|\sum_{k=k_0}^{\fz}A_Q\lf\langle\vec{f},Q_k(x,\cdot)-Q_{k+1}(x,\cdot)\r\rangle
+A_Q\lf\langle\vec{f},Q_{k_0}(x,\cdot)-Q_{k_0}(y,\cdot)\r\rangle\Bigg\|\\
&\quad\quad+\Bigg\|\sum_{k=k_0}^{\fz}A_Q\lf\langle\vec{f},Q_k(y,\cdot)-Q_{k+1}(y,\cdot)\r\rangle\Bigg\| \\
&\quad\ls\sup_{\gfz{\{\{\phi_{k,x}\}_{k\ge k_0}:\ \phi_{k,x}\in\mathcal{F}_{k}(x),\,k\ge k_0\}}
{\{\{\phi_{k,y}\}_{k\ge k_0}:\ \phi_{k,y}\in\mathcal{F}_{k}(y),\,k\ge k_0\}}}
\lf\{\lf\|\sum_{k=k_0}^{\fz}A_Q\lf\langle\vec{f},\phi_{k,x}\r\rangle\r\|+
\lf\|\sum_{k=k_0}^{\fz}A_Q\lf\langle\vec{f},\phi_{k+1,y}\r\rangle\r\|\r\}\\
&\quad\ls\sum_{Q\in\cq_{k_0}}\sup_{\{\{\phi_{k,x}\}_{k\ge k_0}:\ \phi_{k,x}\in\mathcal{F}_{k}(x),\,k\ge k_0\}}
\lf\|\sum_{k=k_0}^{\fz}A_Q\lf\langle\vec{f},\phi_{k,x}\r\rangle\r\|\mathbf{1}_Q(x)\\
&\quad\quad+\sum_{P\in\cq_{k_0}}
\sup_{\{\{\phi_{k,y}\}_{k\ge k_0}:\ \phi_{k,y}\in\mathcal{F}_{k}(y),\,k\ge k_0\}}
\lf\|\sum_{k=k_0}^{\fz}A_P\lf\langle\vec{f},\phi_{k+1,y}\r\rangle\r\|\mathbf{1}_P(y).
\end{align*}
By the property of the supremum, for any given $z\in\cx$ and $k_0\in\zz$,
we choose a sequence of functions,
\begin{equation}\label{eq-def-phi,choose}
\lf\{\wz{\phi}_{k,z}:\ \wz{\phi}_{k,z}\in\mathcal{F}_{k}(z)\r\}_{k\ge k_0},
\end{equation}
such that, for any $Q\in\cq_{k_0}$ that contains $z$,
\begin{align*}
\sup_{\{\{\phi_{k,z}\}_{k\ge k_0}:\ \phi_{k,z}\in\mathcal{F}_{k}(z),\,k\ge k_0\}}
\lf\|\sum_{k=k_0}^{\fz}A_Q\lf\langle\vec{f},\phi_{k,z}\r\rangle\r\|
\le 2\lf\|\sum_{k=k_0}^{\fz}A_Q\lf\langle\vec{f},\wz{\phi}_{k,z}\r\rangle\r\|,
\end{align*}
and define the sequence $\{\vec{g}_l\}_{l\in\zz}$ of vector-valued functions by setting,
for any $l\in\zz$,
\begin{equation*}
\vec{g}_l:=\f{1}{ L_{s,b}(\delta^{l})}\sum_{k=l}^{\fz}\lf\langle\vec{f},\wz{\phi}_{k,\cdot}\r\rangle,
\end{equation*}
where  $\{\wz{\phi}_{k,\cdot}\}_{k\ge l}$ is determined the same as in \eqref{eq-def-phi,choose}.
Collecting all these arguments
and using Lemma \ref{lem-AQ=AQ'},
we conclude that, for any $l\in\zz$ and $x,y\in\cx$ with
$d(x,y)\in[\delta^{l+1},\delta^{l})$
\begin{align*}
&\sum_{Q\in\cq_l}\lf\|A_Q\vec{f}(x)-A_Q\vec{f}(y)\r\|\mathbf{1}_Q(x)\\
&\quad\ls L_{s,b}(\delta^{l})\sum_{Q\in\cq_l}\lf\|A_Q\vec{g}_{l}(x)\r\|\mathbf{1}_Q(x)+
\sum_{P\in\cq_l}\lf\|A_P\vec{g}_{l}(y)\r\|\mathbf{1}_P(y)\\
&\quad\ls L_{s,b}\lf(d(x,y)\r)\sum_{Q\in\cq_l}\lf[\lf\|A_Q\vec{g}_{l}(x)\r\|+
\lf\|A_Q\vec{g}_{l}(y)\r\|\r]\mathbf{1}_Q(x),
\end{align*}
which further implies that
$\{\vec{g}_l\}_{l\in\zz}$ is a positive constant multiple of an element of $\dd^{s,b}_{\{A_Q\}}(\vec{f})$.

Finally, when $p\in(1,\fz)$, we apply the Minkowski inequality
and the H\"{o}lder inequality for $q\in(1,\fz]$ or \eqref{triangle} for $q\in(0,1]$ and,
when $p\in(0,1]$, we apply \eqref{triangle} and
the H\"{o}lder inequality for $q/p\in(1,\fz]$ or \eqref{triangle} for $q/p\in(0,1]$ to obtain
\begin{align*}
\|\vec{f}\,\|_{\dot{N}^{s,b}_{p,q}(W)}
&\ls\lf\|\{\vec{g}_l\}_{l\in\zz}\r\|_{\ell^q(W,L^p)}
\ls \lf\|\lf\{\f{1}{L_{s,b}(\delta^{l})}\sum_{k=l}^{\fz}\sup_{\phi\in\mathcal{F}_{k}(\cdot)}
\lf\|W^{\frac{1}{p}}\lf\langle\vec{f},\phi\r\rangle\r\|\r\}_{l\in\zz}\r\|_{\ell^q(\cx,L^p)}\\
&\ls\lf\{\sum_{k\in\zz}\sum_{l=-\fz}^k\lf[\frac{L_{s,b}(\delta^{k})}{L_{s,b}(\delta^{l})}\r]^{\min\{1,p,q\}}
\lf[L_{s,b}(\delta^{k})\r]^{-q}\lf\|\sup_{\phi\in\mathcal{F}_{k}(\cdot)}
\lf\|W^{\frac{1}{p}}\lf\langle\vec{f},\phi\r\rangle\r\|\,\r\|_{L^p(\cx)}^q\r\}^{\frac 1 q}\\
&\ls\lf\{\sum_{k\in\zz}
\lf[L_{s,b}(\delta^{k})\r]^{-q}\lf\|\sup_{\phi\in\mathcal{F}_{k}(\cdot)}
\lf\|W^{\frac{1}{p}}\lf\langle\vec{f},\phi\r\rangle\r\|\,\r\|_{L^p(\cx)}^q\r\}^{\frac 1 q}
\sim\lf\|\vec{f}\,\r\|_{\mathcal{A}\dot{B}^{s,b}_{p,q}(W)}.
\end{align*}
This finishes the proof of $\|\vec{f}\,\|_{\dot{N}^{s,b}_{p,q}(W)}\ls\|\vec{f}\,\|_{\mathcal{A}\dot{B}^{s,b}_{p,q}(W)}$
and hence Theorem \ref{thm-N=AB}.
\end{proof}

As a straight corollary of both Theorems \ref{thm-B=AB} and \ref{thm-N=AB},
the following Haj{\l}asz pointwise characterization of matrix-weighted logarithmic Besov spaces
on $\cx$ is established.

\begin{theorem}\label{haj-c}
Assume that the doubling measure $\mu$ of $\cx$ has a weak lower bound $\lambda=D$.
Let $\theta$ be the same as in Definition \ref{def-exp-ATI}, $\beta,\gamma\in(0,\theta)$,
$s\in(0,\beta\wedge\gamma)$, $b:=(b_+,b_-)$ be a pair of real numbers,
$p\in(D/[D+s],\fz)$, $q\in(0,\fz]$,
and $W\in A_p(\cx,\cc^m)$. Then
$\dot{N}^{s,b}_{p,q}(W)=\dot{B}^{s,b}_{p,q}(W)$.
\end{theorem}

We point that, in the unweighted scalar-valued case, Theorem \ref{haj-c} when $b=0$ just reduces back to \cite[Theorem 2.16(ii)]{AWYY23}.

\bigskip

\noindent Ziwei Li

\medskip

\noindent Institute of Applied Physics and Computational Mathematics,
Beijing 100088, The People's Republic of China

\smallskip

\noindent{\it E-mail:} \texttt{zwli@buct.edu.cn}

\bigskip

\noindent Dachun Yang and Wen Yuan (Corresponding author)

\medskip

\noindent Laboratory of Mathematics and Complex Systems (Ministry of Education of China),
School of Mathematical Sciences, Beijing Normal University, Beijing 100875,
The People's Republic of China

\smallskip

\noindent{\it E-mails:} \texttt{dcyang@bnu.edu.cn} (D. Yang)

\noindent\phantom{{\it E-mails:} }\texttt{wenyuan@bnu.edu.cn} (W. Yuan)

\end{document}